\documentclass{tac}

\usepackage{amsfonts}
\usepackage{bbm}
\usepackage{amsmath}
\usepackage[all,cmtip,2cell]{xy}
\usepackage{tikz-cd}
\usepackage{tensor}
\usepackage{adjustbox}
\usepackage{amssymb}
\usepackage{xfrac}
\usepackage{faktor}
\UseAllTwocells
\usepackage{pdflscape} 

\usepackage{xy}

\input diagxy

\usepackage[colorlinks=true]{hyperref}

\usepackage[
color=orange!80, 
bordercolor=black,
textwidth=3cm,
textsize=small,
colorinlistoftodos]
{todonotes}

\hypersetup{allcolors=[rgb]{0.1,0.1,0.4}}

\author{Phillip M Bressie}

\thanks{}

\address{Department of Mathematics, Spring Hill College\\
 4000 Dauphin St, Mobile, AL, 36608\\[5pt]
 }

\title {The $\omega$-categorification of Algebraic Theories}

\copyrightyear{2023}

\keywords{Operads, PROs, Higher Categories, Globular Sets, Categorification}
\amsclass{18C10,18C15,18C20,18D05,18D10,18D15,18D20,18D50}

\eaddress{pbressie@shc.edu}

\newcommand{\catname}[1]{{\normalfont\textbf{#1}}}
\newcommand{\Set}{\catname{Set}}

\newcommand{\Cat}{\catname{Cat}}
\newcommand{\Graph}{\catname{Graph}}
\newcommand{\Glob}{\catname{Glob}}

\newcommand{\GlobPRO}{\catname{GlobPRO}}
\newcommand{\Duoidal}{\catname{Duoidal}_{\textit{lax}}}
\newcommand{\CartDuoidal}{\catname{CartDuoidal}}

\newcommand{\Coll}{\catname{Col}} 
\newcommand{\BiGrd}{\catname{BiGrd}}
\newcommand{\Talg}{\catname{T-Alg}}
\newcommand{\Cont}{\catname{Cont}}

\newcommand{\Mon}{\catname{Mon}}

\newcommand{\one}{\textbf{1}}
\newcommand{\Tone}{\mathcal{T}(\one)}
\newcommand{\Tsquared}{\mathcal{T}^2(\one)}
\newcommand{\id}{\mathbbm{1}}

\DeclareSymbolFont{AMSa}{U}{msa}{m}{n}
\DeclareMathSymbol{\mysquare}{\mathord}{AMSa}{"03}

\newcommand{\pullbackmark}[2]{\save ;p+<.8pc,0pc>:(0,-1)::%
	(#1) *{\phantom{Z}} %
	;p+(#2)-(0,0) **@{-}%
	;p-(#1)+(0,0) *{\phantom{Z}} **@{-} \restore}

\newcommand{\pullbackSub}[2]{{}_{#1}\kern-\scriptspace{\times}_{#2}}

\newtheorem{theorem}{Theorem}

\newtheoremrm{rem}{Remark}

\mathrmdef{Hom}
\mathbfdef{Set}

\begin{document}

\maketitle
\begin{abstract}
Batanin and Leinster’s work on globular operads has provided one of many potential definitions of a weak $\omega$-category. Through the language of globular operads they construct a monad whose algebras encode weak $\omega$-categories. The purpose of this work is to show how to construct a similar monad which will allow us to formulate weak $\omega$-categorifications of any equational algebraic theory. We first review the classical theory of operads and PROs. We then present how Leinster’s globular operads can be extended to a theory of globular PROs via categorical enrichment over the category of collections.  It is then shown how a process called globularization allows us to construct from a classical PRO $P$ a globular PRO whose algebras are those algebras for $P$ which are internal to the category of strict $\omega$-categories and strict $\omega$-functors. Leinster’s notion of a contraction structure on a globular operad is then extended to this setting of globular PROs in order to build a monad whose algebras are weakenings of the globularization of the classical PRO $P$.  Among these weakenings is the initial weakening whose algebras are by construction the fully weakened $\omega$-categorifications of the algebraic theory encoded by $P$.
\end{abstract}


\section{Introduction}\label{sec-Introduction}
The purpose of this work is to provide a framework and procedure for the categorification of general equational algebraic theories such as monoids, groups, quandles, rings, etc.  The notion of a categorification of an algebraic structure was first introduced by Crane\cite{crane_1995Clock}.  His original formulation consisted of constructing from an algebraic structure, such as a Hopf algebra, a new algebraic structure with analogous operations one categorical dimension higher, such as a Hopf category\cite{crane_Frenkel_1994TopQuantField}.  Such a construction, from a given model of an algebraic theory, requires choices and is in most nontrivial cases not functorial.  However, when this process is performed at the level of theories, it can be made functorial in a nontrivial and universal way.

There are many various ways one could categorify an abstract theory.  One approach is to provide a definition of your algebraic theory $T$ as an obect in the category $\Set$ equipped with various functions whose composition satisfies certain relations.  These relations then endow any such set with the structure of a $T$-algebra.  We can then internalize this construction in the category $\Cat$, so that the underlying set of the $T$-algebra is now replaced with an underlying category.  Moreover, the operations in the theory are now replaced with functors.  This is precisely what is meant by raising the categorical dimension by one step (in this case from 0 to 1).

The drawback to this procedure is that such a construction gives us a notion of categorification that is strict, in the sense that all relations satisfied by the composition of functors strictly hold by equality.  In many context this is not quite the desired construction.  We often want a notion of categorification that satisfies its defining relations weakly, in the sense that the equations imposed on the operations hold up to an equivalence defined by morphisms one categorical dimension higher.  But this leaves us with a new problem.  Once we construct morphisms one categorical dimension higher (in this case natural transformations between compositions of functors), we must determine what coherence conditions these new morphisms must satisfy when composed in various ways.  And even if appropriate coherence conditions can be determined, we are now left with a construction that gives a strict 2-categorification of the theory $T$.  We can similarly make our strict 2-categorification weak by replacing the coherence condition equations with 3-morphisms.  But once again we must determine what higher coherence conditions these 3-morphisms must satisfy.  This in turn leaves us with a strict 3-categorifiction, leading to an infinite sequence of analogous problems.

In principle we can continue to climb this ladder of categorical dimension hopping forever.  Unfortunately, each step becomes conceptually and computationally more difficult than the last.  It is hence more desirable to have a limiting construction that categorifies $T$ in each finite categorical dimension all at once.  This is one approach taken in higher category theory when attempting to defining notions of weak $\omega$-categories.  Batanin\cite{batanin_1998GlobCat} and Leinster\cite{leinster2004higher} have provided globular notions of weak $\omega$-categories.  After building up a language of globular operads, they then present globular $\omega$-categories as algebras for the initial globular operad constructed via a particular monad.

The process presented in the present work extends Batanin and Leinster's constructions to a formulation of how to build weak $\omega$-categorification of any equational algebraic theory $T$.  This is achieved by extending the theory of globular operads to a theory of globular PROs (a nonsymetric version of MacLane's notion of a PROP\cite{catalgebra:maclane1965}).  We shall first develop this theory of globular PROs.  Once this notion has been made precise, a process for turning the classical PRO $P_T$ for the theory $T$ into a globular PRO $\mathcal{P}_T$ will be presented.  This globular PRO $\mathcal{P}_T$ will have the property that its algebras in $\Glob$, the category of globular sets, are algebras for $P$ internal to the category of strict $\omega$-categories and strict $\omega$-functors.  Leinster's notion of contractibility will then be used to construct an initial weakening of $\mathcal{P}_T$.  This initial weakening of $\mathcal{P}_T$ will then, by construction, have as algebras, weak $\omega$-categorifications of the theory $T$.

	\section{PROs} 
		We begin by recalling the classical notion of a PRO.  Note that any algebraic structure whose axioms require two instances of the same variable on the same side of an equation cannot be represented by an operad, due to the fact that operations in an operad can only have coarity one.  For example, a group $G$ requires that for all $g \in G$
		$$g \cdot g^{-1} = e$$
		with $e$ being the identity element.  But capturing such a relation using abstract operations requires the use of a diagonal map $\Delta: G \times G \rightarrow G$, inversion map $(-)^{-1}: G \rightarrow G$, multiplication map $m: G \times G \rightarrow G$, identity identification map $e: \{*\} \rightarrow G$, and the unique set map $!_G: G \rightarrow \{*\}$ to the terminal one element set, subject to the constraint that the diagram
		$$\xymatrix{ & G \times G \ar[rr]^{\id_G \times (-)^{-1}} & & G \times G \ar[dr]^{m} & \\
			G \ar[ur]^{\Delta} \ar[rr]_{!_G} & & \{*\} \ar[rr]_{e} & & G }$$
		commutes.  But the diagonal map is not something which can exists among the operations of an operad.  In what follows we will instead be working primarily with PROs (whose name is short for product category), a generalization of nonsymetric operads which allows for operations with coarity greater than one, like the diagonal.
		
		\begin{definition}
			A \emph{PRO} $P$ is a strict monoidal category whose object set is isomorphic to $\mathbb{N}$ such that the monoidal product $+:P \times P \rightarrow P$ is identified with addition of natural numbers at the level of objects.
		\end{definition}
		
		We avoid axiomatizing a symmetric group action, preferring PROs to MacLane's notion of PROPs \cite{catalgebra:maclane1965} (whose name is short for product and permutation category).The theory of PROs can be easily extended to that of MacLane's original notion $\cite{catalgebra:maclane1965}$ of a PROP via the following construction.  The \textit{permutation PRO} $\mathcal{S}$ has $\mathbb{N}$ as its set of objects.  Its collection of morphisms are generated by the symmetric groups $\Sigma_n$ in the following sense.  For each $n \in \mathbb{N}$, the PRO component $\mathcal{S}(n,n) = \Sigma_n$, with the total collection of morphisms in the permutation PRO being the free monoid on the disjoint union of each of these components.  The operation + in the permutation PRO is concatenation of elements in the free monoid, corresponding to the obvious group homomorphism $\Sigma_n \times \Sigma_m \rightarrow \Sigma_{n + m}$, with the (co)arity of a word of permutations being simply the sum of the (co)arities of its letters.  Furthermore, since each morphism is either a permutation or a tensor product of such permutations (which is manifestly another permutation itself), composition in the permutation PRO is the usual composition of permutations.  The monoid unit morphism is simply the map which sends each $n$ to the identity permutation on $n$ elements.  Moreover, we require that $\mathcal{S}$ satisfy the following naturality condition.  Let $\tau_{n,m} \in \Sigma_{n + m}$ be the permuatation which swaps the first block of $n$ elements with the second block of $m$ elements.  Then, for all $n,m,n',m' \in \mathbb{N}$ we have that
		$$\tau_{m,m'} \circ (f + f') = (f' + f) \circ \tau_{n,n'}$$
		for all $f:n \rightarrow m$ and $f:n' \rightarrow m'$ in $\mathcal{S}$.  We can now give the following succinct definition for a PROP.
	
		\begin{definition}
			A \textit{PROP} is a PRO which contains the permutation PRO as a sub-PRO and satisfies the naturality condition above.
		\end{definition}
		
		This definition indicates a notable distinction between operads and PROs.  Should one chose to work with symmetic operads, the action of the symmetric group must be encoded as an extra piece of structure that acts upon the underlying non-symmetric operad.  However, being a PROP is simply a property that a PRO may or may not satisfy.  The action of the symetric group can be encoded entirely within the PRO structure of the PROP.
		
		More genearlly, note that PROs are precisely the monoid objects in $\Cat$, with respect to the cartesian product, which have the extra property that their monoid of objects is isomorphic to $(\mathbb{N}, +, 0)$.  Given a PRO $P$, we may think of the morphisms in $P(n,m)$ as operations of arity $n$ and coarity $m$.  Until an algebra for a PRO is specified, the objects $\mathbb{N}$ behave as placeholders for the arities and coarities of these operations, so that they may be composed with one another.  Once an algebra is specified for $P$, these `slots' will be filled with elements from the underlying set of the algebra, justifying the use of the name operations for the morphisms of our PRO.

		\begin{definition}
			An \emph{algebra for a PRO} $P$ in $\Set$ is a set $A$, together with, for all $n,m \in \mathbb{N}$, a family of functions
			$$\omega_{n,m}:P(n,m) \times A^n \rightarrow A^m$$
			which make the following diagrams commute
			$$\adjustbox{max width=\columnwidth}{\xymatrix{[P(m,r) \times P(n,m)] \times A^n\ar[rrr]^{\alpha^{\Set}_{P(m,r), P(n,m), A^n}} \ar[dd]_{\circ_{n,m,r}} & & & P(m,r) \times [P(n,m) \times A^n] \ar[rrr]^-{\id_{P(m,r)} \times \omega_{n,m}} & & & P(m,r) \times A^m \ar[dd]^{\omega_{m,r}} \\
					\\
					P(n,r) \times A^n \ar[rrrrrr]_{\omega_{n,r}} & & & & & & A^r}}$$
			
			$$\xymatrix{[P(n,m) \times P(r,s)] \times [A^n \times A^r] \ar[rrr]^-{+_{n,m,r,s} \times \id_{n + r}} \ar[dd]_{\boxtimes_{P(n,m),P(r,s),A^n,A^r}} & & & P(n + r, m + s) \times A^{n + r} \ar[dd]^{\omega_{n + r}} \\
				\\
				[P(n,m) \times A^n] \times [P(r,s) \times A^r] \ar[rrr]_-{\omega_{n,m} \times \omega_{r,s}} & & & A^m \times A^s \cong A^{m + s} }$$
			
			$$\xymatrix{ & & P(n,n) \times A^n \ar[ddrr]^{\omega_{n,n}} & & \\
				\\
				\{*\} \times A^n \ar[uurr]^{j_n \times \id_{A^n}} \ar[rrrr]_{\lambda^{\Set}_{A^n}} & & & & A^n }$$
			for all $n,m,r,s \in \mathbb{N}$, where the set map $\boxtimes_{X,Y,Z,W}: [X \times Y] \times [Z \times W] \rightarrow [X \times Z] \times [Y \times W]$ is the interchange morphism in $\Set$ which swaps the second factor with the third and reassociates accordingly, while $j_n:\{*\} \rightarrow P(n,n)$ identifies which element of $P(n,n)$ is the identity operation.
		\end{definition}
		
		Alternatively, we can construct what is called the endomorphism PRO on a set $A$.  This allows us to give $A$ the structure of an algebra via a representation homomorphism.
		\begin{definition}
			Given a set $A$, the $\emph{endomorphism PRO on A}$, denoted by $End(A)$, is the PRO which has as its set of objects all successive cartesian powers $A^n = \displaystyle\prod_{i = 1}^n A$ of the underlying set $A$ for all $n \in \mathbb{N}$, which can be naturally identified with $\mathbb{N}$.  Under this identification, the hom-sets $End(A)(m,n)$ in $End(A)$ are the hom-sets $\Set(A^n,A^m)$.  Composition in the PRO is simply the composition induced from $\Set$.  The monoidal product is induced by the product structure on $\Set$, which may be identified with addition of natural numbers.
		\end{definition}
		
		\begin{definition}
			An $\emph{algebra for a}$ PRO $P$ is a strict monoidal functor $F:P \rightarrow End(A)$ for some set $A$.
		\end{definition}
		
		We thus have two equivalent notions of an algebra whose equivalence can be seen via currying the action maps $\omega_{n,m}$ to obtain components of the desired monoidal functor.  In either description, we say that the PRO $P$ $\emph{acts on the object A}$.

		Before moving on, we should note that those readers already familiar with classical PROs may wonder why we have not chosen to work with Lawvere theories, which are the special case of a PRO whose monoidal product is the cartesian product.  Much of the work that follows could faithfully be reproduced using Lawvere theories instead of PROs.  The choice to work with the more general notion of a PRO is motivated by the process of PRO globularization which will be described later.  This process can be done with any PRO, not simplly the special case of Lawvere theories.

	\section{Collections}\label{sec-collections}
		We begin this section by recalling the notion of a globular set.  To do so requires the following category $\mathbb{G}$, known as the \emph{globe category}.  The category $\mathbb{G}$ has $\mathbb{N}$ as its set of objects.  Its morphisms are generated by $\sigma_n: n \rightarrow n+1$ and $\tau_n: n \rightarrow n+1$ for all $n \in \mathbb{N}$ subject to the relations $\sigma_{n+1} \circ \sigma_n = \tau_{n+1} \circ \sigma_n$ and $\sigma_{n+1} \circ \tau_n = \tau_{n+1} \circ \tau_n$.

		\begin{definition}
			A \emph{globular set} is a contravariant functor $G: \mathbb{G} \rightarrow \Set$. The category $\Glob$ of globular sets is the category of presheaves on $\mathbb{G}$.
		\end{definition}

		Another integral piece of structure we will need is the free strict $\omega$-category monad $\mathcal{T}: \Glob \rightarrow \Glob$.  This monad has the following property.
		
		\begin{definition}
            A monad $(T:\mathcal{C} \rightarrow \mathcal{C},\mu: T^2 \Rightarrow T, \eta: \id_{\mathcal{C}} \Rightarrow T)$ is a $\emph{cartesian monad}$ if all naturality squares for $\mu$ and $\eta$ are pullback squares and $T$ preserves all pullbacks.
        \end{definition}
		
		That $\mathcal{T}$ is cartesian, as well as a more detailed description of its construction and use, can be found in \cite{leinster2004higher}.  Briefly, this monad takes a globular set $\mathcal{X}$ and returns the underlying globular set of the free strict $\omega$-category generated by $\mathcal{X}$.  In other words, it takes a globular set $\mathcal{X}$ and constructs the globular set $\mathcal{T}(\mathcal{X})$ consisting of all possible pasting diagrams built out of the cells of $\mathcal{X}$.  We will often describe these pasting diagrams as `globular words' built out of the cells of $\mathcal{X}$.  The motivation for calling such a pasting diagram a word is to emphasize, for the sake of intuition, the analogy to the monoid of words constructed from a set of letters.  Note that a pasting diagram, all of whose cells are cells in $\mathcal{X}$, can be thought of as a generalization of the notion of a word thought of as a string of concatenated elements from some set $Y$.  The main difference between the two notions is that a globular word can be built out of concatenation of cells along any of their boundary cells, as opposed to the classical setting in which elements, or letters, can only be composed as horizontal strings.  So in this way we can think of words on a set (in either setting) as simply an element, or cell, in the underlying object of the free monoid, or $\omega$-category, on the respective notion of set.  Using this convention can sometimes be helpful when gaining intuition for the notion of globular pasting, which will be used frequently in what follows.  Throughout this work, when we do use the terms `pasting diagram' or `pasting scheme', these are in the sense of Power\cite{PowerNPasting1991}.  However, all cells of the pasting schemes and diagrams are globes.
		
		For our purposes we will be specifically interested in the globular set $\Tone$ generated by the terminal globular set $\one$ which has exactly one cell in each dimension.  It is precisely $\Tone$ which allows us to generalize our notion of the arity of an operation.  In the classical case, the arity of an element is the length of the word in $T(\{*\})$ over which the operation sits with respect to the operad's structure as a set fibered over $\mathbb{N}$.  In this more general context, the arity of a cell in a globular set $\mathcal{X}$ is a pasting scheme specified by a globular cell in $\Tone$.  More precisely, we can equip a globular set $\mathcal{X}$ with a morphism $x:\mathcal{X} \rightarrow \Tone$ which specifies globular arities via cells in $\Tone$ that are named by pasting schemes.  We may in turn think of the pasting schemes which name the cells in $\Tone$ as the possible arity shapes with which the globular cells in $\mathcal{X}$ may be equipped.

		\begin{definition}
			A \emph{collection} is a globular set $\mathcal{X}$ equipped with a globular set homomorphism $x:\mathcal{X} \rightarrow \Tone$ called the \emph{arity map}.
		\end{definition}

		By abuse of notation we will often represent a collection by simply writing its underlying globular set $\mathcal{X}$.

		\begin{definition}
			Let $x: \mathcal{X} \rightarrow \Tone$ and $y: \mathcal{Y} \rightarrow \Tone$ be a pair of collections.  A \emph{collection homomorphism} between them is a globular set map $f:\mathcal{X} \rightarrow \mathcal{Y}$ which makes the triangle
			$$\xymatrix{\mathcal{X} \ar[rr]^{f} \ar[dr]_{x} & & \mathcal{Y} \ar[dl]^{y}\\
			& \Tone & }$$
			commute.
		\end{definition}

		We shall use $\Coll$ to denote the category of collections. Note that $\Coll$ is simply the slice category $\Glob / \Tone$.  Furthermore, $\Coll$ has a monoidal structure with respect to a composition tensor product $\mysquare : \Coll \times \Coll \rightarrow \Coll$ defined as follows:
		
		\begin{definition}
			Let $x: \mathcal{X} \rightarrow \Tone$ and $y: \mathcal{Y} \rightarrow \Tone$ be a pair of collections.  Their composition tensor product $x \mysquare y : \mathcal{X} \mysquare \mathcal{Y} \rightarrow \Tone$ is defined by the diagram:
			$$\xymatrix{\mathcal{X} \mysquare \mathcal{Y} \pullbackmark{0,2}{2,0} \ar[rr] \ar[dd] & & \mathcal{T}(\mathcal{Y}) \ar[r]^-{\mathcal{T}(y)} \ar[dd]^{\mathcal{T}(!_{\mathcal{Y}})} & \Tsquared \ar[r]^-{\mu_{\one}} & \Tone \\
				& &  \\
				\mathcal{X} \ar[rr]^{x} & & \Tone}$$
			where $!_{\mathcal{Y}}: \mathcal{Y} \rightarrow \one$ is the unique map from $\mathcal{Y}$ to the terminal globular set.  The underlying globular set $\mathcal{X} \mysquare \mathcal{Y}$ is the pullback of $x$ and $\mathcal{T}(!_{\mathcal{Y}})$ with the arity globular set map $x \mysquare y$ defined to be the composition along the top row.
		\end{definition}
		
		This definition makes $\mathcal{X} \mysquare \mathcal{Y}$ the unique collection whose cells are pairs $(a,\psi)$ consisting of a $k$-cell $a \in \mathcal{X}$ and a `globular word' $\psi$ of $k$-cells from $\mathcal{Y}$ indexed by the arity of $a$.  In $\mathcal{X} \mysquare \mathcal{Y}$, the `globular letters' in the globular word $\psi \in \mathcal{T}(\mathcal{Y})$ may be compatibly `glued together' via the shape of $x(a) \in \Tone$ in the sense that each globular letter of $\psi$ is a $k$-cell whose arity shape under $y$ can replace a particular $k$-cell in the pasting scheme which names the cell $x(a) \in \Tone$. We can thus think of the cells of $\mathcal{X} \mysquare \mathcal{Y}$ as composable pairs specified by a cell of $\mathcal{X}$ and a `word' of cells from $\mathcal{Y}$, each of whose `letters' may be plugged into a sub k-cell of the $k$-dimensional pasting scheme which names the arity cell $x(a)$.
		
		Furthermore, the arity for a composable pair in $\mathcal{X} \mysquare \mathcal{Y}$ may be thought of as the `sum' of the arities of each letter in $\psi$ `glued together' in the shape of the pasting scheme which names $x(a)$.  More precisely, note that the map $\mathcal{T}(y)$ takes a word of cells from $\mathcal{Y}$ and returns a word of arity cells (i.e. a cell in $\mathcal{T}^2(\one)$ which is named by a pasting diagram of pasting schemes).  The component at $\one$ of the unit transformation $\mu$ for $\mathcal{T}$ then takes this globular word of arities and returns the cell in $\Tone$ which is named by the pasting scheme we would get if we strictly pasted together this diagram of schemes.  We can think of the cells in $\mathcal{T}^2(\one)$ as being named by factorizations of pasting schemes.  From this perspective, $\mu_{\one}$ essentially reduces this factorization by specifying the cell in $\Tone$ which is named by the strict pasting composition specified by the factorization.

	\section{Special Collections}
		There are several particular collections that are worth noting.  The first is the terminal collection $\id: \Tone \rightarrow \Tone$.  This collection is the unit for the cartesian product in $\Coll$.  Recall that the cartesian product in a slice category is defined via the pullback of a sliced object along another.  Hence the cartesian product of a collection $a: A \rightarrow \Tone$ against $\id$ creates an isomorphic collection
		whose underlying globular set consists of the elements of $A$ paired together with their arity shape specified by $a$.
		
		Also of note is the collection $I: \one \hookrightarrow \Tone$ whose arity map is simply the inclusion of generators.  This collection is important because it is the unit for $\mysquare$ in $\Coll$.  Note that applying $\mysquare$ to $a: A \rightarrow \Tone$ with $I$ on the right gives a collection whose underlying globular set consists of pairs whose entries are an element of $A$ together with the globular pasting scheme from $\Tone$ which represents its arity.  Similarly, tensoring with $I$ on the left gives a collection whose underlying globular set consists of pairs whose entries are a generic $n$-cell from $\one$ together with an $n$-cell from $A$.  As each of these collections is isomorphic to $a: A \rightarrow \Tone$, the collection $I: \one \hookrightarrow \Tone$ must be the unit for $\mysquare$ in $\Coll$.  Moreover, when enriching over $\Coll$ with respect to $\mysquare$, this collection can be used to distinguish elements in a particular hom-object.  For example, if $a: A \rightarrow \Tone$ is a hom-object collection, then a cell $x$ of $A$ may be distinguished by a collection morphism $[x] : \one \rightarrow A$.
		
		Note that the previous two collections are units for $\times$ and $\mysquare$ respectively.  These units give $\Coll$ two different monoidal structures.  Note that the $\mysquare$ unit $I$ is a sub-object of the cartesian unit $\id$.  Note also that monoids with respect to the product $\mysquare$ are precisely the objects of key interest in Leinster's construction.
		
		\begin{definition}
			A \emph{globular operad} is a monoid in $\Coll$ with respect to the monoidal product $\mysquare$.
		\end{definition}

		Another collection of note, which we shall use later, is the initial collection $\{\}: \emptyset \rightarrow \Tone$  whose arity map is the vacuous mapping from the empty globular set.  We may think of this as the empty collection.  For any collection $a: A \rightarrow \Tone$, it follows that both the cartesian and $\mysquare$ product (on either side) with $\{\}$ is simply $\{\}$.  With respect to these two products, $\{\}$ essentially behaves like multiplication by 0.
		
		One final collection worth mentioning is given by the globular set map $[id]: \one \rightarrow \Tone$.  Note that among the many cells in $\Tone$ there are the underlying globular cells of identity morphisms created when $\mathcal{T}$ produces the underlying globular set of the free strict $\omega$-category on $\one$.  Among these identities are the following special identities.  There is the underlying 1-cell of the identity on the single vertex in $\one$.  This identity map then has an identity 2-cell that sits over it.  And over this identity 2-cell there is an identity 3-cell that sits above it.  Continuing this process, we see that there is an inclusion of the terminal globular set $\one$ into $\Tone$ whose cells are exactly the iterated identities on the single 0-cell.  This sub-object can be thought of as the globular $\omega$-analogue of the additive identity $0 \in \mathbb{N}$.  The map $[id]$ is then the globular set map which sends each of the single $n$-cells in $\one$ to the corresponding iterated identity cell of dimension $n$ from the construction just described.  The map $[id]$ is, in this way, an identification of this `tower' of iterated identities as the particular identities for each $n$-dimensional pasting composition.

	\section{The Internal Hom in \Coll}\label{sec-exp-coll}
        Just as algebras for classical operads and PROs can be given by representation maps into an endomorphism object, so too can algebras for globular operads and PROs.  To better understand this process, we will briefly sketch how such endomorphsim objects are constructed in $\Coll$.  The key is understanding the construction of the internal hom in $\Coll$.  Fortunately, there is a convenient process for concretely understanding how this internal hom is constructed.
 
		Let $\phi: \mathcal{A} \rightarrow \mathcal{B}$ be a globular set map.  This then induces a change of base functor $\phi^*:\Glob/\mathcal{B} \rightarrow \Glob/\mathcal{A}$ between slice categories.  It takes a globular set map and returns its pullback along $\phi$.  The functor $\phi^*$ has both a left and right adjoint.  Its left adjoint $\Sigma_{\phi}: \Glob/\mathcal{A} \rightarrow \Glob/\mathcal{B}$ is simply composition with $\phi$.  Its right adjoint $\Pi_{\phi}: \Glob/\mathcal{A} \rightarrow \Glob/\mathcal{B}$ is more complicated to describe in general.   More detail on the general construction of $\Pi_{\phi}$ can be found in \cite{maclane1994sheaves}.  Fortunately, in a category whose objects have elements, such as $\Glob$, this right adjoint has a relatively nice description.  We can intuitively think of the fibered globular sets in $\Pi_{\phi}(\xi:\mathcal{X} \rightarrow \mathcal{A})$ as the globular set fibered over $\mathcal{B}$ of generalized sections of the globular set map $\xi$.  Let $\psi : \mathcal{Y} \rightarrow \mathcal{A}$ be any morphism in $\Glob/\mathcal{A}$.  The map $\Pi_{\phi}(\psi): \Gamma \rightarrow \mathcal{B}$ is constructed by specifying the fiber over each point.  Take an element $b \in \mathcal{B}$ and consider its fiber $A_b$ along the map $\phi$.  Each element $c \in \mathcal{A}_b$ then has a fiber $\mathcal{Y}_c$ sitting above it along the map $\psi$.  We can then define the fiber $\Gamma_b$ along the map $\Pi_{\phi}(\psi)$ to be the product $\displaystyle\prod_{c \in \mathcal{A}_b} Y_c$.  Following this construction for each $b \in \mathcal{B}$ gives the complete map from $\Gamma := \displaystyle\coprod_{b \in \mathcal{B}} \displaystyle\prod_{c \in \mathcal{A}_b} \mathcal{Y}_c$ to $\mathcal{B}$.

		Note then that, given the functor $- \mysquare \mathcal{B}: \Coll \rightarrow \Coll$ for the collection $b:\mathcal{B} \rightarrow \Tone$, we can decompose $- \mysquare \mathcal{B}$ as the following composition:
		$$- \mysquare \mathcal{B} = \Sigma_{\mu_{\one}}\Sigma_{\mathcal{T}(b)}\mathcal{T}(!_{\mathcal{B}})^*$$
		This functor takes the collection $a:\mathcal{A} \rightarrow \Tone$ to the collection $a \mysquare b: \mathcal{A} \mysquare \mathcal{B} \rightarrow \Tone$, where the arity map $a \mysquare b$ is $\Sigma_{\mu_{\one}}\Sigma_{\mathcal{T}(b)}\mathcal{T}(!_{\mathcal{B}})^*(a)$. Notice that this is exactly composition in the augmented pullback diagram used to define the composition tensor product $\mysquare$ in $\Coll$.  But by decomposing the functor $- \mysquare \mathcal{B}$ in this way, we can compute the appropriate right adjoint $[\mathcal{B},-]:\Coll \rightarrow \Coll$ by taking the right adjoint of each factor in the composition and reversing the order in which they are composed.  This then leads to the following formula:
		$$[\mathcal{B},-] = \Pi_{\mathcal{T}(!_{\mathcal{B}})}\mathcal{T}(b)^*\mu_{\one}^*$$
		Let us briefly consider how the composite $\mathcal{T}(b)^*\mu_{\one}^*$ acts on a collection $a:\mathcal{A} \rightarrow \Tone$.  Note that the map $\mathcal{T}(b)^*\mu_{\one}^*(a)$ is given as the topmost edge in the appropriate double pullback diagram to get the globular set map
		$$\mathcal{T}(b)^*\mu_{\one}^*(a): (\mathcal{A} \pullbackSub{a}{\mu_{\one}^{*}} \mathcal{T}^2(\one)) \pullbackSub{\mu_{\one}^{*}(a)}{\mathcal{T}(b)} \mathcal{T}(\mathcal{B}) \rightarrow \mathcal{T}(\mathcal{B})$$
		which is simply second projection.  We can intuitively think of this map as associating to each cell $\beta \in \mathcal{T}(\mathcal{B})$, which is named by a pasting diagram labeled by cells in $\mathcal{B}$, a pair $(\alpha, t) \in \mathcal{A} \times \mathcal{T}(\Tone)$ consisting of cell $\alpha \in \mathcal{A}$ and cell of cells $t$ of shape $\sigma \in \Tone$ (i.e. a globular word of cells in $\Tone$ indexed by the diagram of shape $\sigma$) such that the shape of $\alpha$ is the same as the shape of the cell obtained by gluing together the cells in $t$ per the pasting formula given by $\sigma$.  Moreover, the unlabeled cells of $t$ each have the same shape as the corresponding cells which make up the labeled diagram $\beta$.  We then apply $\Pi_{\mathcal{T}(!_{\mathcal{B}})}$ to $\mathcal{T}(b)^*\mu_{\one}^*(a)$ to get the desired internal hom.  

		Via this construction, we can now compute the internal hom $\mathcal{H}_{\mathcal{B},\mathcal{A}}: [\mathcal{B},\mathcal{A}] \rightarrow \Tone$ in $\Coll$ between any collections $b:\mathcal{B} \rightarrow \Tone$ and $a:\mathcal{A} \rightarrow \Tone$.  We can intuitively think of cells in each fiber $[A,B]_{\sigma}$ of our internal hom in the following way:  Recall that the internal hom is constructed as the object of general sections of the globular set map $\mathcal{T}(b)^*\mu_{\one}^*(a)$ defined above.  Moreover, a cell $\beta \in \mathcal{T}(\mathcal{B})_{\sigma}$ can be thought of as a choice of cells $\{\beta_{\tau}\}_{\tau \in \sigma}$ from $\mathcal{B}$ glued together along halves of their boundaries as prescribed by the pasting formula given by the pasting diagram $\sigma$.  Or rather, we can think of them as a coloring of the diagram $\sigma$ by cells in $\mathcal{B}$.  This allows us to think of a cell $\gamma \in [\mathcal{B},\mathcal{A}]_{\sigma}$ as a choice of a cell $\alpha \in \mathcal{A}$ to correspond to each coloring of the diagram $\sigma$ by cells of $\mathcal{B}$ so that the shape of $\alpha$ is the same as the shape of the diagram obtained by gluing the cells $\{\beta_{\tau}\}_{\tau \in \sigma}$ together via the pasting formula given by $\sigma$.  In other words, a `map' in the internal hom is roughly a thing that takes a coloring of the diagram $\sigma \in \Tone$ by cells from the source and picks a cell of the target that has the same arity shape as the cells from the source after all the pasting compositions prescribed by the diagram $\sigma$ have been performed.
		
		We can thus conclude this section with the following theorem.
		
		\begin{theorem}
			The category $\Coll$ has a closed monoidal structure with respect to the monoidal product $\mysquare$.
		\end{theorem}

	\section{The Endomorphism Globular Operad}\label{sec-glob-end}
		Consider the collection $x:\mathcal{X} \rightarrow \Tone$.  We shall now construct the endomorphism globular operad on $\mathcal{X}$, denoted $Gend(\mathcal{X})$.  We define $Gend(\mathcal{X}):= [\mathcal{X},\mathcal{X}]$ via the internal hom construction in $\Coll$.  The underlying collection for the endomorphism globular operad on $\mathcal{X}$ can be thought of as abstractly encoding all the possible operations that take a coloring of a pasting diagram of shape $\sigma \in \Tone$ by globular cells from $\mathcal{X}$ to a single globular cell from $\mathcal{X}$ whose shape is the same as the `word of cells' after each of the pasting compositions prescribed by $\sigma$ are evaluated to give a composed cell in $\mathcal{X}$.  But since they are constructed using the internal hom, rather than the set valued hom, these `maps' from $\mathcal{X}$ to $\mathcal{X}$ naturally fiber over $\Tone$ so that we can place a canonical operad structure on $Gend(\mathcal{X})$.  The operad identity is given by the map $\iota: \one \rightarrow [\mathcal{X},\mathcal{X}]$ which maps each single $k$-cell of $\one$ to the the respective $k$-cell of $[\mathcal{X},\mathcal{X}]$  which corresponds to the identity operation on $k$-cells of $\mathcal{X}$.  This map $\iota$ can be constructed canonically as the currying of the left unitor $\lambda_{\mathcal{X}}: \one \mysquare \mathcal{X} \rightarrow \mathcal{X}$ for the monoidal structure in $\Coll$.  The composition map $\nu:[\mathcal{X},\mathcal{X}] \mysquare [\mathcal{X},\mathcal{X}] \rightarrow [\mathcal{X},\mathcal{X}]$ is the canonical map which takes a pair $(a,w) \in [\mathcal{X},\mathcal{X}] \mysquare [\mathcal{X},\mathcal{X}]$ and composes each of the letters from the word $w \in \mathcal{T}([\mathcal{X},\mathcal{X}])$ with each of the respective inputs for the operation $a \in [\mathcal{X},\mathcal{X}]$.  It can be canonically constructed as follows.  Consider the counit $\epsilon^{\mathcal{A}}: [\mathcal{A},-] \mysquare \mathcal{A} \Rightarrow \id_{\Coll}$ of the hom-tensor adjunction in $\Coll$ between $- \mysquare \mathcal{A}$ and $[\mathcal{A},-]$, which has components
		$\epsilon^{\mathcal{B}}_{\mathcal{A}}:[\mathcal{A},\mathcal{B}]\mysquare \mathcal{A} \rightarrow \mathcal{B}$ for collection $\mathcal{A}$.  We then get a map
		$$\mathcal{K}: ([\mathcal{X},\mathcal{X}] \mysquare [\mathcal{X},\mathcal{X}]) \mysquare \mathcal{X} \rightarrow [\mathcal{X},\mathcal{X}] \mysquare ([\mathcal{X},\mathcal{X}] \mysquare \mathcal{X}) \rightarrow [\mathcal{X},\mathcal{X}] \mysquare \mathcal{X} \rightarrow \mathcal{X}$$
		which is the composite $\mathcal{K} := \epsilon^{\mathcal{X}}_{\mathcal{X}} \circ (\id_{\mathcal{X}} \mysquare \epsilon^{\mathcal{X}}_{\mathcal{X}}) \circ \alpha_{[\mathcal{X},\mathcal{X}], [\mathcal{X},\mathcal{X}], \mathcal{X}}$.  The operad multiplication for $[\mathcal{X},\mathcal{X}]$ is then the currying of the map $\mathcal{K}$.

		\begin{theorem}
			Given a collection $x: \mathcal{X} \rightarrow \Tone$, the collection Gend($\mathcal{X}):[\mathcal{X},\mathcal{X}] \rightarrow \Tone$ admits the structure of a globular operad.
		\end{theorem}
		
		\begin{proof}
			We need only to show that for $Gend(\mathcal{X}):[\mathcal{X},\mathcal{X}] \rightarrow \Tone$ the collection morphisms $\iota: \one \rightarrow [\mathcal{X},\mathcal{X}]$ and $\nu: [\mathcal{X},\mathcal{X}] \mysquare [\mathcal{X},\mathcal{X}] \rightarrow [\mathcal{X},\mathcal{X}]$ satisfy the commutative diagrams required of a monoid object in $\Coll$.  This can be seen by first currying the maps in the relevant diagrams and checking to see that these new curried diagrams, whose commutativity is equivalent with that of the originals, do in fact commute (a routine check).
		\end{proof}

	\section{Algebras for a Globular Operad}
		The notion of an algebra for a globular operad is analogous to that of algebras for classical operads.  However, with the following definition, we can describe how a globular operad acts on a globular set using the $\mysquare$ structure in $\Coll$.  This merely requires thinking of globular sets as collections in the following way.
		
		\begin{definition}
			A collection $x:\mathcal{X} \rightarrow \Tone$ is said to be \emph{degenerate} if the arity map factors as $x = [id] \circ !_{\mathcal{X}}$, where $[id]: \one \rightarrow \Tone$ is the map which identifies the unique copy of $\one$ in $\Tone$ consisting exclusively of iterated identities on the single 0-cell, and $!_{\mathcal{X}}:\mathcal{X} \rightarrow \one$ is the unique map from $\mathcal{X}$ to the terminal globular set $\one$.
		\end{definition}
		
		\begin{definition}
			An \emph{algebra} $\mathcal{A}$ \emph{for a globular operad} $(\mathcal{O},\circ,e)$ is a globular set $\mathcal{A}$, thought of as a degenerate collection, together with a collection homomorphism $\omega: \mathcal{O} \mysquare \mathcal{A} \rightarrow \mathcal{A}$ which makes the diagrams
			$$\xymatrix{(\mathcal{O} \mysquare \mathcal{O}) \mysquare \mathcal{A} \ar[rr]^{\alpha^{\Coll}_{\mathcal{O},\mathcal{O},\mathcal{A}}} \ar[d]_{\circ^{\mathcal{O}} \mysquare \id_{\mathcal{A}}} & & \mathcal{O} \mysquare (\mathcal{O} \mysquare \mathcal{A}) \ar[rr]^{\id_{\mathcal{O}} \mysquare \omega} & & \mathcal{O} \mysquare \mathcal{A} \ar[d]^{\omega} \\
				\mathcal{O} \mysquare \mathcal{A} \ar[rrrr]_{\omega} & & & & \mathcal{A} }$$
			
			$$\xymatrix{ & \mathcal{O} \mysquare \mathcal{A} \ar[ddr]^{\omega} & \\
				\\
				\one \mysquare \mathcal{A} \ar[uur]^{e \mysquare \id_{\mathcal{A}}} \ar[rr]_-{\lambda_{\mathcal{A}}^{\Coll}} & & \mathcal{A}}$$
			commute. 
		\end{definition}
		
		Note here that the algebras for a globular operad in this sense would be more general than desired if we did not require that $\mathcal{A}$ be a degenerate collection.  Recall that in theory of classical operads, a module for an operad is graded over $\mathbb{N}$.  In that case, an algebra for that operad is the special case of a module concentrated in degree zero. An analogous relationship holds between modules and algebras for a globular operad, where degenerate collections serve as the special case of a collection concentrated in degree zero.  By thinking of our globular sets as degenerate collections, the $k$-cells in $\mathcal{A}$ have arities given by the $k$th iterated identity on the single 0-cell in $\Tone$.  This allows us to describe the action of globular operads entirely through the language of the composition tensor product in $\Coll$.  Because each globular cell in $\mathcal{A}$ sits above one of the iterated identity cells in $\Tone$ described above, it can be thought of as a collection whose globular cells have `empty arity', in the sense that the arity map essentially tracks only the dimension of the cell.  This perspective also has the benefit of allowing us to think of the category of globular sets as a subcategory of $\Coll$.
		
		We can alternatively use $Gend(\mathcal{X})$ as defined above to define algebras as a representation of our globular operad.
		
		\begin{definition}
			Let $o:\mathcal{O} \rightarrow \Tone$ be a globular operad.  An $\mathcal{O}\emph{-module}$ is a globular operad homomorphism $f:\mathcal{O} \rightarrow Gend(\mathcal{A})$ for some collection $a:\mathcal{A} \rightarrow \Tone$.  An $\emph{algebra}$ for $\mathcal{O}$ is an $\mathcal{O}$-module such that the collection is degenerate.
		\end{definition}
		
        These two definitions of an algebra are equivalent, as seen by currying the map $\omega$ via the adjunction between $-\mysquare \mathcal{A}$ and $[\mathcal{A},-]$, to get a collection map which has the structure of a globular operad homomorphism.

	\section{Cartesian-Duoidal Enriched Categories}
		What we shall eventually define to be a globular PRO turns out to be a special type of monoidal category enriched over a particular type of duoidal category. We shall first recall the structure of a duoidal category, as well as how to enrich over such categories, as presented by Batanin and Markl\cite{CentEnrMonCat:Batanin-Markl:2012}.  Let us begin this construction with the following definition.

		\begin{definition}
			A \emph{duoidal category} is a nonuple $(\mathcal{D}, \otimes, I, \odot, U, \delta, \phi, \theta, \boxplus)$ consisting of a category $\mathcal{D}$, a pair of 2-variable functors $\otimes: \mathcal{D} \times \mathcal{D} \rightarrow \mathcal{D}$ and $\odot: \mathcal{D} \times \mathcal{D} \rightarrow \mathcal{D}$, a pair of unit objects $I$ and $U$, three morphism $\delta: I \rightarrow I \odot I$, $\phi: U \otimes U \rightarrow U$, and $\theta: I \rightarrow U$ in $\mathcal{D}$, and a natural transformation
			$$\boxplus : \otimes (\odot(-,-), \odot (-,-)) \Rightarrow \odot( \otimes(-,-) \otimes(-,-))$$
			given by components
			$$\boxplus_{A,B,C,D}: [A \odot B] \otimes [C \odot D] \rightarrow [A \otimes C] \odot [B \otimes D]$$
			with $A,B,C,D \in Obj(\mathcal{C})$, specifying a lax middle-four interchange law between the product structures.  All of this data must satisfy the properties that $(\mathcal{D}, \otimes, I)$ and $(\mathcal{D}, \odot, U)$ are both monoidal category structures on $\mathcal{D}$, $U$ is a monoid object in $(\mathcal{D}, \otimes, I)$, $I$ is a comonoid object in $(\mathcal{D}, \odot, U)$, and for all $A,B,C,D,E,F \in Obj(\mathcal{D})$ the following diagrams commute.
			$$\adjustbox{max width=\columnwidth}{\xymatrix{((A \odot B) \otimes (C \odot D)) \otimes (E \odot F) \ar[rrrrrr]^{\alpha^{\otimes}_{A \odot B, C \odot D, E \odot F}} \ar[dd]_{\boxplus_{A,B,C,D} \otimes \id_{E \odot F}} & & & & & & (A \odot B) \otimes ((C \odot D) \otimes (E \odot F)) \ar[dd]^{\id_{A \odot B} \otimes \boxplus_{C,D,E,F}} \\
					& & & & & & \\
					((A \otimes C) \odot (B \otimes D)) \otimes (E \odot F) \ar[dd]_{\boxplus_{A \otimes C, B \otimes D, E, F}} & & & & & & (A \odot B) \otimes ((C \otimes E) \odot (D \otimes F)) \ar[dd]^{\boxplus_{A, B, C \otimes E, D \otimes F}} \\
					& & & & & & \\
					((A \otimes C) \otimes E) \odot ((B \otimes D) \otimes F) \ar[rrrrrr]_{\alpha^{\otimes}_{A, C, E} \odot \alpha^{\otimes}_{B, D, F}} & & & & & & (A \otimes (C \otimes E)) \odot(B \otimes (D \otimes F)) }}$$
			
			$$\adjustbox{max width=\columnwidth}{\xymatrix{((A \odot B) \odot C) \otimes ((D \odot E) \odot F) \ar[rrrrrr]^{\alpha^{\odot}_{A, B, C} \otimes \alpha^{\odot}_{D, E, F}} \ar[dd]_{\boxplus_{A \odot B, C, D \odot E, F}} & & & & & & (A \odot (B \odot C)) \otimes (D \odot (E \odot F)) \ar[dd]^{\boxplus_{A, B \odot C, D, E \odot F}} \\
					& & & & & & \\
					((A \odot B) \otimes (D \odot E)) \odot (C \otimes F) \ar[dd]_{\boxplus_{A, B, D, E} \odot \id_{C \otimes F}} & & & & & & (A \otimes D) \odot ((B \odot C) \otimes (E \odot F)) \ar[dd]^{\id_{A \otimes D} \odot \boxplus_{B, C, E, F}} \\
					& & & & & & \\
					((A \otimes D) \odot (B \otimes E)) \odot (C \otimes F) \ar[rrrrrr]_{\alpha^{\odot}_{A \otimes D, B \otimes E, C \otimes F}} & & & & & & (A \otimes D) \odot ((B \otimes E) \odot (C \otimes F))
			}}$$
			
			$$\adjustbox{max width=\columnwidth}{\xymatrix{A \odot B \ar[rr]^-{\overline{\lambda}^{\otimes}_{A} \odot \overline{\lambda}^{\otimes}_{B}} \ar[dd]_{\overline{\lambda}^{\otimes}_{A \odot B}} & & (I \otimes A) \odot (I \otimes B) & & A \odot B \ar[rr]^-{\overline{\rho}^{\otimes}_{A} \odot \overline{\rho}^{\otimes}_{B}} \ar[dd]_{\overline{\rho}^{\otimes}_{A \odot B}} & & (A \otimes I) \odot (B \otimes I) \\
					& & & & & &  \\
					I \otimes (A \odot B) \ar[rr]_-{\delta \otimes \id_{A \odot B}} & & (I \odot I) \otimes (A \odot B) \ar[uu]_{\boxplus_{I, I, A, B}} & & (A \odot B) \otimes I \ar[rr]_-{\id_{A \odot B} \otimes \delta} & & (A \odot B) \otimes (I \odot I) \ar[uu]_{\boxplus_{A, B, I, I}} }}$$
			
			$$\adjustbox{max width=\columnwidth}{\xymatrix{A \otimes B \ar[rr]^-{\overline{\lambda}^{\odot}_{A} \otimes \overline{\lambda}^{\odot}_{B}} \ar[dd]_{\overline{\lambda}^{\odot}_{A \otimes B}} & & (U \odot A) \otimes (U \odot B) \ar[dd]^{\boxplus_{U, A, U, B}} & & A \otimes B \ar[rr]^-{\overline{\rho}^{\odot}_{A} \otimes \overline{\rho}^{\odot}_{B}} \ar[dd]_{\overline{\rho}^{\odot}_{A \otimes B}} & & (A \odot U) \otimes (B \odot U) \ar[dd]^{\boxplus_{A, U, B, U}} \\
					& & & & & &  \\
					U \odot (A \otimes B) & & (U \otimes U) \odot (A \otimes B) \ar[ll]^-{\phi \odot \id_{A \otimes B}} & & (A \otimes B) \odot U & & (A \otimes B) \odot (U \otimes U) \ar[ll]^-{\id_{A \otimes B} \odot \phi} }}$$
		\end{definition}
		
		We think of a duoidal category $(\mathcal{D}, \otimes, I, \odot, U, \delta, \phi, \theta, \boxplus)$ as a monoidal category $(\mathcal{D}, \otimes, I)$ equipped with two lax-monoidal functors $\odot: \mathcal{D} \times \mathcal{D} \rightarrow \mathcal{D}$ and $U: \one_{\Mon\Cat_{lax}} \rightarrow \mathcal{D}$ over the monoidal product $\otimes$, where $\one_{\Mon\Cat_{lax}}$ is the trivial monoidal category with one object, and $(\mathcal{D}, \odot, U)$ is a pseudomonoid with respect to the cartesian product in $\Mon\Cat_{lax}$, the category of monoidal categories and lax-monoidal functors.  The laxivity of the functor $\otimes$ induces the interchange transformation $\boxplus$ and the morphism $\delta$.  Similarly, the laxivity of $U$ induces the morphisms $\phi$ and $\theta$.  There are then six coherence conditions all of this data satisfy, which follow from the associativity and unity coherence conditions that make $\odot$ a lax-monoidal functor over $\otimes$ and those which make $\otimes$ an oplax monoidal functor over $\odot$.  Moreover, all of this data makes $I$ a comonoid object with respect to $\odot$.  It similarly follows that $U$ (thought of as an object in $\mathcal{D}$) is a monoid object with respect to $\otimes$.

		It is straightforward to define what a lax-duoidal functor between duoidal categories must be.  We then get that duoidal categories, together with all of the lax-duoidal functors between them, form a category which we shall here denote $\Duoidal$.  As Batanin and Markl point out, it is possible to enrich over objects in this category via the following construction.
		
		\begin{definition}
			A \emph{category enriched over a duoidal category} $(\mathcal{D}, \otimes, I, \odot, U, \delta, \phi, \theta, \boxplus)$, or simply a $\mathcal{D}$-category, is an enriched category with respect to the monoidal structure $(\mathcal{D}, \otimes, I)$.  A $\mathcal{D}$\emph{-functor} is an enriched functor between two $\mathcal{D}$-categories which is enriched with respect to the same monoidal structure $(\mathcal{D}, \otimes, I)$.  A $\mathcal{D}$\emph{-transformation} is an enriched natural transformation between two $\mathcal{D}$-functors.
		\end{definition}

		Note that these enriched categories do not initially appear to use the second monoidal structure.  This second product will however become manifest when looking at the category of categories enriched over a fixed duoidal category $\mathcal{D}$.  All such $\mathcal{D}$-categories and $\mathcal{D}$-functors between them form a category, which we shall here denote $\mathcal{D}\Cat$.  Here the second monoidal product from the duoidal structure on $\mathcal{D}$ induces a monoidal structure on $\mathcal{D}\Cat$.  The tensor product
		$$\oplus: \mathcal{D}\Cat \times \mathcal{D}\Cat \rightarrow \mathcal{D}\Cat$$
		of $\mathcal{D}$-categories $\mathcal{E}$ and $\mathcal{F}$ is given as the cartesian product on objects and for $A, B \in Obj(\mathcal{E})$ and $X, Y \in Obj(\mathcal{F})$ we have
		$$E \oplus F((A, X),(B, Y)) := E(A, B) \odot F(X, Y)$$
		as the hom-objects in $\mathcal{E} \oplus \mathcal{F}$.  The unit $\one_{\oplus}$ with respect to this tensor product is the trivial $\mathcal{D}$-category consisting of a single object $*$ and a single hom-object $\one_{\oplus}(*,*) := U$, which is precisely the monoidal unit for the second monoidal structure in the underlying duoidal category $\mathcal{D}$ over which the enrichment structure is defined.  This allows us to define the following type of $\mathcal{D}$-category.
		
		\begin{definition}
			A \emph{monoidal} $\mathcal{D}$\emph{-category} $(\mathcal{M}, \diamond, \iota)$ is a pseudomonoid in the monoidal category $(\mathcal{D}\Cat, \oplus, \one_{\oplus})$, where $\mathcal{M} \in Obj(\mathcal{D}\Cat)$, the $\mathcal{D}$-functor $\diamond: \mathcal{M} \oplus \mathcal{M} \rightarrow \mathcal{M}$ is the monoidal product, and $\iota: \one_{\oplus} \rightarrow \mathcal{M}$ is the unit $\mathcal{D}$-functor such that $\diamond$ is associative and unital, with respect to $\iota$, up to $\mathcal{D}$-transformations.
		\end{definition}
		
		Having the structure of a pseudomonoid implies the existence of a morphism
		$$\boxdot_{X, Y, Z, W}: M(X,Y) \odot M(Z, W) \rightarrow M(X \diamond Z, Y \diamond W)$$
		for every $X, Y, Z, W \in Obj(\mathcal{M})$, which encodes how $\diamond$ acts on morphisms.  Moreover, these satisfy the usual pentagon and triangle coherence conditions to ensure that $(\mathcal{M}, \diamond, \iota)$ is a pseudomonoid.  This structure will play a role below when defining lax-monoidal functors between $\mathcal{D}$ categories.  But before we define these functors, it is important to note one final fact.  Notice that $\mathcal{D}$-categories come equipped with an underlying category.  The underlying category has the same objects as the $\mathcal{D}$-category.  Morphisms in the underlying category are given by
		$$U(M)(X,Y) := D(I, M(X,Y))$$
		for $X, Y \in \mathcal{M}$.
  
        As found in \cite{CentEnrMonCat:Batanin-Markl:2012},  we have the following notion of a lax-monoidal functor between monoidal $\mathcal{D}$-categories.  
		
		\begin{definition}
			A \emph{lax-monoidal} $\mathcal{D}$\emph{-functor} between monoidal $\mathcal{D}$-categories $(\mathcal{M}, \diamond, \iota)$ and $(\mathcal{M}', \diamond', \iota')$ is a triple $(F, \widehat{F}, e)$ consisting of an underlying $\mathcal{D}$-functor $F: \mathcal{M} \rightarrow \mathcal{M}'$ together with a $\mathcal{D}$-transformation
			$$\widehat{F}: \diamond'(F(-), F(-)) \Rightarrow F(\diamond(-,-))$$
			given by components
			$$\widehat{F}_{X, Y}: I \rightarrow M'(F(X) \diamond' F(Y), F(X \diamond Y))$$
			for $X, Y \in Obj(\mathcal{M})$, and a morphism
			$$e: \iota'(*) \rightarrow F(\iota(*))$$
			in $(\mathcal{M}', \diamond', \iota')$ such that the underlying functor $F$ is a lax-monoidal functor between the underlying monoidal categories, not thought of as $\mathcal{D}$-categories.  Moreover, this data must satisfy the following coherence condition for all $X, Y, Z, W \in Obj(\mathcal{M})$
			$$\adjustbox{max width=\columnwidth}{\xymatrix{ & \text{\small$M(X \diamond Z, Y \diamond W) \otimes I$} \ar[dr]^{F \otimes \widehat{F}_{X, Z}} & \\ 
					\text{\small$M(X \diamond Z, Y \diamond W)$} \ar[ur]^{\overline{\rho}^{\mathcal{D}}_{M(X \diamond Z, Y \diamond W)}\qquad} & & \text{\small$M'(F(X \diamond Z), F(Y \diamond W)) \otimes M'(F(X) \diamond' F(Z), F(X \diamond Z))$} \ar[dd]^{\circ^{\mathcal{M}'}_{F(X) \diamond' F(Z), F(X \diamond Z), F(Y \diamond W)}} \\
					& & \\ 
					\text{\small$M(X, Y) \odot M(Z,W)$} \ar[uu]^{\boxdot^{\mathcal{M}}_{X, Y, Z, W}} \ar[dd]_{F \odot F} & & \text{\small$M'(F(X) \diamond' F(Z), F(Y \diamond W))$} \\ 
					& & \\ 
					\text{\small$M'(F(X),F(Y)) \odot M'(F(Z),F(W))$} \ar[dr]_(.40){\overline{\lambda}^{\mathcal{D}}_{M'(F(X), F(Y)) \odot M'(F(Z), F(W))}\qquad} & & \text{\small$M'(F(Y) \diamond' F(W), F(Y \diamond W)) \otimes M'(F(X) \diamond' F(Z), F(Y) \diamond' F(W))$} \ar[uu]_{\circ^{\mathcal{M}'}_{F(X) \diamond' F(Z), F(Y) \diamond' F(W), F(Y \diamond W)}} \\ 
					& \text{\small$I \otimes [M'(F(X), F(Y)) \odot M'(F(Z), F(W))]$} \ar[ur]_(.55){\qquad\qquad\qquad\widehat{F}_{Y, W} \otimes \boxdot^{\mathcal{M}'}_{F(X), F(Y), F(Z), F(W)}} & }}$$
			ensuring that the two pseudomonoid structures are compatible.
		\end{definition}
		
		We will need the following special type of duoidal category.
		
		\begin{definition}
			A \emph{cartesian-duoidal category} is a duoidal category $(\mathcal{D}, \otimes, I, \times, \id, \delta, \phi,$ $\theta, \boxplus)$ such that the second monoidal structure $(\mathcal{D}, \times, \id)$ is a cartesian monoidal category.
		\end{definition}
		
		Note that a routine check shows that cartesian-duoidal categories are precisely monoidal categories with finite products.  We shall however use the term cartesian-duoidal here for brevity.  Moreover, Cartesian-duoidal categories form a subcategory of $\Duoidal$, which we shall here denote by $\CartDuoidal$.  We shall call a category enriched over a cartesian-duoidal category $\mathcal{C}$ a $\mathcal{C}$\emph{-category}.  For a fixed $\mathcal{C}$ we shall denote the category of all such enriched categories $\mathcal{C}\Cat$.  We can finally state succinctly the key definition of this section.

		\begin{definition}
			A \emph{cartesian-duoidal enriched monoidal category} is a pseudomonoid in $(\mathcal{C}\Cat, \oplus, \one_{\oplus})$.
		\end{definition}
		
		Note that a strict cartesian-duoidal enriched monoidal category would simply be a monoid in $(\mathcal{C}\Cat, \oplus, \one_{\oplus})$.  Our present interest in these monoidal categories is that they allow us to generalize the classical definition of a PROs, which are merely a special type of monioid in $\Cat$.  We conclude this section with the following definition.
		
		\begin{definition}
			An \emph{enriched PRO} is a strict duoidal enriched monoidal category enriched over a duoidal category $\mathcal{D}$ such that the object set can be identified with $\mathbb{N}$ and the monoidal product on objects is identified with addition of natural numbers.
		\end{definition}

        Note that the duoidal structure in $\Coll$ that is used to create globular PROs does force them to be like the special case of the Lawvere variant of the more general notion of a PRO since the second monoidal structure is cartesian product. But if one were to enrich over another duoidal category and consider the subcategory of monoids in this category of enriched categories that have the extra properties that their objects are identifiable with $\mathbb{N}$ and the monoidal product on objects is addition on $\mathbb{N}$, we would get something more general than an enriched Lawvere theory.  We currently leave the properties of such structures open for future investigation.  The objects of study in the present work are simply the globular variant of the more general enriched PROs, not all of which prima facie must come from enrichment over cartesian-duoidal categories.
		
		\section{Defining Globular PROs}
		Just as classical PROs may be presented as a specific type of monoidal category, in what follows we will see that a globular PRO is simply a specific type of cartesian-duoidal enriched category.  Before formally defining globular PROs we first need to ensure that $\Coll$ is cartesian-duoidal.  But since $\Coll$ is a slice topos it has a natural cartesian product.  It then follows immediately from the proposition above that since $\Coll$ has finite products it is moreover cartesian-duoidal.  This then ensures us that the category $\Coll$ has the appropriate structure for us to define globular PROs via the following construction.
		
		\begin{definition}
			A \emph{globular PRO} is a strict cartesian-duoidal enriched monoidal category $(\mathcal{P}, +, O)$ enriched over the cartesian-duoidal category $\Coll$ such that the object set of $\mathcal{P}$ is isomorphic to $\mathbb{N}$, the bifunctor $+: \mathcal{P} \times \mathcal{P} \rightarrow \mathcal{P}$ acts as addition of natural numbers on objects, and the unit $\Coll$-functor $O: \one_{\oplus} \rightarrow \mathcal{P}$ maps * to the additive identity $0 \in \mathbb{N}$.
		\end{definition}
		
		Note that a globular PRO is precisely an enriched cartesian PRO enriched over $\Coll$.  More explicitly, a globular PRO $\mathcal{P}$ has the following structure.  $\mathcal{P}$ has as its object set $Obj(\mathcal{P}) \cong \mathbb{N}$.  For each pair $n,m \in \mathbb{N}$ we have a \emph{hom-object} $h_{n,m}:\mathcal{P}(n,m) \rightarrow \Tone$ from $\Coll$, which we will often simply write as $\mathcal{P}(n,m)$.  For each triple $n,m,l \in \mathbb{N}$ we have a collection homomorphism $\circ_{n,m,l}: \mathcal{P}(m,l) \mysquare \mathcal{P}(n,m) \rightarrow \mathcal{P}(n,l)$ called \emph{composition at} $(n,m,l)$.  We also have for each $n \in \mathbb{N}$ a collection homomorphism $j_n: \one \rightarrow \mathcal{P}(n,n)$ called the \emph{identity identification at} $n$.
		
		All of this data must satisfy, for all $n,m,l,k \in \mathbb{N}$, the following two commutative diagrams ensuring that composition in $\mathcal{P}$ is associative and unital.
		$$\adjustbox{max width=\columnwidth}{\xymatrix{[\mathcal{P}(l,k) \mysquare \mathcal{P}(m,l)] \mysquare \mathcal{P}(n,m) \ar[rrrr]^{\alpha^{\Coll_{\mysquare}}_{\mathcal{P}(l,k),  \mathcal{P}(m,l), \mathcal{P}(n,m)}} \ar[dd]_{\circ_{m,l,k} \mysquare \id_{\mathcal{P}(n,m)}} & & & & \mathcal{P}(l,k) \mysquare [\mathcal{P}(m,l) \mysquare \mathcal{P}(n,m)] \ar[dd]^{\id_{\mathcal{P}(l,k)} \mysquare \circ_{n,m,l}} \\
				& & & & \\
				\mathcal{P}(m,k) \mysquare \mathcal{P}(n,m) \ar[ddrr]_{\circ_{n,m,k}\quad} & & & & \mathcal{P}(l,k) \mysquare \mathcal{P}(n,l) \ar[ddll]^{\quad\circ_{n,l,k}} \\
				& & & & \\
				& & \mathcal{P}(n,k) & & }}$$
		
		$$\xymatrix{\mathcal{P}(m,m) \mysquare \mathcal{P}(n,m) \ar[rr]^-{\circ_{n,m,m}} & & \mathcal{P}(n,m) & & \mathcal{P}(n,m) \mysquare \mathcal{P}(n,n) \ar[ll]_-{\circ_{n,n,m}} \\
			& & & & \\
			\one \mysquare \mathcal{P}(n,m) \ar[uu]^{j_m \mysquare \id_{\mathcal{P}(n,m)}} \ar[uurr]_{\lambda^{\Coll_{\mysquare}}_{\mathcal{P}(n,m)}} & & & & \mathcal{P}(n,m) \mysquare \one \ar[uu]_{\id_{\mathcal{P}(n,m)} \mysquare j_n} \ar[uull]^{\rho^{\Coll_{\mysquare}}_{\mathcal{P}(n,m)}} }$$
		Here $\Coll_{\mysquare}$ is used to denote that these structure maps are those for the $\mysquare$ monoidal product as opposed to that of the cartesian product in $\Coll$.
		
		The globular PRO $\mathcal{P}$ must also come equipped with a monoidal structure encoded in the 2-variable functor $+:\mathcal{P} \times \mathcal{P} \rightarrow \mathcal{P}$.  Since $\mathcal{P}$ is an enriched category, the functor $+$ must moreover be an enriched functor of 2-variables.  More precisely, this means that $+$ is given on objects by the addition map $+:\mathbb{N} \times \mathbb{N} \rightarrow \mathbb{N}$ together with, for each $n,m,l,k \in \mathbb{N}$, collection homomorphisms $+_{n,m,l,k}: \mathcal{P}(n,m) \times \mathcal{P}(l,k) \rightarrow \mathcal{P}(n+l,m+k)$, all of which must, for all $n,m,l,k,r,s \in \mathbb{N}$, make the following diagrams commute.
		$$\adjustbox{max width=\columnwidth}{\xymatrix{[\mathcal{P}(m,r) \times \mathcal{P}(k,s)] \mysquare [\mathcal{P}(n,m) \times \mathcal{P}(l,k)] \ar[rrrr]^-{\boxtimes_{\mathcal{P}(m,r), \mathcal{P}(k,s), \mathcal{P}(n,m), \mathcal{P}(l,k)}} \ar[dd]_{+_{m,r,k,s} \mysquare +_{n,m,l,k}} & & & & [\mathcal{P}(m,r) \mysquare \mathcal{P}(n,m)] \times [\mathcal{P}(k,s) \mysquare \mathcal{P}(l,k)] \ar[dd]^{\circ_{n,m,r} \times \circ_{l,k,s}} \\
				& & & & \\
				\mathcal{P}(m+k, r+s) \mysquare \mathcal{P}(n+l, m+k) \ar[ddrr]^{\quad\circ_{n+l,m+k,r+s}} & & & & \mathcal{P}(n,r) \times \mathcal{P}(l,s) \ar[ddll]_{+_n,r,l,s\quad} \\
				& & & & \\
				& & \mathcal{P}(n+l, r+s) & & }}$$
		
		$$\xymatrix{ & & \one \times \one \cong \one \ar[ddll]_{j_n \times j_m} \ar[ddrr]^{j_{n+m}} & & \\
			& & & & \\
			\mathcal{P}(n,n) \times \mathcal{P}(m,m) \ar[rrrr]^{+_{n,n,m,m}} & & & & \mathcal{P}(n+m,n+m) }$$
		
		$$\adjustbox{max width=\columnwidth}{\xymatrix{[\mathcal{P}(n,m) \times \mathcal{P}(l,k)] \times \mathcal{P}(r,s) \ar[rrrr]^{\alpha^{\Coll_{\times}}_{\mathcal{P}(n,m), \mathcal{P}(l,k), \mathcal{P}(r,s)}} \ar[dd]_{+_{n,m,l,k} \times \id_{\mathcal{P}(r,s)}} & & & & \mathcal{P}(n,m) \times [\mathcal{P}(l,k) \times \mathcal{P}(r,s)] \ar[dd]^{\id_{\mathcal{P}(n,m)} \times +_{l,k,r,s}} \\
				\\
				\mathcal{P}(n+l,m+k) \times \mathcal{P}(r,s) \ar[ddrr]^{\quad+_{n+l,m+k,r,s}} & & & & \mathcal{P}(n,m) \times \mathcal{P}(l+r,k+s) \ar[ddll]_{+_{n,m,l+r,k+s}\qquad}
				\\
				\\
				& & \mathcal{P}(n+l+r,m+k+s) & & }}$$
		
		$$\adjustbox{max width=\columnwidth}{\xymatrix{\Tone \times \mathcal{P}(n,m) \ar[rrr]^-{O \times \id_{\mathcal{P}(n,m)}} \ar[ddrrr]_{\lambda^{\Coll_{\times}}_{\mathcal{P}(n,m)}} & & & \mathcal{P}(0,0) \times \mathcal{P}(n,m) \cong \mathcal{P}(n,m) \times \mathcal{P}(0,0) \ar[dd]^{+_{n,m,0,0}}_{+_{0,0,n,m}} & & & \mathcal{P}(n,m) \times \Tone \ar[lll]_-{\id_{\mathcal{P}(n,m)} \times O} \ar[ddlll]^{\quad\rho^{\Coll_{\times}}_{\mathcal{P}(n,m)}} \\
				\\
				& & & \mathcal{P}(n,m) & & & }}$$
		The first two diagrams ensure that $+$ is a $\Coll$-functor.  The second two ensure that $\mathcal{P}$ is a monoid object $\Coll\Cat$ with respect to the product $\oplus$, which in this context is simply the cartesian product on homsets.  We again adopt the notation $\Coll_{\times}$ to distinguish the structure maps from the cartesian structure on $\Coll$ from the $\mysquare$ monoidal product.
		
		\begin{definition}
			A \emph{morphism of globular PROs} between globular PROs $\mathcal{P}$ and $\mathcal{P}'$ is a strict monoidal $\Coll$-functor $(F,\widehat{F}, e): \mathcal{P} \rightarrow \mathcal{P}'$.  More precisely, such a morphism consists of an underlying $\Coll$-functor
			$$F: \mathcal{P} \rightarrow \mathcal{P}'$$
			that is the identity on objects, a $\Coll$-enriched natural transformation
			$$\widehat{F}:+(F(-),F(-)) \Rightarrow F(+(-,-))$$
			with each component
			$$\widehat{F}_{n,m}: I \rightarrow \mathcal{P}'(F(n) + F(m), F(n + m))$$
			for $n,m \in \mathbb{N}$ having $\widehat{F}_{n,m} = j^{\mathcal{P}'}_{n + m}$, and a morphism
			$$e: I \rightarrow \mathcal{P}'(0,F(0))$$
			such that $e=j^{\mathcal{P}'}_{0}$, all of which makes $F$ a strict monoidal functor between the underlying categories $\mathcal{P}$ and $\mathcal{P}'$ not thought of as $\Coll$-categories.  Moreover, the diagram
			$$\adjustbox{max width=\columnwidth}{\xymatrix{ & \text{\small$P(n + l, m + k) \mysquare I$} \ar[dr]^{F \mysquare \widehat{F}_{n, l}} & \\ 
					\text{\small$P(n + l, m + k)$} \ar[ur]^{\overline{\rho}^{\Coll}_{P(n + l, m + k)}\qquad} & & \text{\small$P'(F(n + l), F(m + k)) \mysquare P'(F(n) +' F(l), F(n + l))$} \ar[dd]^{\circ^{\mathcal{P}'}_{F(n) +' F(l), F(n + l), F(m + k)}} \\
					& & \\ 
					\text{\small$P(n, m) \times P(l,k)$} \ar[uu]^{\boxdot^{\mathcal{P}}_{n, m, l, k}} \ar[dd]_{F \times F} & & \text{\small$P'(F(n) +' F(l), F(m + k))$} \\ 
					& & \\ 
					\text{\small$P'(F(n),F(m)) \times P'(F(l),F(k))$} \ar[dr]_(.4){\overline{\lambda}^{\Coll}_{P'(F(n), F(m)) \times P'(F(l), F(k))}\qquad} & & \text{\small$P'(F(m) +' F(k), F(m + k)) \mysquare P'(F(n) +' F(l), F(m) +' F(k))$} \ar[uu]_{\circ^{\mathcal{P}'}_{F(n) +' F(l), F(m) +' F(k), F(m + k)}} \\ 
					& \text{\small$I \mysquare [P'(F(n), F(m)) \times P'(F(l), F(k))]$} \ar[ur]_{\quad\qquad\qquad\widehat{F}_{m, k} \mysquare \boxdot^{\mathcal{P}'}_{F(n), F(m), F(l), F(k)}} & }}$$
			must commute for all $n,m,l,k \in \mathbb{N}$.
		\end{definition}
		
		Together with the morphisms between them, Globular PROs form a category which we shall here denote by $\GlobPRO$.
		
		We define below a special class of globular PROs that will become of great importance later.
		
	    \begin{definition}
			A \emph{weakenable globular PRO} is a globular PRO with the property that, for each $n \in \mathbb{N}$, there exists a collection morphism $\Gamma_n: \Tone \rightarrow \mathcal{P}(n,n)$.
		\end{definition}
		
		Note that weakenable globular PRO's have extra structure that allows for us to define an `extrinsic' composition on the cells in each hom-collection $\mathcal{P}(n,m)$.  In fact, it essentially imposes upon each hom-collection the structure of a strict $\omega$-category.  This is done via the unlabled map defined via the following commutative diagram.
		
		$$\xymatrix{\Tone \mysquare \mathcal{P}(n,m) \ar[rr] \ar[ddr]_{\Gamma_m \mysquare \one} & & \mathcal{P}(n,m) \\
		                                                                            \\
		             & \mathcal{P}(m,m) \mysquare \mathcal{P}(n,m) \ar[uur]_{\circ_{n,m,m}} & 
		 }$$

	\section{The Endomorphism Globular PRO}
		Just as with ordinary PROs, before formalizing the notion of an algebra for a globular PRO, we will first construct the $\emph{endomorphism globular PRO}$ given a degenerate collection $a:A \rightarrow \Tone$, which we shall denote by $GEnd(A)$.  Note that in the construction that follows it is not strictly necessary that the collection $a:A \rightarrow \Tone$ be degenerate in order to define a endomorphism globular PRO.  We however make this assumption for the purpose of defining algebras for globular PROs.  If $a:A \rightarrow \Tone$ is not degenerate, the final result of this construction gives instead the structure of a module.
		
		We first construct the PRO $GEnd(A)$ by specifying its objects.  $GEnd(A)$ has as its set of objects all successive cartesian powers $A^n$ in $\Coll$ for $n \in \mathbb{N}$.  These can, as in the non-globular case, be naturally identified with $\mathbb{N}$.  Under this identification the hom-objects $GEnd(A)(n,m)$ in $GEnd(A)$ are exactly the internal hom $[A^n,A^m]$ of the closed structure corresponding to the product $\mysquare$ in $\Coll$.  To understand composition in $GEnd(A)$ we first need to consider again the hom-tensor adjunction $- \mysquare B \dashv [B,-]:\Coll \rightarrow \Coll$.  Let
		$$\epsilon^{B}: [B,-] \mysquare B \Rightarrow \id_{\Coll}$$
		be the counit of this adjunction, which has components
		$$\epsilon^{B}_{X}: [B,X] \mysquare B \rightarrow X$$
		for each collection $x:X \rightarrow \Tone$.  We will call each of these components $\emph{evaluation}$.  Now consider the composition
		$$\adjustbox{max width=\columnwidth}{\xymatrix{\theta_{X,Y,Z}:([Y,Z] \mysquare [X,Y]) \mysquare X \ar[rrr]^-{\alpha^{\Coll_{\mysquare}}_{[Y,Z],[X,Z],X}} & & & [Y,Z] \mysquare ([X,Y] \mysquare X) \ar[rr]^-{\id_{[Y,Z]} \mysquare \epsilon^{X}_{Y}} & & [Y,Z] \mysquare Y \ar[r]^-{\epsilon^{Y}_{Z}} & Z
		}}$$
		in $\Coll$.  Since $\theta_{X,Y,Z} \in \textrm{Hom}_{\Coll}(([Y,Z] \mysquare [X,Y]) \mysquare X, Z)$ we can curry it to get the morphism
		$$\circ_{X,Y,Z}: [Y,Z] \mysquare [X,Y] \rightarrow [X,Z]$$
		which we define to be the $(X,Y,Z)$ component of the composition in $GEnd(A)$.  In order to get identities for $GEnd(A)$ we must then consider the left unitor 
		$$\lambda^{\Coll_{\mysquare}}_X:\one \mysquare X \rightarrow X$$
		with respect to the monoidal product $\mysquare$ in $\Coll$.  The identity identification at $X$ is then defined to be the currying of $\lambda^{\Coll_{\mysquare}}_X$,
		$$j_X: \one \rightarrow [X,X]$$
		similar to that of composition.  To define the monoidal product $+$, we first consider the morphism $\kappa_{A^n,A^m,A^l,A^k}$ which we define via the diagram below.
		$$\adjustbox{max width=\columnwidth}{\xymatrix{ ([A^n,A^m] \times [A^l,A^k]) \mysquare A^{n+l} \cong ([A^n,A^m] \times [A^l,A^k]) \mysquare (A^n \times A^l) \ar[ddr]_{\boxtimes_{A^n,A^m,A^l,A^k}\qquad} \ar@{-->}[rrrr]^{\kappa_{A^n,A^m,A^l,A^k}} & & & & A^m \times A^k \cong A^{m+k} \\
				\\
				& ([A^n,A^m] \mysquare A^n) \times ([A^l,A^k] \mysquare A^l) \ar[uurrr]_{\epsilon^{A^n}_{A^m} \times \epsilon^{A^l}_{A^k}}
		}}$$
		This allows us to then define
		$$+_{A^n,A^m,A^l,A^k}:[A^n,A^m] \times [A^l,A^k] \rightarrow [A^{n+l},A^{m+k}]$$
		as the currying of $\kappa_{A^n,A^m,A^l,A^k}$.  This is then the monoidal product in $GEnd(A)$.  It remains to show that all of this data satisfies the following commutative diagrams:
		$$\adjustbox{max width=\columnwidth}{\xymatrix{([A^l,A^k] \mysquare [A^m,A^l]) \mysquare [A^n,A^m] \ar[rrrr]^{\alpha^{\Coll_{\mysquare}}_{[A^l,A^k], [A^m,A^l], [A^n,A^m]}} \ar[dd]_{\circ_{A^m,A^l,A^k} \mysquare \id_{[A^n,A^m]}} & & & & [A^l,A^k] \mysquare ([A^m,A^l] \mysquare [A^n,A^m]) \ar[dd]^{\id_{[A^l,A^k]} \mysquare \circ_{A^n,A^m,A^l}} \\
				& & & & \\
				[A^m,A^k] \mysquare [A^n,A^m] \ar[ddrr]_{\circ_{A^n,A^m,A^k}\quad} & & & & [A^l,A^k] \mysquare [A^n,A^l] \ar[ddll]^{\qquad\circ_{A^n,A^l,A^k}} \\
				& & & & \\
				& & [A^n,A^k] & & }}$$
		
		$$\xymatrix{[A^m,A^m] \mysquare [A^n,A^m] \ar[rr]^-{\circ_{A^n,A^m,A^m}} & & [A^n,A^m] & & [A^n,A^m] \mysquare [A^n,A^n] \ar[ll]_-{\circ_{A^n,A^n,A^m}} \\
			& & & & \\
			\one \mysquare [A^n,A^m] \ar[uu]^{j_{A^m} \mysquare \id_{[A^n,A^m]}} \ar[uurr]_{\lambda^{\Coll_{\mysquare}}_{[A^n,A^m]}} & & & & [A^n,A^m] \mysquare \one \ar[uu]_{\id_{[A^n,A^m]} \mysquare j_{A^n}} \ar[uull]^{\rho^{\Coll_{\mysquare}}_{[A^n,A^m]}} }$$
		
		$$\adjustbox{max width=\columnwidth}{\xymatrix{([A^m,A^r] \times [A^k,A^s]) \mysquare ([A^n,A^m] \times [A^l,A^k]) \ar[rrrr]^-{\boxtimes_{[A^m,A^r], [A^k,A^s], [A^n,A^m], [A^l,A^k]}} \ar[dd]_{+_{A^m,A^r,A^k,A^s} \mysquare +_{A^n,A^m,A^l,A^k}} & & & & ([A^m,A^r] \mysquare [A^n,A^m]) \times ([A^k,A^s] \mysquare [A^l,A^k]) \ar[dd]^{\circ_{A^n,A^m,A^r} \times \circ_{A^l,A^k,A^s}} \\
				& & & & \\
				[A^{m+k},A^{r+s}] \mysquare [A^{n+l},A^{m+k}] \ar[ddrr]^{\quad\circ_{A^{n+l},A^{m+k},A^{r+s}}} & & & & [A^n,A^r] \times [A^l,A^s] \ar[ddll]_{+_{A^n,A^r,A^l,A^s}\quad} \\
				& & & & \\
				& & [A^{n+l},A^{r+s}] & & }}$$
		
		$$\xymatrix{ & & \one \times \one \cong \one \ar[ddll]_{j_{A^n} \times j_{A^m}} \ar[ddrr]^{j_{A^{n+m}}} & & \\
			& & & & \\
			[A^n,A^n] \times [A^m,A^m] \ar[rrrr]^{+_{A^n,A^n,A^m,A^m}} & & & & [A^{n+m},A^{n+m}] }$$
		
		$$\adjustbox{max width=\columnwidth}{\xymatrix{([A^n,A^m] \times [A^l,A^k]) \times [A^r,A^s] \ar[rrrr]^{\alpha^{\Coll_{\times}}_{[A^n,A^m], [A^l,A^k], [A^r,A^s]}} \ar[dd]_{+_{A^n,A^m,A^l,A^k} \times \id_{[A^r,A^s]}} & & & & [A^n,A^m] \times ([A^l,A^k] \times [A^r,A^s]) \ar[dd]^{\id_{[A^n,A^m]} \times +_{A^l,A^k,A^r,A^s}} \\
				\\
				[A^{n+l},A^{m+k}] \times [A^r,A^s] \ar[ddrr]^{\qquad+_{A^{n+l},A^{m+k},A^r,A^s}} & & & & [A^n,A^m] \times [A^{l+r},A^{k+s}] \ar[ddll]_{+_{A^n,A^m,A^{l+r},A^{k+s}}\qquad}
				\\
				\\
				& & [A^{n+l+r},A^{m+k+s}] & & }}$$
		
		$$\adjustbox{max width=\columnwidth}{\xymatrix{\Tone \times [A^n,A^m] \ar[rrr]^-{O^{GEnd(A)} \times \id_{[A^n,A^m]}} \ar[ddrrr]_{\lambda^{\Coll_{\times}}_{[A^n,A^m]}} & & & [\Tone,\Tone] \times [A^n,A^m] \cong [A^n,A^m] \times [\Tone,\Tone] \ar[dd]^{+_{A^n,A^m,\one,\one}}_{+_{\one,\one,A^n,A^m}} & & & [A^n,A^m] \times \Tone \ar[lll]_-{\id_{[A^n,A^m]} \times O^{GEnd(A)}} \ar[ddlll]^{\quad\rho^{\Coll_{\times}}_{[A^n,A^m]}} \\
				\\
				& & & [A^n,A^m] & & & }}$$
		Note that the final diagram maps $\Tone$ to $[\Tone, \Tone]$ rather than to $[A^0,A^0]$.  This is because $a^0: A^0 \rightarrow \Tone$ is the empty cartesian product and is hence the terminal collection $\id: \Tone \rightarrow \Tone$.  But moreover, the collection $[\Tone, \Tone]$ is also $\id$.  To see this, consider all possible collection homomorphisms from $a \mysquare \id$ to $\id$.  Because $\id$ is terminal, there is only one.  And since the functor $- \mysquare \id$ is adjoint to $[\id, -]$, the natural isomorphism of homsets implies that there is a single unique map from $A$ to $[\id, \id]$ (i.e. $[\Tone, \Tone]$).  Hence $[\Tone, \Tone]$ is the terminal collection $\id$.
		
		In showing the commutativity of these diagrams we will often suppress associators by MacLane's coherence theorem. For the first diagram, the one asserting associativity of composition in $GEnd(A)$, we consider the diagram in Figure 1 whose boundary is obtained by currying the boundary of the original diagram.  The commutativity of this second diagram then implies the commutativity of the original.  The commutativity of the top leftmost square follows by the functoriality of $\mysquare$.  The commutativity of the middle, bottom left, and top right squares follows from the fact that composing and then evaluating is equivalent by definition to two consecutive evaluations.  Finally, the bottom right square commutes trivially.  Therefore composition in $GEnd(A)$ is associative.
		\begin{landscape}
			$$\adjustbox{max width=\columnwidth}{\xymatrix{ & & & & [A^l,A^k] \mysquare [A^m,A^l] \mysquare [A^n,A^m] \mysquare A^n \ar[ddllll]_{\circ_{A^m,A^l,A^k} \mysquare \id_{[A^n,A^m]} \mysquare \id_{A^n}\qquad\quad} \ar[ddddll]^{\quad\id_{[A^l,A^k]} \mysquare \id_{[A^m,A^l]} \mysquare \epsilon^{A^m}_{A^l}} \ar[ddddrr]^{\quad\id_{[A^l,A^k]} \mysquare \circ_{A^n,A^m,A^l} \id_{A^n}} \ar[ddrrrr]^{\qquad\qquad\id_{[A^l,A^k]} \mysquare \id_{[A^m,A^l]} \mysquare \epsilon^{A^n}_{A^m}} & & & & \\
					\\
					[A^m,A^k] \mysquare [A^n,A^m] \mysquare A^n \ar[dddd]_{\id_{[A^m,A^k]} \mysquare \epsilon^{A^n}_{A^m}} & & & & & & & & [A^l,A^k] \mysquare [A^m,A^l] \mysquare A^m \ar[dddd]^{\id_{[A^l,A^k]} \mysquare \epsilon^{A^m}_{A^l}} \\
					\\
					& & [A^l,A^k] \mysquare [A^m,A^l] \mysquare A^m \ar[ddll]_{\circ_{A^m,A^l,A^k} \mysquare \id_{A^m}\quad} \ar[ddrr]^{\qquad\id_{[A^l,A^k]} \mysquare \epsilon^{A^m}_{A^l}} & & & & [A^l,A^k] \mysquare [A^n,A^l] \mysquare A^n \ar[ddll]_{\id_{[A^l,A^k]} \mysquare \epsilon^{A^n}_{A^l}\qquad} \ar[ddrr]^{\quad\id_{[A^l,A^k]} \mysquare \epsilon^{A^n}_{A^l}}  & & \\
					\\
					[A^m,A^k] \mysquare A^m \ar[ddrrrr]^{\epsilon^{A^m}_{A^k}} & & & & [A^l,A^k] \mysquare A^l \ar[dd]_{\epsilon^{A^l}_{A^k}} & & & & [A^l,A^k] \mysquare A^l \ar[ddllll]_{\epsilon^{A^l}_{A^k}} \\
					\\
					& & & & A^k & & & & 
			}}$$
			$$\textbf{Figure 1}$$
		\end{landscape}

		We then consider the following two diagrams
		$$\adjustbox{max width=\columnwidth}{\xymatrix{ & & [A^m,A^m] \mysquare [A^n, A^m] \mysquare A^n \ar[rrrr]^{\id_{[A^m,A^m]} \mysquare \epsilon^{A^n}_{A^m}} & & & & [A^m, A^m] \mysquare A^m \ar[ddrr]^{\epsilon^{A^m}_{A^m}}  & & \\
				\\
				\one \mysquare [A^n,A^m] \mysquare A^n \ar[uurr]^{j_{A^m} \mysquare \id_{[A^n,A^m]} \mysquare \id_{A^n}\qquad} \ar[ddrrrr]_{\id_{\one} \mysquare \epsilon^{A^n}_{A^m}\qquad} & & & & & & & & A^m \\
				\\
				& & & & \one \mysquare A^m \ar[uuuurr]^{j_{A^m} \mysquare \id_{A^m}} \ar[uurrrr]_{\quad\lambda^{\Coll_{\mysquare}}_{A^m}} & & & & 
		}}$$
		
		$$\adjustbox{max width=\columnwidth}{\xymatrix{[A^n,A^m] \mysquare (\one \mysquare A^n) \ar[rrr]^-{\id_{[A^n,A^m]} \mysquare j_{A^n} \mysquare \id_{A^n}} \ar[rrrdd]^{\qquad\id_{[A^n,A^m]} \mysquare \lambda^{\Coll_{\mysquare}}_{A^n}} \ar[dd]_{\overline{\alpha}^{\Coll_{\mysquare}}_{[A^n,A^m], \one, A^n}} & & & [A^n,A^m] \mysquare ([A^n,A^n] \mysquare A^n) \ar[rrr]^-{\id_{[A^n,A^m]} \mysquare \epsilon^{A^n}_{A^n}} & & & [A^n,A^m] \mysquare A^n \ar[dd]^{\epsilon^{A^n}_{A^m}} \\
				\\
				([A^n,A^m] \mysquare \one) \mysquare A^n \ar[rrr]^-{\rho^{\Coll_{\mysquare}}_{[A^n,A^m]} \mysquare \id_{A^n}} & & & [A^n,A^m] \mysquare A^n \ar[uurrr]^{=} \ar[rrr]^{\epsilon^{A^n}_{A^m}} & & & A^m
		}}$$
		whose boundaries are obtained by currying the boundaries of the left and right unit axiom diagrams, respectively, for $GEnd(A)$ to be a $\Coll$-cat.  Note then that the left square in the first curried diagram commutes by the naturality of $\epsilon$ while the right triangle commutes due to the fact that $j_X$ was defined to be the currying of $\lambda^{\Coll_{\mysquare}}_X$, the $X$ component of the left unitor from $\Coll$ with respect to the product $\mysquare$.  For the second diagram we have that the leftmost triangle commutes as an instance of the triangle coherence condition with respect to the monoidal product $\mysquare$ in $\Coll$.  The middle triangle commutes by the definition of $j$, just as we saw for the rightmost triangle in the previous diagram.  The final triangle in second diagram commutes trivially.  Thus we have that composition in $GEnd(A)$ is also unital with respect to the same unitors in $\Coll$.
		
		We then consider the diagram in Figure 2 whose boundary is obtained by currying the diagram asserting that the product $+$ respects composition in $GEnd(A)$.  The top left square commutes by the functoriality of the $\mysquare$ product.  The square to the right of this functoriality square commutes by the adjunction used to define $+$.  The square to the right of these first two commutative squares also commutes by the functoriality of $\mysquare$.  The bottom left square commutes by the definition of $\epsilon$ implying that composition followed by evaluation is the same as double evaluation.  The square to its right commutes by the adjunction used to define $+$.  The top right square commutes from the fact that the product $\boxtimes$ involves a projection.  Hence the two sides of this square must commute as they differ only in the order in which those projections occur.  The square below and to the left as well as the square below and to the right of the previous square commute by the naturality of the $\boxtimes$ product.  The bottom left square commutes by the fact that composition followed by evaluation is the same as double evaluation.
		\begin{landscape}
			$$\adjustbox{max width=\columnwidth}{\xymatrix{ & & & & & & ([A^m,A^r] \times [A^k,A^s]) \mysquare ([A^n,A^m] \times [A^l,A^k]) \mysquare (A^n \times A^l) \ar[dddllllll]_{+_{A^m,A^r,A^k,A^s} \mysquare +_{A^n,A^l,A^m,A^k} \mysquare \id_{A^{n+l}}\qquad\qquad\qquad} \ar[ddddddll]_{\id_{[A^m,A^r] \times [A^k,A^s]} \mysquare +_{A^n,A^m,A^l,A^k} \mysquare \id_{A^{n+l}}\qquad\quad} \ar[ddd]^{\id_{[A^m,A^r] \times [A^k,A^s]} \mysquare \boxtimes_{[A^n,A^m],[A^l,A^k],A^n,A^l}} \ar[dddrrrr]^{\qquad\qquad\qquad\qquad\qquad\boxtimes_{[A^m,A^r],[A^k,A^s],[A^n,A^m],[A^l,A^k]} \mysquare \id_{A^n \times A^l}} & & & & \\
					\\
					\\
					[A^{m+k},A^{r+s}] \mysquare [A^{n+l},A^{m+k}] \mysquare A^{n+l} \ar[ddddddddd]_{\circ_{A^{n+l},A^{m+k},A^{r+s}} \id_{A^{n+l}}} \ar[ddddddrr]^{\id_{[A^{m+k},A^{r+s}]} \mysquare \epsilon^{A^{n+l}}_{A^{m+k}}} & & & & & & ([A^m,A^r] \times [A^k,A^s]) \mysquare [([A^n,A^m] \mysquare A^n) \times ([A^l,A^k] \mysquare A^l)] \ar[dddddd]^{\id_{[A^m,A^r] \times [A^k,A^s]} \mysquare \epsilon^{A^n}_{A^m} \times \epsilon^{A^l}_{A^k}} \ar[dddrr]^{\qquad\qquad\qquad\boxtimes_{[A^m,A^r],[A^k,A^s],[A^n,A^m] \mysquare A^n,[A^l,A^k] \mysquare A^l}} & & & & [([A^m,A^r] \mysquare [A^n,A^m] ) \times ([A^k,A^s] \mysquare [A^l,A^k])] \mysquare (A^n \times A^l) \ar[dddll]_{\boxtimes_{[A^m,A^r] \mysquare [A^n,A^m], [A^k,A^s] \mysquare [A^l,A^k], A^n, A^l)}\qquad\qquad\quad\qquad} \ar[dddddd]^{(\circ_{A^n,A^m,A^r} \times \circ_{A^l,A^k,A^s}) \mysquare (\id_{A^n} \times \id_{A^l})} \\
					\\
					\\
					& & & & ([A^m,A^r] \times [A^k,A^s]) \mysquare [A^{n+l},A^{m+k}] \mysquare A^{n+l} \ar[uuullll]_{\qquad\qquad\qquad+_{A^m,A^r,A^k,A^s} \mysquare \id_{[A^{n+l},A^{m+k}]} \mysquare \id_{A^{n+l}}} \ar[dddrr]^{\qquad\qquad\id_{[A^m,A^r] \times [A^k,A^s]} \mysquare \epsilon^{A^{n+l}}_{A^{m+k}}} & & & & ([A^m,A^r] \mysquare [A^n,A^m] \mysquare A^n) \times ([A^k,A^s] \mysquare [A^l,A^k] \mysquare A^l) \ar[ddddddll]^{\qquad(\id_{[A^m,A^r]} \mysquare \epsilon^{A^n}_{A^m}) \times (\id_{[A^k,A^s]} \mysquare \epsilon^{A^l}_{A^k})} \ar[ddddddrr]_{(\circ_{A^n,A^m,A^r} \mysquare \id_{A^n}) \times (\circ_{A^l,A^k,A^s} \mysquare \id_{A^l})\qquad\quad} & & \\
					\\
					\\
					& & [A^{m+k},A^{r+s}] \mysquare A^{m+k} \ar[ddddddrrrr]^{\epsilon^{A^{m+k}}_{A^{r+s}}} & & & & ([A^m,A^r] \times [A^k,A^s]) \mysquare (A^m \times A^k) \ar[llll]_{+_{A^m,A^r,A^k,A^s} \mysquare \id_{A^{m+k}}} \ar[ddd]_{\boxtimes_{[A^m,A^r], [A^k,A^s], A^m, A^k}} & & & & ([A^n,A^r] \times [A^l,A^s]) \mysquare (A^n \times A^l) \ar[ddd]^{\boxtimes_{[A^n,A^r], [A^l,A^s], A^n, A^l}} \\
					\\
					\\
					[A^{n+l},A^{r+s}] \mysquare A^{n+l} \ar[dddrrrrrr]_{\epsilon^{A^{n+l}}_{A{r+s}}\qquad\qquad} & & & & & & ([A^m,A^r] \mysquare A^m) \times ([A^k,A^s] \mysquare A^k) \ar[ddd]^{\epsilon^{A^m}_{A^r} \times \epsilon^{A^k}_{A^s}} & & & & ([A^n,A^r] \mysquare A^n) \times ([A^l,A^s] \mysquare A^l) \ar[dddllll]^{\qquad\qquad\epsilon^{A^n}_{A^r} \times \epsilon^{A^l}_{A^s}} \\
					\\
					\\
					& & & & & & A^{r+s} = A^r \times A^s & & & &  
			}}$$
			$$\textbf{Figure 2}$$
		\end{landscape}
		\noindent
		
		We then have the diagram	
		$$\adjustbox{max width=\columnwidth}{\xymatrix{ & & ([A^n,A^n] \times [A^m, A^m]) \mysquare (A^n \times A^m) \ar[rrrr]^{\boxtimes_{[A^n,A^n], [A^m, A^m], A^n, A^m}} & & & & ([A^n, A^n] \mysquare A^n) \times ([A^m, A^m] \mysquare A^m) \ar[ddrr]^{\quad\epsilon^{A^n}_{A^n} \times \epsilon^{A^m}_{A^m}}  & & \\
				\\
				(\one \times \one) \mysquare (A^n \times A^m) \ar[uurr]^{(j_{A^n} \times j_{A^m}) \mysquare (\id_{A^n} \times \id_{A^m})\qquad\qquad} \ar[rrrr]^{\boxtimes_{\one, \one, A^n, A^m}} \ar[ddrrrr]_{p_1 \mysquare \id_{A^n \times A^m}\qquad} & & & & (\one \mysquare A^n) \times (\one \mysquare A^m) \ar[uurr]^{(j_{A^n} \mysquare \id_{A^n}) \times (j_{A^m} \mysquare \id_{A^m})\qquad\qquad} \ar[rrrr]^{\lambda^{\Coll_{\mysquare}}_{A^n} \times \lambda^{\Coll_{\mysquare}}_{A^m}} & & & & A^n \times A^m \\
				\\
				& & & & \one \mysquare (A^n \times A^m) \ar[uurrrr]_{\qquad\qquad\lambda^{\Coll_{\mysquare}}_{A^n \times A^m}} & & & & 
		}}$$
		whose boundary comes from the currying of the diagram which asserts that $+$ preserves identities.  For this diagram, the top left square commutes by the naturality of the interchange morphism $\boxtimes$.  The top right triangle commutes by the definition of $\epsilon$.  The bottom square commutes as it is the inverse of the unit coherence diagram for $\overline{\lambda}$ following from the duoidal structure for $\Coll$.  Note that although the coherence condition in the definition of a duoidal category is presented with respect to $\overline{\lambda}$ and the morphism $\delta: I \odot I \rightarrow I$, the inverse diagram shown here also follows from the fact $\lambda$ is an isomorphism and the collections $\one \times \one \rightarrow \Tone$ and $\one \hookrightarrow \Tone$ are isomorphic.  Hence this square, and therefore the outer diagram, must commute.
		
		Next we consider the diagram in Figure 3 whose boundary is obtained by currying the associativity diagram required of $GEnd(A)$ to be a monoid object in $\Coll\Cat$.  The top pentagon commutes from the fact that $\boxtimes$ is defined via a projection and hence the order in which we project does not change the result.  In the bottom square, the top triangle commutes by the definition of the associator.  The remaining three triangles in this square commute by the functoriality of the cartesian product.
		
		We next consider the diagrams in Figures 4 and 5 whose boundaries are obtained by currying the left and right unit diagrams required of $GEnd(A)$ to be a monoid object in $\Coll\Cat$, respectively.
		
		We first consider Figure 4.  The leftmost square commutes by the functoriality of $\mysquare$.  We shall now, following clockwise from the top of the diagram, check the commutativity of the five regions incident with this leftmost square.  The first square commutes by the functoriality of $\mysquare$.  The pentagon commutes by the definition of $+$.  The next square commutes by the naturality of $\boxtimes$.  The square incident only on an edge is an instance of the sixth commutativity axiom for $\Coll$ to be a duoidal category.  The bottom adjacent triangle commutes by the definition of $\overline{\rho}^{\Coll_{\times}}$ and functoriality of $\mysquare$.  We shall now look at the right end of the diagram.  Starting with the top right-most triangle, we see that this region commutes by the functoriality of the internal hom $[-,-]$.  The adjacent square to its right commutes by the naturality of $\epsilon$.  The adjacent square directly below this one, along the bottom of the diagram (we shall look at the square to its left last), commutes by the naturality of $\rho$.  The square to the left of this one, which shares an edge with it, commutes by the functoriality of $\times$.  The adjacent square above this one also commutes by the functoriality of $\times$.  The triangle to the right of this one commutes by the definition of $\mathcal{O}$ and functoriality of $\times$.  The final region, the square which was previously skipped, commutes by the fact that the two composites which bound it are two factorizations of
		\begin{landscape}
			$$\adjustbox{max width=\columnwidth}{\xymatrix{ & & ([A^n,A^m] \times [A^l,A^k] \times [A^r,A^s]) \mysquare (A^n \times A^l \times A^r) \ar[ddll]^{\quad\qquad\qquad\qquad\boxtimes_{[A^n,A^m] \times [A^l,A^k],[A^r,A^s],A^n \times A^l,A^r}} \ar[ddrr]_{\boxtimes_{[A^n,A^m],[A^l,A^k] \times [A^r,A^s], A^n, A^l \times A^r}\qquad\qquad\quad} & & \\
					\\
					(([A^n,A^m] \times [A^l,A^k]) \mysquare (A^n \times A^l)) \times ([A^r,A^s] \mysquare A^r) \ar[dd]^{\boxtimes_{[A^n,A^m],[A^l,A^k],A^n,A^l} \times \id_{[A^r,A^s] \mysquare A^r}} & & & & ([A^n,A^m] \mysquare A^n) \times (([A^l,A^k] \times [A^r,A^s]) \mysquare (A^l \times A^r)) \ar[dd]_{\id_{[A^n,A^m] \mysquare A^n} \times \boxtimes_{[A^l,A^k], [A^r,A^s],A^l,A^r}} \\
					\\
					(([A^n,A^m] \mysquare A^n) \times ([A^l,A^k] \mysquare A^l)) \times ([A^r,A^s] \mysquare A^r) \ar[rrrr]^{\alpha^{\Coll_{\times}}_{[A^n,A^m] \mysquare A^n,[A^l,A^k] \mysquare A^l,[A^r,A^s] \mysquare A^r}} \ar[ddrr]^{\qquad\qquad\quad(\id_{[A^n,A^m] \mysquare A^n} \times \epsilon^{A^l}_{A^k}) \times \id_{[A^r,A^s] \mysquare A^r}} \ar[dddd]^{(\epsilon^{A^n}_{A^m} \times \epsilon^{A^l}_{A^k}) \times \id_{[A^r,A^s] \mysquare A^r}} & & & & ([A^n,A^m] \mysquare A^n) \times (([A^l,A^k] \mysquare A^l) \times ([A^r,A^s] \mysquare A^r)) \ar[ddll]_{\id_{[A^n,A^m] \mysquare A^n} \times (\epsilon^{A^l}_{A^k} \times \id_{[A^r,A^s] \mysquare A^r})\quad\qquad\qquad} \ar[dddd]_{\id_{[A^n,A^m] \mysquare A^n} \times (\epsilon^{A^l}_{A^k} \times \epsilon^{A^r}_{A^s})} \\
					\\
					& & ([A^n,A^m] \mysquare A^n) \times A^k \times ([A^r,A^s] \mysquare A^r) \ar[ddll]_{\epsilon^{A^n}_{A^m} \times \id_{A^k} \times \id_{[A^r,A^s] \mysquare A^r}\quad\qquad} \ar[ddrr]^{\qquad\qquad\id_{[A^n,A^m] \mysquare A^n} \times \id_{A^k} \times \epsilon^{A^r}_{A^s}} & & \\
					\\
					A^m \times A^k \times ([A^r,A^s] \mysquare A^r) \ar[rr]^{\id_{A^m} \times \id_{A^k} \times \epsilon^{A^r}_{A^s}} & & A^m \times A^k \times A^s & & ([A^n,A^m] \mysquare A^n) \times A^k \times A^s \ar[ll]_{\epsilon^{A^n}_{A^m} \times \id_{A^k} \times \id_{A^s}} \\
			}}$$
			$$\textbf{Figure 3}$$
		\end{landscape}
		\noindent

		\begin{landscape}
			$$\adjustbox{max width=\columnwidth}{\xymatrix{ & & & & ([A^n,A^m] \times [\Tone, \Tone]) \mysquare A^n \ar[rrrrrr]^{+ \mysquare \id} \ar[ldddd]^{\id \mysquare \overline{\rho}} & & & & & & [A^n \times \Tone, A^m \times \Tone] \mysquare A^n \ar[rrrrrr]^{[\overline{\rho},\rho] \mysquare \id} \ar[ddrrrr]^{[\overline{\rho}, \id] \mysquare \id} \ar[ddlll]_{\id \mysquare \overline{\rho}} & & & & & & [A^n, A^m] \mysquare A^n \ar@/^1.5pc/[rrrrdddddd]^{\epsilon} \\
					\\
					& & & & & & & [A^n \times \Tone, A^m \times \Tone] \mysquare (A^n \times \Tone) \ar[ddrrrrrrrrr]^{\epsilon} & & & & & & & [A^n, A^m \times \Tone] \mysquare A^n \ar[uurr]_{\quad[\id,\rho] \mysquare \id} \ar[ddrr]^{\epsilon} \\
					\\
					& & & ([A^n, A^m] \times [\Tone, \Tone]) \mysquare (A^n \times \Tone) \ar[uurrrr]^{+ \mysquare \id} \ar[rrrrr]^{\boxtimes} & & & & & ([A^n, A^m] \mysquare A^n) \times ([\Tone, \Tone] \mysquare \Tone) \ar[rrr]_{\epsilon \times \id} & & & A^m \times ([\Tone, \Tone] \mysquare \Tone) \ar[rrrrr]_{\id \times \epsilon} & & & & & A^m \times \Tone \ar[rrrrdd]^{\rho} \\
					\\
					([A^n,A^m] \times \Tone) \mysquare A^n \ar@/^2pc/[uuuuuurrrr]^{(\id \times \mathcal{O}) \mysquare \id\quad} \ar[rrrr]^{\id \mysquare \overline{\rho}} \ar@/_1.5pc/[ddddrrrrrrrrrr]_{\rho \mysquare \id} & & & & ([A^n,A^m] \times \Tone) \mysquare (A^n \times \Tone) \ar[luu]^{(\id \times \mathcal{O}) \mysquare \id\qquad} \ar[rrrr]^{\boxtimes} \ar[ddddrrrrrr]^{\rho \mysquare \rho} & & & & ([A^n,A^m] \mysquare A^n) \times (\Tone \mysquare \Tone) \ar[uu]^{\id \times (\mathcal{O} \mysquare \id)} \ar[rrrr]^{\epsilon \times \id} \ar[ddrrrr]_{\id \times \mu} & & & & A^m \times (\Tone \mysquare \Tone) \ar[luu]^{\id \times (\mathcal{O} \mysquare \id)\quad} \ar[uurrrr]^{\id \times \mu} & & & & & & & & A^m \\
					\\
					& & & & & & & & & & & & ([A^n,A^m] \mysquare A^n) \times \Tone \ar[lldd]_{\rho} \ar[uuuurrrr]^{\epsilon \times \id} \\
					\\
					& & & & & & & & & & [A^n,A^m] \mysquare A^n \ar@/_2pc/[uuuurrrrrrrrrr]_{\epsilon}
			}}$$
			$$\textbf{Figure 4}$$
			
			
			$$\adjustbox{max width=\columnwidth}{\xymatrix{ & & & & ([\Tone, \Tone] \times [A^n,A^m]) \mysquare A^n \ar[rrrrrr]^{+ \mysquare \id} \ar[ldddd]^{\id \mysquare \overline{\lambda}} & & & & & & [\Tone \times A^n, \Tone \times A^m] \mysquare A^n \ar[rrrrrr]^{[\overline{\lambda},\lambda] \mysquare \id} \ar[ddrrrr]^{[\overline{\lambda}, \id] \mysquare \id} \ar[ddlll]_{\id \mysquare \overline{\lambda}} & & & & & & [A^n, A^m] \mysquare A^n \ar@/^1.5pc/[rrrrdddddd]^{\epsilon} \\
					\\
					& & & & & & & [\Tone \times A^n, \Tone \times A^m] \mysquare (\Tone \times A^n) \ar[ddrrrrrrrrr]^{\epsilon} & & & & & & & [A^n, \Tone \times  A^m] \mysquare A^n \ar[uurr]_{\quad[\id,\lambda] \mysquare \id} \ar[ddrr]^{\epsilon} \\
					\\
					& & & ([\Tone, \Tone] \times [A^n, A^m]) \mysquare (\Tone \times A^n) \ar[uurrrr]^{+ \mysquare \id} \ar[rrrrr]^{\boxtimes} & & & & & ([\Tone, \Tone] \mysquare \Tone) \times ([A^n, A^m] \mysquare A^n) \ar[rrr]_{\id \times \epsilon} & & & ([\Tone, \Tone] \mysquare \Tone) \times A^m \ar[rrrrr]_{\epsilon \times \id} & & & & & \Tone \times A^m \ar[rrrrdd]^{\lambda} \\
					\\
					(\Tone \times [A^n,A^m]) \mysquare A^n \ar@/^2pc/[uuuuuurrrr]^{(\mathcal{O} \times \id) \mysquare \id\quad} \ar[rrrr]^{\id \mysquare \overline{\lambda}} \ar@/_1.5pc/[ddddrrrrrrrrrr]_{\lambda \mysquare \id} & & & & (\Tone \times [A^n,A^m]) \mysquare (\Tone \times A^n) \ar[luu]^{(\mathcal{O} \times \id) \mysquare \id\qquad} \ar[rrrr]^{\boxtimes} \ar[ddddrrrrrr]^{\lambda \mysquare \lambda} & & & & (\Tone \mysquare \Tone) \times ([A^n,A^m] \mysquare A^n) \ar[uu]^{(\mathcal{O} \mysquare \id) \times \id} \ar[rrrr]^{\id \times \epsilon} \ar[ddrrrr]_{\mu \times \id} & & & & (\Tone \mysquare \Tone) \times A^m \ar[luu]^{(\mathcal{O} \mysquare \id) \times \id\quad} \ar[uurrrr]^{\mu \times \id} & & & & & & & & A^m \\
					\\
					& & & & & & & & & & & & \Tone \times ([A^n,A^m] \mysquare A^n) \ar[lldd]_{\lambda} \ar[uuuurrrr]^{\id \times \epsilon} \\
					\\
					& & & & & & & & & & [A^n,A^m] \mysquare A^n \ar@/_2pc/[uuuurrrrrrrrrr]_{\epsilon}
			}}$$
			$$\textbf{Figure 5}$$
		\end{landscape}
		\noindent

        \noindent
        the currying of the following map:
		$$[\rho_{A^n}^{\Coll_{\times}},\id_{A^m \times \Tone}]:[A^n \times \Tone, A^m \times \Tone] \rightarrow [A^n,A^m \times \Tone]$$
		
		We now consider Figure 5.  The explanations of why each of these regions commutes are completely analogous to those for Figure 4.  The only essential differences are either that the content of certain maps lies in a different cartesian factor (i.e. on the left side of an identity map rather than the right) or that some regions are given in terms of the left unitor transformation instead of the right.
		
		We can hence conclude that the endomorphism globular PRO is in fact a globular PRO.
		
	\section{Algebras for a Globular PRO}
		Just as in the classical case, we can define algebras for a globular PRO, without use of $GEnd(A)$, via a sequence of action maps as follows.
		
		\begin{definition}
			An \emph{algebra for a globular PRO} $\mathcal{P}$ is given by a degenerate collection $a:A \rightarrow \Tone$ together with, for all $n,m \in \mathbb{N}$, a series of collection homomorphisms ${\Omega_{n,m} : \mathcal{P}(n,m) \mysquare A^n \rightarrow A^m}$ which each make the following diagrams commute.
		\end{definition}
		$$\adjustbox{max width=\columnwidth}{\xymatrix{[\mathcal{P}(m,r) \mysquare \mathcal{P}(n,m)] \mysquare A^n \ar[rrr]^{\alpha^{\Coll_{\mysquare}}_{\mathcal{P}(m,r), \mathcal{P}(n,m), A^n}} \ar[dd]_{\circ_{n,m,r} \mysquare \id_{A^n}} & & & \mathcal{P}(m,r) \mysquare [\mathcal{P}(n,m) \mysquare A^n] \ar[rrr]^-{\id_{\mathcal{P}(m,r)} \mysquare \Omega_{n,m}} & & &  \mathcal{P}(m,r) \mysquare A^m \ar[dd]^{\Omega_{m,r}} \\
				\\
				\mathcal{P}(n,r) \mysquare A^n \ar[rrrrrr]_{\Omega_{n,r}} & & & & & & A^r}}$$
		
		$$\adjustbox{max width=\columnwidth}{\xymatrix{[\mathcal{P}(n,m) \times \mathcal{P}(r,s)] \mysquare [A^n \times A^r] \ar[rrr]^-{+_{n,m,r,s} \times \id_{n+r}} \ar[dd]_{\boxtimes_{\mathcal{P}(n,m), \mathcal{P}(r,s), A^n, A^r}} & & & \mathcal{P}(n+r,m+s) \mysquare A^{n+r} \ar[dd]^{\Omega_{n+r,m+s}} \\
				& & & \\
				[\mathcal{P}(n,m) \mysquare A^n] \times [\mathcal{P}(r,s) \mysquare A^r] \ar[rrr]_-{\Omega_{n,m} \times \Omega_{r,s}} & & & A^m \times A^s = A^{m+s} }}$$
		
		$$\xymatrix{ & & \mathcal{P}(n,n) \mysquare A^n \ar[ddrr]^{\Omega_{n,n}} & & \\
			& & & & \\
			\one \mysquare A^n \ar[rrrr]_{\lambda^{\Coll_{\mysquare}}_{A^n}} \ar[uurr]^{j_n \mysquare \id_{A^n}} & & & & A^n }$$
		
		We can again express the previous notion of algebras instead as representations of our PRO via currying.
		
		\begin{definition}
			A $P$\emph{-module} for a globular PRO $\mathcal{P}$ is a globular PRO homomorphism ${f:\mathcal{P} \rightarrow GEnd(A)}$ for some collection $a:A \rightarrow \Tone$.  An $\emph{algebra}$ is a $\mathcal{P}$-module such that the collection $a:A \rightarrow \Tone$ is degenerate.  In particular, the arity map factors as $a = [id] \circ !_{A}$.
		\end{definition}
		
		We immediately get the following result regarding induced algebras.
		
		\begin{theorem} \label{PullbackAlg}
			An algebra for a globular PRO $\mathcal{P}$ is an algebra for every globular PRO $\mathcal{Q}$ which maps to $\mathcal{P}$.  In particular, an algebra for $\mathcal{P}$ is an algebra for every globular sub-PRO.
		\end{theorem}
		
		Note that we previously defined a special class of globular PROs, called `weakenable' globular PROs, that had the appropriate structure to give each hom-collection the structure of a strict $\omega$-category.  One reason these PRO's are of particular interest is because the algebras for a weakenable globular PRO also gain the structure of a strict $\omega$-category.  Hence, the algebras for a weakenable globular PRO may be thought of as strict $\omega$-categories with extra structure.
	
	\section{The Free Monoidal and Path Category PROs on a \Coll-graph}
	
	Every category has an underlying graph.  It is obtained by forgetting the composition and identity structure.  Analogously, for every enriched category there is an underlying enriched graph which is obtained by the same process.  In general, these special graphs are defined as follows.
	\begin{definition}
		Given a duoidal category $\mathcal{D}$, a $\mathcal{D}\emph{-graph}$ $\mathbf{G} = (V,E)$ consists of a set of objects $V$, the elements of which are called $\emph{vertices}$, and a family of objects $E$, which we shall call $\emph{edge objects}$, consisting of, for all $X,Y \in V$, an object $G(X,Y)$ in $\mathcal{D}$.
	\end{definition}
	
	\begin{definition}
		A $\mathcal{D} \emph{-graph morphism}$ $H:\mathcal{A} \rightarrow \mathcal{B}$ consists of a function $H:\text{Obj}(\mathcal{A}) \rightarrow \text{Obj}(\mathcal{B})$ together with a family of morphisms $\{H_{X,Y}:A(X,Y) \rightarrow B(H(X),H(Y))\}$ from $\mathcal{D}$ with $X,Y \in \text{Obj}(\mathcal{A})$.
	\end{definition}
	
	Let $\mathbb{N}\mathcal{D}\Graph$ be the full subcategory of $\mathcal{D}\Graph$ consisting of the $\mathcal{D}$-graphs whose object set is $\mathbb{N}$.  If $\mathcal{D}$ has all countable coproducts, a $\mathcal{D}$-graph $\mathbf{G} = (V,E)$ in $\mathbb{N}\mathcal{D}\Graph$ can be seen as a bi-graded object in $\mathcal{D}$ since every object $G(n,m) \in E$ is indexed by a pair of natural numbers.  This fact induces a bi-grading on $E$.  But since $\mathcal{D}$ has all countable coproducts, the coproduct over the objects of $V$ gives a single object in $\mathcal{D}$ doubly graded over $\mathbb{N}$.  This allows us to canonically identify $\mathbb{N}\mathcal{D}\Graph$ with the category consisting of objects in $\mathcal{D}$ equipped with a bi-grading over $\mathbb{N}$ together with the maps which preserve the bi-grading.  We shall here denote this category $\BiGrd\mathcal{D}$.  Note that, given a pair of objects $X,Y \in obj(\mathcal{D})$, these are precisely the morphisms $f: X \rightarrow Y$ which may be written as a two parameter family of $\mathcal{D}$-morphism $\{f_{i,j}:X(i,j) \rightarrow Y(i,j) \}$.
	
	The category $\BiGrd\mathcal{D}$ has a natural monoidal structure given by the functor
	$$\oplus: \BiGrd\mathcal{D} \times \BiGrd\mathcal{D} \rightarrow \BiGrd\mathcal{D}$$
	which maps a pair of bi-graded objects $X$ and $Y$ from $\mathcal{D}$ to the object $X \oplus Y$, which has the following induced grading
	$$(X \oplus Y)(n,m) := \coprod_{\substack{n=i+j \\ m=l+k}}X(i,l) \odot Y(j,k)$$
	where $\odot$ is the second monoidal product in the duoidal category $\mathcal{D}$.  Given two morphisms $f = \{f_{i,j}:X(i,j) \rightarrow Z(i,j) \}$ and $g = \{g_{l,k}:Y(l,k) \rightarrow W(l,k) \}$ in $\BiGrd\mathcal{D}$ we get
	$$f \oplus g = \{(f \oplus g)_{n,m}:(X \oplus Y)(l,k) \rightarrow (Z \oplus W)(l,k) \}$$
	with components given by
	$$(f \oplus g)_{n,m} := \coprod_{\substack{n=i+j \\ m=l+k}}f_{i,l} \odot g_{j,k}$$
	for each $n,m \in \mathbb{N}$.  The monoidal unit for the product $\oplus$ is the object $\mathfrak{I}$ which is given as $\mathfrak{I}(0,0) = U$, where  $U$ is the monoidal unit for the second monoidal structure $\odot$ on $\mathcal{D}$, and $\mathfrak{I}(i,j) = E$, where  $E$ is the initial object defined by the empty coproduct in $\mathcal{D}$, for all other $i,j \in \mathbb{N}$.  Note that $E$ must exist by the requirement that $\mathcal{D}$ have all countable coproducts.  Moreover, the coproduct structure in $\mathcal{D}$ induces a coproduct $X \coprod Y$ in $\BiGrd\mathcal{D}$.  It is defined to be the identity on objects and has the coproduct in $\mathcal{D}$ of edge objects $X(i,j)\coprod Y(i,j)$ as its edge object $(X\coprod Y)(i,j)$.
	
	\begin{definition}
		A \emph{monoidal} $\mathbb{N}\mathcal{D}$\emph{-graph} $(\mathbf{M}, \diamond, \iota)$ is a monoid in the category $\BiGrd\mathcal{D}$, where $\mathbf{M} \in Obj(\BiGrd\mathcal{D})$, the bi-graded $\mathcal{D}$-morphism $\diamond: \mathbf{M} \oplus \mathbf{M} \rightarrow \mathbf{M}$ is the monoidal product, and $\iota: \mathfrak{I} \rightarrow \mathbf{M}$ is the unit bi-graded $\mathbf{D}$-morphism such that $\diamond$ is associative and unital with respect to $\oplus$.
	\end{definition}
	
	Now that we have a notion of monoidal $\mathcal{D}$-graph, it's then natural to ask: given $\mathbf{G} \in \mathbb{N}\mathcal{D}\Graph$, can we construct a free monoidal $\mathcal{D}$-graph $M(\mathbf{G})$ on $\mathbf{G}$?  Fortunately we can.
	
	\begin{definition}
		Given a $\mathcal{D}$-graph $\mathbf{G} \in \mathbb{N}\mathcal{D}\Graph$, the $\emph{free monoidal}$ $\mathbb{N}\mathcal{D}$-graph on $\mathbf{G}$ is the $\mathcal{D}$-graph
		$$M(\mathbf{G}) := \coprod\limits_{n \in \mathbb{N}} \bigoplus\limits_{k=1}^{n} \mathbf{G}$$
		where both $\mathbf{G}$ and $M(\mathbf{G})$ are thought of as objects in $\BiGrd\mathcal{D}$.  The monoidal product for $M(\mathbf{G})$ is given by the canonical functor $\oplus: M(\mathbf{G}) \oplus M(\mathbf{G}) \rightarrow M(\mathbf{G})$ which is closed by construction.  Note that when $n=0$ the product $\bigoplus\limits_{k=1}^{n} \mathbf{G}$ is the monoidal unit in $\BiGrd\mathcal{D}$.  Hence, the unit morphism $\iota_{M(\mathbf{G})}: \mathfrak{I} \rightarrow M(\mathbf{G})$ is the canonical functor which sends the only non-empty summand of $\mathfrak{I}$, $\mathfrak{I}(0,0) = U$, identically to the only non-empty summand of the empty $\oplus$-product $M(\mathbf{G})_{0}(0,0) = U$.  Moreover, $M: \mathbb{N}\mathcal{D}\Graph \rightarrow \Mon\mathbb{N}\mathcal{D}\Graph$ gives a functor by sending a given bi-graded $\mathcal{D}$-morphism $f = \{f_{i,j}:X(i,j) \rightarrow Y(i,j) \}$ to the morphism
		$$M(f) = \{M(f)_{i,j}:X(i,j) \rightarrow Y(i,j) \}$$
		whose components are given by
		$$M(f)_{i,j} = \coprod\limits_{n \in \mathbb{N}} \bigoplus\limits_{k=0}^{n}f_{i,j}$$
		for $i,j \in \mathbb{N}$.
	\end{definition}
	
	It is furthermore clear that the functor $M:\mathbb{N}\mathcal{D}\Graph \rightarrow \Mon\mathbb{N}\mathcal{D}\Graph$ has a right adjoint $W:\Mon\mathbb{N}\mathcal{D}\Graph \rightarrow \mathbb{N}\mathcal{D}\Graph$ that forgets the monoidal product and unit morphisms with which our $\mathcal{D}$-graph $\mathbf{G}$ is equipped.  In the special case where $\mathcal{D}$ is the category $\Coll$, we have the following result.
	
	\begin{theorem} \label{MonoidalFinAndMod}
		The functor $\mathcal{W}:\Mon\mathbb{N}\Coll\Graph \rightarrow \mathbb{N}\Coll\Graph$ which forgets both the monoidal product and unit structures for a given monoidal $\mathbb{N}\Coll$-graph is finitary and monadic over $\mathbb{N}\Coll\Graph$.
	\end{theorem}
	\begin{proof}
		It is immediately clear from construction that the functor $\mathcal{M}:\mathbb{N}\Coll\Graph \rightarrow \Mon\mathbb{N}\Coll\Graph$ is left adjoint to the forgetful functor $\mathcal{W}$.  It is furthermore clear from construction that $\mathcal{M}(\mathbb{N}\Coll\Graph)$ is the category of algebras for the monad $\mathcal{W}(\mathcal{M})$.  Here $\mathcal{M}$ is precisely the free functor dual to the structure forgotten by $\mathcal{W}$.  Hence the comparison functor $K^{\mathcal{W}(\mathcal{M})}:\Mon\mathbb{N}\Coll\Graph \rightarrow \left(\mathbb{N}\Coll\Graph \right)^{\mathcal{W}(\mathcal{M})}$ is an equivalence of categories.  It remains then to show that $\mathcal{W}$ preserves filtered colimits and is hence finitary.  But this is clear from the fact that $\mathcal{W}$ simply forgets the monoidal concatenation operation structure and that the special summand $M(\mathbf{G})(0,0) = E$ has unit structure with respect to this product.  This implies that given a filtered diagram in $\mathbb{N}\Coll\Cat$, any objects or morphisms that become equal in a colimit on that diagram were already made equal at some level in the filtered diagram.  Moreover, given any filtered diagram in $\mathbb{N}\Coll\Cat$ in which any new elements are generated, the components of that object already existed at some level in the diagram on which the colimit is taken.  Hence the preservation of the $\mathbb{N}\Coll\Graph$ structure in the filtered diagram ensures the preservation of the structure in the colimit.  And thus $\mathcal{W}$ preserves filtered colimits and is therefore finitary.
	\end{proof}

	In a similar way we can both create and forget the category structure on a given $\mathcal{D}$-graph as well.  We will here follow the construction as presented by Wolff$\cite{VCatVGraph:Wolff:1973}$.  First of all, the general process of forgetting the composition and identity structure for a generic $\mathcal{D}$-category to get a corresponding $\mathcal{D}$-graph gives a forgetful functor $U:\mathcal{D}\Cat \rightarrow \mathcal{D}\Graph$ which we shall use in the following definition.
	
	\begin{definition}
		Given a $\mathcal{D}$-graph $\textbf{G}$ and a $\mathcal{D}$-category $\mathcal{C}$, a $\mathcal{D} \emph{-diagram of type}$ $\textbf{G}$ $\emph{in}$ $\mathcal{C}$ is a $\mathcal{D}$-graph morphism $\varphi_{\textbf{G}}: \textbf{G} \rightarrow U(\mathcal{C})$ to the underlying $\mathcal{D}$-graph of $\mathcal{C}$.
	\end{definition}
	
	We are specifically interested in $\Coll$-graphs whose vertex set is the natural numbers.  Given such a graph $\mathbf{G}$ we can construct the free globular PRO $P(\mathbf{G})$ on $\mathbf{G}$.  But before describing this construction in detail, we first mention the following alternative free construction on a $\Coll$-graph.
	
	\begin{definition}
		Given a $\Coll$-graph $\mathbf{G}$, the $\emph{free}$ $\Coll$-category $F(\mathbf{G})$ $\emph{generated by}$ $\mathbf{G}$ is constructed as follows.  First set Obj$(F(\mathbf{G})) = $ Obj$(\mathbf{G})$.  Then take $X,Y \in$ Obj$(\mathbf{G})$.  If $X \neq Y$ we define the hom-object
		$$F(G)(X,Y) := \coprod G(E_0,E_1) \mysquare G(E_1,E_2) \mysquare ... \mysquare G(E_{n-1},E_n)$$
		where the coproduct is taken over all finite sequences $(E_0 = X, E_1, E_2, ..., E_{n-1}, E_n = Y)$ with $E_i \in$ Obj$(\mathbf{G})$ for $i \in \{0, 1, 2, ..., n-1, n\}$ and $n \geq 1$.  If $X = Y$ then we define the hom-object
		$$F(G)(X,X) := [\coprod G(X,E_1) \mysquare G(E_1,E_2) \mysquare ... \mysquare G(E_{n-1},X)] \coprod I$$
		to account for the fact that this hom-object should have enough structure to include identities.  The composition map
		$$\circ_{X,Y,Z}:F(G)(Y,Z) \mysquare F(G)(X,Y) \rightarrow F(G)(X,Z)$$
		is then defined in each coproduct summand by concatenating (via the operation $\mysquare$) strings of hom-objects from the corresponding summands.  More explicitly, if we have that both ${\tau_1 = (A, E_1,...,E_{n-1}, B)}$ and $\tau_2 = (B,D_1,...,D_{n-1},C)$ are strings of objects, then if we define
		$$\tau_1 \bullet \tau_2 := (A,E_1,...,E_{n-1}, B,D_1,...,D_{n-1},C)$$
		and suppose that both $A \neq B$ and $B \neq C$, then we can define $\circ_{A,B,C}$ to be the collection homomorphism which satisfies the equation
		$$\circ_{A,B,C}(\iota_{\tau_{1}} \mysquare \iota_{\tau_{2}}) = \iota_{\tau_1 \bullet \tau_2}(\alpha^k)$$
		where $\alpha^k$ is enough copies of the associator so that the source is completely left parenthesized and $\iota_{\tau_i}$ is the canonical inclusion into the coproduct summand corresponding to the sequence $\tau_i$.  If $A = B$ then we define $\circ_{A,B,C}$ so that it satisfies
		$$\circ_{A,B,C}(\iota_{\tau_{1}} \mysquare \iota_{\tau_{2}}) = \iota_{\tau_2}(\lambda_{\tau_2})$$
		where $\lambda_{\tau_2}$ is the $\tau_2$ component of the left unitor from $\Coll$.  If $B=C$ then $\circ_{A,B,C}$ is defined to satisfy
		$$\circ_{A,B,C}(\iota_{\tau_{1}} \mysquare \iota_{\tau_{2}}) = \iota_{\tau_1}(\rho_{\tau_1})$$
		with $\rho_{\tau_1}$ being the $\tau_1$ component of the right unitor from $\Coll$.  And we define $\circ_{A,B,C}$ to satisfy
		$$\circ_{A,B,C}(\iota_{\tau_{1}} \mysquare \iota_{\tau_{2}}) = \rho_I = \lambda_I$$
		if $A=B=C$.  The identity identifications $j_A: I \rightarrow F(G)(X,X)$ are defined to be the canonical inclusion map into the $I$ summand of the corresponding coproduct.
	\end{definition}
	
	This construction extends to maps of $\Coll$-graphs in the obvious way to give a functor $F: \Coll\Graph \rightarrow \Coll\Cat$.  We can now generate free $\Coll$-categories by listing certain generating hom-objects at the graph level.
	
	\begin{theorem} \label{UnderCatFinAndMod}
		The functor $U:\mathbb{N}\Coll\Cat \rightarrow \mathbb{N}\Coll\Graph$ which sends any $\mathbb{N}\Coll\Cat$ to its underlying $\mathbb{N}\Coll\Graph$ is finitary and monadic over $\mathbb{N}\Coll\Graph$.
	\end{theorem}
	\begin{proof}
		It is clear from construction both that the functor $F:\mathbb{N}\Coll\Graph \rightarrow \mathbb{N}\Coll\Cat$ is left adjoint to $U:\mathbb{N}\Coll\Cat \rightarrow \mathbb{N}\Coll\Graph$ and that $F(\mathbb{N}\Coll\Graph)$ is precisely the category of algebras for the monad $U(F):\mathbb{N}\Coll\Graph \rightarrow \mathbb{N}\Coll\Graph$.  In other words, $F$ is precisely the free functor dual to the structure forgotten by $U$.  Hence the comparison functor $K^{U(F)}:\mathbb{N}\Coll\Cat \rightarrow \left(\mathbb{N}\Coll\Graph \right)^{U(F)}$ is an equivalence of categories.  It remains then to show that $U$ preserves filtered colimits and is hence finitary.  But this is clear from the fact that $U$ simply forgets the concatenation operation and that certain hom-objects have unit structures with respect to this concatenation.  This implies that given a filtered diagram in $\mathbb{N}\Coll\Cat$, any objects or morphisms that become equal in a colimit on that diagram were already made equal at some level in the filtered diagram.  Moreover, given any filtered diagram in $\mathbb{N}\Coll\Cat$ in which any new elements are generated, the components of that element already existed in some previous object at some level in the diagram over which the colimit is taken.  Hence the preservation of the $\mathbb{N}\Coll\Graph$ structure in the filtered diagram ensures the preservation of the structure in the colimit.  And thus $U$ preserves filtered colimits and is therefore finitary.
	\end{proof}

	\section{PRO Globularization}
	It is well know $\cite{BasicConEnCatTheory:Kelly:2005}$ that the functor which takes an enriched category to its underlying ordinary category has a left adjoint which generates the enriched structure.  This is done by taking copowers of the monoidal unit from the category over which the enrichment is taking place.  We will here perform a similar construction which constructs, from a classical PRO $P$, a globular PRO $\mathcal{P}$ whose algebras are precisely the algebras for $P$ in $\Glob$ which have an $\omega$-category structure with operations given by strict $\omega$-functors.  This is done by taking copowers not of the unit collection $I:\one \hookrightarrow \Tone$, but rather the terminal collection $\id:\Tone \rightarrow \Tone$.  Let $P$ be any ordinary set PRO and consider the following functor $G_P:P \rightarrow \mathcal{P}$ which maps $P$ to its globularization.
	$$n \mapsto n$$
	$$P(n,m) \mapsto \mathcal{P}(n,m) := P(n,m) \cdot \id = \coprod_{P(n,m)}\id$$
	Furthermore, the operations $\circ$ and $+$ are induced by the structure in $P$.
	
	Note first that, for all $n,m,p \in \mathbb{N}$ the hom-object $\mathcal{P}(m,p) \mysquare \mathcal{P}(n,m)$ can be written
	$$\mathcal{P}(m,p) \mysquare \mathcal{P}(n,m) = \left(\coprod_{P(m,p)}\id \right) \mysquare \left(\coprod_{P(n,m)}\id\right) = \coprod_{P(m,p)}\left( \id \mysquare \left(\coprod_{P(n,m)}\id\right) \right) =$$
	$$\coprod_{P(m,p)}\left( \coprod_{P(n,m)} \left( \id \mysquare \id \right) \right) \cong \coprod_{P(m,p) \times P(n,m)}\left( \id \mysquare \id \right)$$
	where the final isomorphism is simply a reindexing of the double coproduct by a single coproduct of pairs.  Note that this is possible because $\mysquare$ distributes over coproduct in both variables.  It distributes over the first variable because $-\mysquare \mathcal{A}$ is left adjoint to $[\mathcal{A},-]$, and hence preserves colimits.  To see why it distributes over the second variable, we consider the geometric realization functor on $\Coll$.  Each globular cell has a connected geometric realization.  And because coproducts in a topos are disjoint, and $\Coll$ is itself a topos, for an element from a collections $\mathcal{A}\mysquare\mathcal{B}$, all of the globular cells from $\mathcal{B}$ which color the pasting diagram which names the arity of the cell from $\mathcal{A}$ must come from the same summand in any coproduct, or else their geometric realization would be disconnected.  It therefore follows that $\mysquare$ distributes over coproduct in the second variable.
 
    Given the formulation of $\mathcal{P}(m,p) \mysquare \mathcal{P}(n,m)$ above, the induced composition operations on $\mathcal{P}$, for all $n,m,p \in \mathbb{N}$ are then given by
	$$\circ^{\mathcal{P}}_{n,m,p} := \circ^P_{n,m,p} \cdot \phi:(P(m,p) \times P(n,m)) \cdot (\id \mysquare \id) \rightarrow P(n,p) \cdot \id$$
	where $\circ^P_{n,m,p}$ is composition in $P$ and $\phi: \id \mysquare \id \rightarrow \id$ is the morphism in $\Coll$ ensuring that $\id$ is a monoid object with respect to the product $\mysquare$.
	
	We can similarly write $\mathcal{P}(n,m) \times \mathcal{P}(l,k)$
	$$\mathcal{P}(n,m) \times \mathcal{P}(l,k) = \left(\coprod_{P(n,m)}\id \right) \times \left(\coprod_{P(l,k)}\id\right) = \coprod_{P(n,m)}\left( \id \times \left(\coprod_{P(l,k)}\id\right) \right) =$$
	$$\coprod_{P(n,m)}\left( \coprod_{P(l,k)} \left( \id \times \id \right) \right) \cong \coprod_{P(n,m) \times P(l,k)}\left( \id \times \id \right)$$
	for all $n,m,l,k \in \mathbb{N}$.  The induced addition operations on $\mathcal{P}$, for all $n,m,l,k \in \mathbb{N}$ are then given by
	$$+^{\mathcal{P}}_{n,m,l,k} := +^P_{n,m,l,k} \cdot \Phi:(P(n,m) \times P(l,k)) \cdot (\id \times \id) \rightarrow P(n+l,m+k) \cdot \id$$
	where $+^P_{n,m,l,k}$ is addition in $P$ and $\Phi$ is the canonical isomorphism which is described by the left (equivalently right) cartesian unitor.
	
	The identity identifications $j_n: \one \rightarrow \mathcal{P}(n,n)$ are induced by the composition
	$$I \xrightarrow{\xi_n} P(n,n) \cdot I \hookrightarrow P(n,n) \cdot \id$$
	$$\sigma_n \mapsto (\iota(*), \sigma_n) \hookrightarrow (\iota_n(*), \sigma_n)$$
	for each $n \in \mathbb{N}$, where $\iota_n: \{*\} \rightarrow P(n,n)$ is the identity identification from the underlying set PRO $P$.

    We reiterate here why we prefer the usage of globular PRO instead of globular Lawvere theory.  Note that nothing in this construction requires that the PRO $P$ being globularlized be the special case of a Lawvere theory.  Any classical PRO $P$ can be globularized.
	
	\begin{theorem}
		The globularization $\mathcal{P}$ of a PRO $P$ is a globular PRO.
	\end{theorem}
	
	\begin{proof}
		It is clear from construction that $\mathcal{P}$ is a cartesian-duoidal enriched category enriched over the cartesian-duoidal category $\Coll$ with object set $\mathbb{N}$.  It is furthermore clear from construction that $+^{\mathcal{P}}$ is simply addition at the level of objects.  We need then that $\mathcal{P}$ is a monoid $(\mathcal{P},+^{\mathcal{P}},\mathcal{I})$ in $(\Coll\Cat, \times, \one_{*})$, where $\one_{*}$ is the terminal $\Coll$-cat and $\mathcal{I}: \one_{*} \rightarrow \mathcal{P}$ is the $\Coll$-functor which maps the single object $* \in \id_{*}$ to $0 \in \mathbb{N}$ and the unique hom-object $\one_{*}(*,*)$ to $\mathcal{P}(0,0) = \id$, the unit for $\times$ in $\Coll\Cat$.  But this follows immediately from the fact that each of the relevant commutative diagrams was satisfied in the original non-globularized PRO.  As this structure is faithfully preserved by the indexing on each hom-object, the induced operations on the globularlized PRO satisfy the analogous commutativity conditions which ensure that $\mathcal{P}$ is a globular PRO as well.  Finally, the commutativity of the appropriate diagrams required of $\mathcal{P}$ in order for it to be a globular PRO follow immediately from construction.  Therefore $\mathcal{P}$ is a globular PRO.
	\end{proof}
	
	This leads us to the following key theorem of this work.  Among the consequences of this theorem is the result that the globularization of an ordinary PRO is canonically a weakenable globular PRO by construction.  This is of utmost importance because, as noted earlier, weakenable PROs impose upon each hom-object the structure of a strict $\omega$-category.  And as the content of the following theorem shows that the operations in the algebras for a globularized PRO are given by strict $\omega$-functors, when we eventually weaken our weakenable PROs through the use of contractions, we will get algebras for our weakenings that have operations given by weak $\omega$ functors in the sense described by Leinster in $\cite{leinster2000operads}$.

	\begin{theorem}
		Let $\mathcal{P}$ be the globularization of the ordinary PRO $P$.  The algebras for $\mathcal{P}$ are exactly the strict $\omega$-categories which are algebras for $P$ whose operations in $P$ are given by strict $\omega$-functors.
	\end{theorem} 
	
	\begin{proof}
		Let $A$ be an algebra for the globular PRO $\mathcal{P}$.  Consider the hom-object $\mathcal{P}(1,1)$ which acts on $A$ via the action map $\omega: P(1,1) \cdot \Tone \mysquare A \rightarrow A$.  Note that the component of $\omega$ corresponding to the identity in $P$ gives a map of globular sets $\omega_{\id_P}: \Tone \mysquare A \rightarrow A$ which encodes the structure of a strict $\omega$-category on the globular set $A$ (as an algebra for the terminal collection).  To see that $A$ is moreover an algebra for the set PRO $P$, consider that the action map $\Omega_{n,m}: \mathcal{P}(n,m) \mysquare A^n \rightarrow A^m$ may be restricted so that the globular pasting portion of the action only acts by the image of the inclusion of generators $\one \hookrightarrow \Tone$.  This restricted map is precisely an action of the indexing set for the globular operations (i.e. an induced set $P(n,m)$) on the set $A$.  Furthermore, collectively these maps, for all $n,m \in \mathbb{N}$, satisfy the appropriate diagrams to induce the structure of a $P$-algebra on $A$.  It remains to show that the action of operations in $\mathcal{P}$ act on $A$ by strict $\omega$-functors.  This means that two components of an action (the cartesian portion taking place in the indexing set PRO and the globular pasting portion) can be applied in either order.  But this follows immediately from the fact that the action map can be factored so that either operation may be performed first together with the fact that each pair in the source $\mathcal{P}(n,m) \mysquare A^n$ maps to a particular cell in $A^m$ under $\Omega_{n,m}$.  Hence both of these factorization show that the operations in $\mathcal{P}$ act on $A$ by strict $\omega$-functors.
		
		Conversely, assume that $A$ is an algebra in $\Glob$ for the set PRO $P$ which has the structure of a strict $\omega$-category and whose operations in $P$ are given by strict $\omega$-functors.  We wish to show that it is also an algebra for $\mathcal{P}$.  Since $A$ is a strict $\omega$-category it admits the structure of an algebra for the terminal collection $\id$.  Hence there exists an action map $\omega: \Tone \mysquare A \rightarrow A$ where $A$ is here the collection equipped with arity map $[id] \circ !_A: A \rightarrow \Tone$.  Since $A$ is an algebra for $P$ it also admits a map to the ordinary endomorphism PRO on $A$.  This means that for each $n,m \in \mathbb{N}$ we have a map $P(n,m) \rightarrow \Glob(A^n,A^m)$, each of which can be curried to get maps $\nu_{n,m}: P(n,m) \cdot A^n \rightarrow A^m$.  Note then that using the identities discussed above we can construct an induced action map $\Omega_{n,m}: \mathcal{P}(n,m) \mysquare A^n \rightarrow A^m$ by first rewriting the domain as
		$$\mathcal{P}(n,m) \mysquare A^n = P(n,m) \cdot \id = \left(\coprod_{P(n,m)}\Tone \right) \mysquare A^n \cong \coprod_{P(n,m)}\left(\Tone \mysquare A^n\right) \cong \coprod_{P(n,m)}\left(\Tone \mysquare A\right)^n$$
		and letting $\Omega_{n,m}$ be defined as the composition
		$$\Omega_{n,m} := \nu_{n,m} \circ (\id_{P(n,m)} \cdot \omega^n)$$
		where $\omega^n: (\Tone \mysquare A)^n \rightarrow A^n$ is simply the $n$th cartesian power of $\omega$.  All that remains to be shown is that the diagrams
		$$\adjustbox{max width=\columnwidth}{\xymatrix{[\mathcal{P}(m,l) \mysquare \mathcal{P}(n,m)] \mysquare A^n \ar[rrr]^{\alpha^{\Coll_{\mysquare}}_{\mathcal{P}(m,l), \mathcal{P}(n,m), A^n}} \ar[dd]_{\circ_{n,m,l} \mysquare \id_{A^n}} & & & \mathcal{P}(m,l) \mysquare [\mathcal{P}(n,m) \mysquare A^n] \ar[rrr]^-{\id_{\mathcal{P}(m,l)} \mysquare \Omega_{n,m}} & & &  \mathcal{P}(m,l) \mysquare A^m \ar[dd]^{\Omega_{m,l}} \\
				\\
				\mathcal{P}(n,l) \mysquare A^n \ar[rrrrrr]_{\Omega_{n,l}} & & & & & & A^l}}$$
		
		$$\adjustbox{max width=\columnwidth}{\xymatrix{[\mathcal{P}(n,m) \times \mathcal{P}(l,k)] \mysquare [A^n \times A^l] \ar[rrr]^-{\boxtimes_{\mathcal{P}(n,m), \mathcal{P}(l,k), A^n, A^l}} \ar[dd]_{+_{n,m,l,k} \times \id_{n+l}} & & & [\mathcal{P}(n,m) \mysquare A^n] \times [\mathcal{P}(l,k) \mysquare A^l] \ar[dd]^{\Omega_{n,m} \times \Omega_{l,k}} \\
				& & & \\
				\mathcal{P}(n+l,m+k) \mysquare A^{n+l} \ar[rrr]_-{\Omega_{n+l,m+k}} & & & A^{m+k} = A^m \times A^k }}$$	
		
		$$\xymatrix{ & & \mathcal{P}(n,n) \mysquare A^n \ar[ddrr]^{\Omega_{n,n}} & & \\
			& & & & \\
			\one \mysquare A^n \ar[rrrr]_{\lambda^{\Coll_{\mysquare}}_{A^n}} \ar[uurr]^{j_n \mysquare \id_{A^n}} & & & & A^n }$$
		commute for all $n,m,l,k \in \mathbb{N}$.  When unpacking these diagrams explicitly via the definitions provided above for the relevant maps, the first two unfortunately become quite large.  This makes it impractical to attempt typesetting the complete diagrams all at once.  Instead, in order to show that these three diagrams commute, a schematic has been provided below for the complete diagrams with subsections of the center faces cut out and labeled.  Explicit versions of each of these subsections can then be found below, together with an explanation of why this subsection commutes.  The third diagram is small enough to be shown explicitly in a single diagram and follows the first two.  Note also that all unlabeled edges correspond to either a reindexing operation or a sequence of instances of unitors and interchange morphisms (here used to include a $\mysquare$ product with a cartesian power of a collection in the second variable into a cartesian power of $\mysquare$ products).
		\newpage
		
		\textbf{COMPOSITION PRESERVES ACTION}
		
		$$\adjustbox{max width=\columnwidth}{\xymatrix{ & & & & & & & & & & \cdot \ar[rr]^{\id_{P(m,l) \cdot \Tone} \mysquare [\id_{P(n,m)} \cdot \omega^n]} & & \cdot \ar[rrrr]^{\quad\id_{P(m,l) \cdot \Tone} \mysquare \nu_{n,m}} \ar[ddrr] & & & & \cdot \ar[rrrr] & & & & \cdot \ar[rrrrdddd] & & & & & & & & \\
				\\
				& & & & & & & & \cdot \ar[uurr] \ar[dd] & & & & & & \cdot \ar[ddrr] & & & & & & & & & & & & & & \\
				\\
				& & & & & & & & \cdot \ar[dd] & & & & \textbf{B} & & & & \cdot \ar[ddrr] & & & & \textbf{C} & & & & \cdot \ar[ddddrrrr]^{\id_{P(m,l) \cdot \omega^m}} & & & & \\
				\\
				& & & & \cdot \ar[uuuurrrr] & & & & \cdot \ar[dd] & & & & & & & & & & \cdot \ar[ddrr] & & & & & & & & & & \\
				\\
				& & & & & \textbf{A} & & & \cdot \ar[ddrrrr] & & & & & & & & & & & & \cdot \ar[rrrrrrrr] \ar[ddrrrr] & & & & & & & & \cdot \ar[dd]_{\nu_{m,l}} \\
				& & & & & & & & & & & & & & & & & & & & & & & & & & \textbf{F} & & \\
				\circ \ar[rr] \ar[rrrruuuu]^{\qquad\qquad\alpha^{\Coll_{\mysquare}}_{P(m,l) \cdot \Tone, P(n,m) \cdot \Tone, A^n}} \ar[dddd] & & \cdot \ar[rr] & & \cdot \ar[rr] & & \cdot \ar[rr] & & \cdot \ar[uu] \ar[ddrrrr] & & & & \cdot \ar[rrrr] & & & & \cdot \ar[rrrr] & & & & \cdot \ar[rrrr] & & & & \cdot \ar[ddrrrr] & & & & \circ \\
				& \textbf{D} & & & & & & & & & & & & & & & & & & \textbf{E} & & & & & & & & & \\
				& & & & & & & & \textbf{G} & & & & \cdot \ar[rrr] \ar[rrrrrdd] & & & \cdot \ar[rrr] & & & \cdot \ar[rrr] & & & \cdot \ar[rrr] & & & \cdot \ar[rrrr] & & & & \cdot \ar[uu] \\
				& & & & & & & & & & & & & & & & & & & \textbf{H} & & & & & & & & & \\
				\cdot \ar[rrrrr]_{[\circ^P_{n,m,l} \cdot \id_{\Tone \mysquare \Tone}] \mysquare \id_{A^n}\qquad} \ar[rrrrrrrruuuu] & & & & & \cdot \ar[rrrrrr]_{[\id_{P(n,l)} \cdot \phi] \mysquare \id_{A^n}} & & & & & & \cdot \ar[rrrrrr] & & & & & & \cdot \ar[rrrrrr] & & & & & & \cdot \ar[rrrrruu]_{\quad\quad\id_{P(n,l)} \cdot \omega^n\qquad} & & & & &
		}}$$
		
		\newpage
		\textbf{DIAGRAM A}
		
		$$\adjustbox{max width=\columnwidth}{\xymatrix{ & [(P(m,l) \cdot \Tone) \mysquare (P(n,m) \cdot \Tone)] \mysquare A^n \ar[ddl] \ar[ddr]^{\qquad\qquad\alpha^{\Coll_{\mysquare}}_{P(m,l) \cdot \Tone, P(n,m) \cdot \Tone, A^n}} & \\
				\\
				[P(m,l) \cdot [\Tone \mysquare (P(n,m) \cdot \Tone)]] \mysquare A^n \ar[dd] & & (P(m,l) \cdot \Tone) \mysquare [(P(n,m) \cdot \Tone) \mysquare A^n] \ar[dd] \\
				\\
				[P(m,l) \cdot [P(n,m) \cdot (\Tone \mysquare \Tone)]] \mysquare A^n \ar[dd] & & (P(m,l) \cdot \Tone) \mysquare [P(n,m) \cdot (\Tone \mysquare A^n)] \ar[dd] \\
				\\
				[(P(m,l) \times P(n,m)) \cdot (\Tone \mysquare \Tone)] \mysquare A^n \ar[dd] & & P(m,l) \cdot [\Tone \mysquare [P(n,m) \cdot (\Tone \mysquare A^n)]] \ar[dd] \\
				\\
				(P(m,l) \times P(n,m)) \cdot [(\Tone \mysquare \Tone) \mysquare A^n] \ar[ddr]^{\qquad\qquad\id_{P(m,l) \times P(n,m)} \cdot \alpha^{\Coll_{\mysquare}}_{\Tone, \Tone, A^n}} & & P(m,l) \cdot [P(n,m) \cdot [\Tone \mysquare (\Tone \mysquare A^n)]] \ar[ddl] \\
				\\
				& (P(m,l) \times P(n,m)) \cdot [\Tone \mysquare (\Tone \mysquare A^n)] & 
		}}$$
		
		We shall show the commutativity of this diagram by describing how each edge of this diagram acts on a generic element.  Let $\sigma_{\phi} \in P(m,l) \cdot \Tone$ be a cell of shape $\sigma \in \Tone$ indexed by an operation $\phi \in P(m,l)$.  Then let $\Sigma \in \mathcal{T}(P(n,m) \cdot \Tone)$ be a coloring of $\sigma_{\phi}$ by composite cells ${}^{\tau}\Sigma_{\psi} \in P(n,m)\cdot\Tone$, one for each sub-cell $\tau \in \sigma_{\phi}$.  Note that each composite cell may be indexed by the same $\psi$ because of the connectedness of $\Tone$.  Moreover, let $\kappa_n$ be a coloring of the shape of $(\sigma_{\phi},\Sigma)$.  Thus we start both compositions with a cell $((\sigma_{\phi},\Sigma),\kappa_n) \in [(P(m,l)\cdot\Tone) \mysquare (P(n,m)\cdot\Tone)] \mysquare A^n$.
		
		We begin the first composition by applying the associator for $\mysquare$ to $((\sigma_{\phi},\Sigma),\kappa_n)$ to get a coloring $\int$ of $\sigma_{\phi}$, where $\int$ is induced by the coloring of $\Sigma$ by $\kappa_n$.  Hence, for each sub-cell $\tau \in \sigma_{\phi}$, the composite cell of $\int$ which colors it is $({}^{\tau}\Sigma_{\psi},\kappa_n) \in [P(n,m)\cdot\Tone]\mysquare A^n$.  Since the colored cells of $\int$ all come from the same summand, we may send $({}^{\tau}\Sigma_{\psi},\kappa_n)$ to ${({}^{\tau}\Sigma,\kappa_n)_{\psi} \in P(n,m)\cdot[\Tone\mysquare A^n]}$.  Similarly, the sub-cells of $\sigma_{\phi}$ being colored do not rely on the summand denoted by $\phi$ to be colored.  Hence, this and the previous step together send $(\sigma_{\phi},(\Sigma_{\psi},\kappa_n))$ to $(\sigma,(\Sigma,\kappa_n)_{\psi})_{\phi}$.  Citing this independence from the summand index a third time gives $((\sigma,(\Sigma,\kappa_n))_{\psi})_{\phi}$ which can be re-indexed by a single operation ${(\psi,\phi) \in P(m,l) \times P(n,m)}$ to get $(\sigma,(\Sigma,\kappa_n))_{(\psi,\phi)} \in [P(m,l) \times P(n,m)]\cdot[\Tone \mysquare (\Tone \mysquare A^n)]$.
		
		Along the other composition, we first re-index $((\sigma_{\phi},\Sigma),\kappa_n)$ to get $((\sigma,\Sigma)_{\phi},\kappa_n)$.  Again by the connectedness of $\Tone$ we can re-index to get $(((\sigma,\Sigma)_{\psi})_{\phi},\kappa_n)$.  Reindexing further gives $((\sigma,\Sigma)_{(\psi,\phi)},\kappa_n)$; hence ${((\sigma,\Sigma),\kappa_n)_{(\psi,\phi)} \in [P(m,l) \times P(n,m)]\cdot[(\Tone \mysquare \Tone) \mysquare A^n]}$.  Then applying the associator within this single summand corresponding to $(\psi,\phi) \in P(m,l) \times P(n,m)$ must give the same cell $(\sigma,(\Sigma,\kappa_n))_{(\psi,\phi)}$ in \newline $[P(m,l) \times P(n,m)]\cdot[\Tone \mysquare (\Tone \mysquare A^n)]$ from above.
		\\

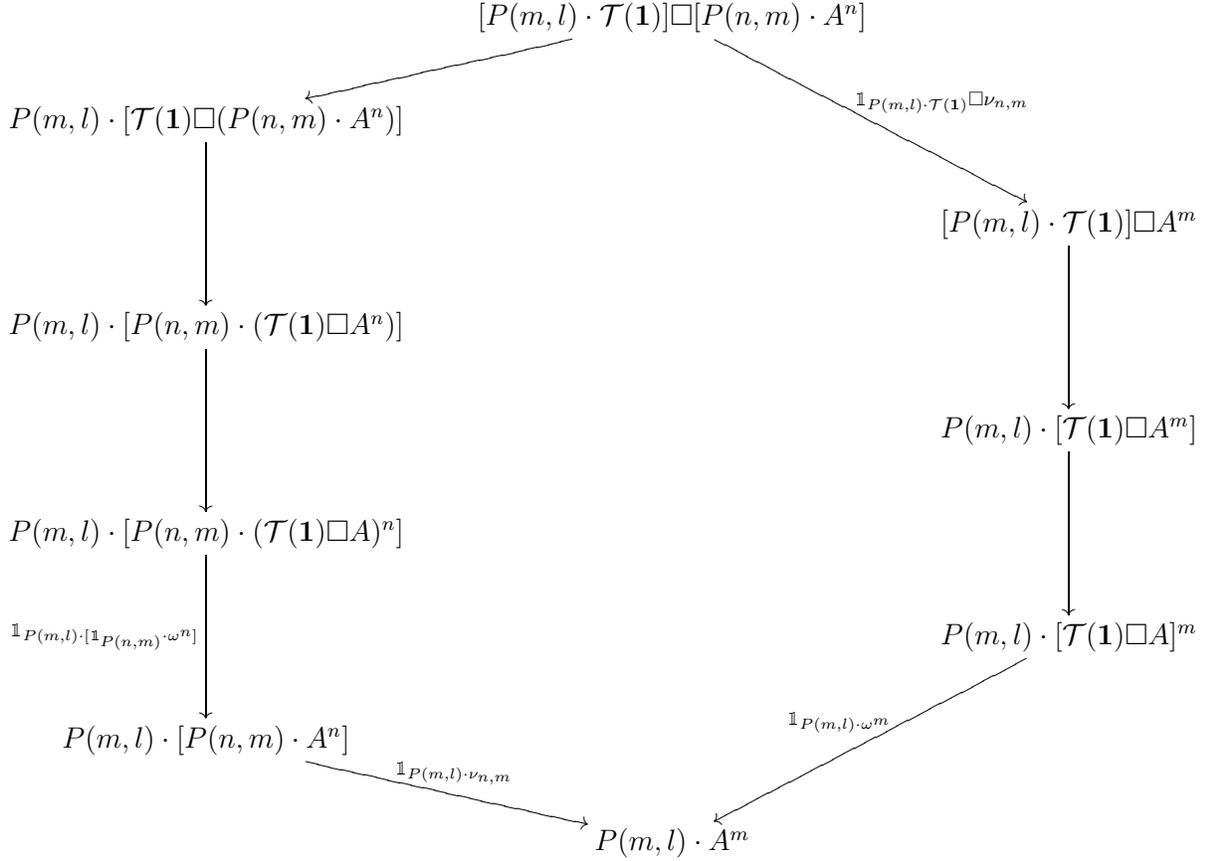
		\textbf{DIAGRAM B}
		
		$$\adjustbox{max width=\columnwidth}{\xymatrix{ & [P(m,l) \cdot \Tone] \mysquare [P(n,m) \cdot (\Tone \mysquare A^n)] \ar[dl] \ar[ddr] & \\
				P(m,l) \cdot [\Tone \mysquare [P(n,m) \cdot (\Tone \mysquare A^n)]] \ar[dd] & & \\
				& & [P(m,l) \cdot \Tone] \mysquare [P(n,m) \cdot (\Tone \mysquare A)^n] \ar[dd]^{\id_{P(m,l) \cdot \Tone} \mysquare [\id_{P(n,m)} \cdot \omega^n]} \ar[dl] \\
				P(m,l) \cdot [P(n,m) \cdot [\Tone \mysquare (\Tone \mysquare A^n)]] \ar[dd] & P(m,l) \cdot [\Tone \mysquare [P(n,m) \cdot (\Tone \mysquare A)^n]] \ar[dd] & \\
				& & [P(m,l) \cdot \Tone] \mysquare [P(n,m) \cdot A^n] \ar[dd] \\
				[P(m,l) \times P(n,m)] \cdot [\Tone \mysquare (\Tone \mysquare A^n)] \ar[dd] & P(m,l) \cdot [P(n,m) \cdot [\Tone \mysquare (\Tone \mysquare A)^n] \ar[ddl] & \\
				& & P(m,l) \cdot [\Tone \mysquare (P(n,m) \cdot A^n)] \ar[dd] \\
				[P(m,l) \times P(n,m)] \cdot [\Tone \mysquare (\Tone \mysquare A)^n] \ar[dd]_{\id_{P(m,l) \times P(n,m)} \cdot [\id_{\Tone} \mysquare \omega^n]} & & \\
				& & P(m,l) \cdot [P(n,m) \cdot (\Tone \mysquare A^n)] \ar[dd] \ar[dll] \\
				[P(m,l) \times P(n,m)] \cdot [\Tone \mysquare A^n] \ar[dd] & & \\
				& & P(m,l) \cdot [P(n,m) \cdot (\Tone \mysquare A)^n] \ar[dd]^{\id_{P(m,l)} \cdot [\id_{P(n,m)} \cdot \omega^n]} \ar[dll] \\
				[P(m,l) \times P(n,m)] \cdot [\Tone \mysquare A]^n \ar[ddr]^{\qquad\id_{P(m,l) \times P(n,m)} \cdot \omega^n} & & \\
				& & P(m,l) \cdot [P(n,m) \cdot A^n] \ar[dl] \\
				& [P(m,l) \times P(n,m)] \cdot A^n & \\
		}}$$
		
		Each of the top two regions of this diagram commute by the fact that $\mysquare$ preserves coproduct, and hence $\cdot$ is preserved.  The bottom two squares commute by the naturality of the operation of reindexing copowers.
		
		\newpage
		\textbf{DIAGRAM C}
		
		$$\adjustbox{max width=\columnwidth}{\xymatrix{ & [P(m,l) \cdot \Tone] \mysquare [P(n,m) \cdot A^n] \ar[dl] \ar[ddr]^{\quad\id_{P(m,l) \cdot \Tone} \mysquare \nu_{n,m}} & \\
				P(m,l) \cdot [\Tone \mysquare (P(n,m) \cdot A^n)] \ar[dd] & & \\
				& & [P(m,l) \cdot \Tone] \mysquare A^m \ar[dd] \\
				P(m,l) \cdot [P(n,m) \cdot (\Tone \mysquare A^n)] \ar[dd] & & \\
				& & P(m,l) \cdot [\Tone \mysquare A^m] \ar[dd] \\
				P(m,l) \cdot [P(n,m) \cdot (\Tone \mysquare A)^n] \ar[dd]_{\id_{P(m,l) \cdot [\id_{P(n,m)} \cdot \omega^n]}} & & \\
				& & P(m,l) \cdot [\Tone \mysquare A]^m \ar[ddl]_{\id_{P(m,l) \cdot \omega^m}} \\
				P(m,l) \cdot [P(n,m) \cdot A^n] \ar[dr]^{\id_{P(m,l) \cdot \nu_{n,m}}} & & \\
				& P(m,l) \cdot A^m & \\
		}}$$
		
		This diagram commutes by the fact that the operations in $\mathcal{P}$ act on $A$ as strict $\omega$-functors.  Hence, the $\nu$ and $\omega$ portion of an operation in $\mathcal{P}$ may be performed in either order.
		\\
		
		\textbf{DIAGRAM D}
		
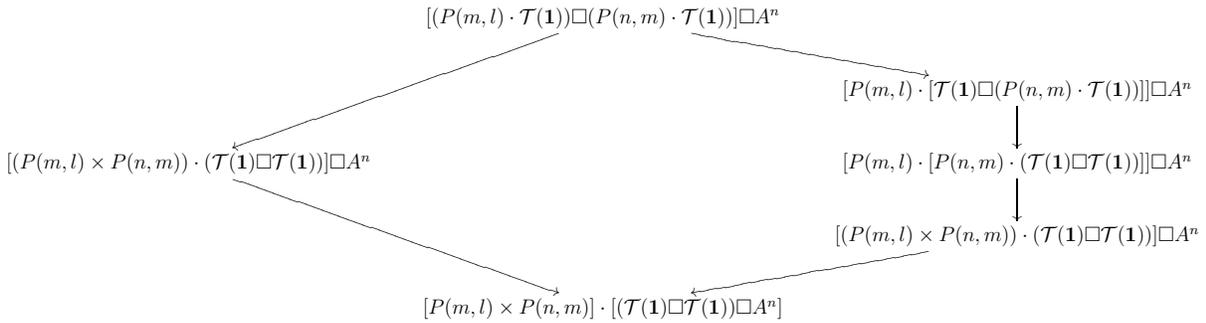
		$$\adjustbox{max width=\columnwidth}{\xymatrix{ & [(P(m,l) \cdot \Tone) \mysquare (P(n,m) \cdot \Tone)] \mysquare A^n \ar[dr] \ar[ddl] & \\
				& & [P(m,l) \cdot [\Tone \mysquare (P(n,m) \cdot \Tone)]] \mysquare A^n \ar[d] \\
				[(P(m,l) \times P(n,m)) \cdot (\Tone \mysquare \Tone)] \mysquare A^n \ar[ddr] & & [P(m,l) \cdot [P(n,m) \cdot (\Tone \mysquare \Tone)]] \mysquare A^n \ar[d] \\
				& & [(P(m,l) \times P(n,m)) \cdot (\Tone \mysquare \Tone)] \mysquare A^n \ar[dl] \\
				& [P(m,l) \times P(n,m)] \cdot [(\Tone \mysquare \Tone) \mysquare A^n] & \\
		}}$$
		
		This diagram commutes by the fact that each of the two sides of the diagram two ways of performing the same copower reindexing.
		
		\newpage
		\textbf{DIAGRAM E}
		
		$$\adjustbox{max width=\columnwidth}{\xymatrix{ & [P(m,l) \times P(n,m)] \cdot [(\Tone \mysquare \Tone) \mysquare A^n] \ar[dl]_{\circ^P_{n,m,l} \cdot \id_{(\Tone \mysquare \Tone) \mysquare A^n}\qquad} \ar[ddr]^{\qquad\qquad\id_{P(m,l) \times P(n,m)} \cdot \alpha^{\Coll_{\mysquare}}_{\Tone, \Tone, A^n}} & \\
				P(n,l) \cdot [(\Tone \mysquare \Tone) \mysquare A^n] \ar[dd]_{\id_{P(n,l)} \cdot \alpha^{\Coll_{\mysquare}}_{\Tone, \Tone, A^n}} & & \\
				& & [P(m,l) \times P(n,m)] \cdot [\Tone \mysquare (\Tone \mysquare A^n)] \ar[dd] \ar[dll]_{\circ^P_{n,m,l} \cdot \id_{\Tone \mysquare (\Tone \mysquare A^n)}\qquad} \\
				P(n,l) \cdot [\Tone \mysquare (\Tone \mysquare A^n)] \ar[dd] & & \\
				& & [P(m,l) \times P(n,m)] \cdot [\Tone \mysquare (\Tone \mysquare A)^n] \ar[dd]^{\id_{P(m,l) \times P(n,m)} \cdot [\id_{\Tone} \mysquare \omega^n]} \ar[dll]_{\circ^P_{n,m,l} \cdot \id_{\Tone \mysquare (\Tone \mysquare A)^n}\qquad} \\
				P(n,l) \cdot [\Tone \mysquare (\Tone \mysquare A)^n] \ar[dd]_{\id_{P(n,l)} \cdot [\id_{\Tone} \mysquare \omega^n]} & & \\
				& & [P(m,l) \times P(n,m)] \cdot [\Tone \mysquare A^n] \ar[dd] \ar[dll]_{\circ^P_{n,m,l} \cdot \id_{\Tone \mysquare A^n}} \\
				P(n,l) \cdot [\Tone \mysquare A^n] \ar[dd] & & \\
				& & [P(m,l) \times P(n,m)] \cdot [\Tone \mysquare A]^n \ar[dd]^{\id_{P(m,l) \times P(n,m)} \cdot \omega^n} \ar[dll]_{\circ^P_{n,m,l} \cdot \id_{[\Tone \mysquare A]^n}} \\
				P(n,l) \cdot [\Tone \mysquare A]^n \ar[ddr]^{\id_{P(n,l)} \cdot \omega^n} & & \\
				& & [P(m,l) \times P(n,m)] \cdot A^n \ar[dl]_{\circ^P_{n,m,l} \cdot \id_{A^n}} \\
				& P(n,l) \cdot A^n & \\
		}}$$
		
		Each of these squares commute by the functoriality of the $\cdot$ operation.
		
		\newpage
		\textbf{DIAGRAM F}
		
		$$\adjustbox{max width=\columnwidth}{\xymatrix{ & P(m,l) \cdot [P(n,m) \cdot A^n] \ar[dl] \ar[ddr]^{\id_{P(m,l)} \cdot \nu_{n,m}} & \\
				[P(m,l) \times P(n,m)] \cdot A^n \ar[dd]_{\circ^P_{n,m,l} \cdot \id_{A^n}} & & \\
				& & P(m,l) \cdot A^m \ar[ddl]_{\nu_{m,l}} \\
				P(n,l) \cdot A^n \ar[dr]^{\nu_{n,l}} & & \\
				& A^l & 
		}}$$
		
		This diagram commutes by the fact that $A$ is an algebra for the underlying set PRO $P$.
		\\
		
		\textbf{DIAGRAM G}
		
		$$\adjustbox{max width=\columnwidth}{\xymatrix{ & [(P(m,l) \times P(n,m)) \cdot (\Tone \mysquare \Tone)] \mysquare A^n \ar[dr] \ar[ddl]_{[\circ^P_{n,m,l} \cdot \id_{\Tone \mysquare \Tone}] \mysquare \id_{A^n}\qquad} & \\
				& & [P(m,l) \times P(n,m)] \cdot [(\Tone \mysquare \Tone) \mysquare A^n] \ar[dd]^{\circ^P_{n,m,l} \cdot \id_{(\Tone \mysquare \Tone) \mysquare A^n}} \\
				[P(n,l) \cdot (\Tone \mysquare \Tone)] \mysquare A^n \ar[dd]_{[\id_{P(n,l)} \cdot \phi] \mysquare \id_{A^n}} \ar[drr] & & \\
				& & P(n,l) \cdot [(\Tone \mysquare \Tone) \mysquare A^n] \ar[ddl]^{\qquad\id_{P(n,l)} \cdot [\phi \mysquare \id_{A^n}]} \\
				[P(n,l) \cdot \Tone] \mysquare A^n \ar[dr] & & \\
				& P(n,l) \cdot [\Tone \mysquare A^n] & \\
		}}$$
		
		Both of these squares commute by the naturality of the reindexing operation.
		
		\newpage
		\textbf{DIAGRAM H}
		
		$$\adjustbox{max width=\columnwidth}{\xymatrix{ & P(n,l) \cdot [(\Tone \mysquare \Tone) \mysquare A^n] \ar[dr]^{\qquad\id_{P(n,l)} \cdot \alpha^{\Coll_{\mysquare}}_{\Tone, \Tone, A^n}} \ar[dl]_{\id_{P(n,l)} \cdot [\phi \mysquare \id_{A^n}]\qquad} & \\
				P(n,l) \cdot [\Tone \mysquare A^n] \ar[ddd] & & P(n,l) \cdot [\Tone \mysquare (\Tone \mysquare A^n)] \ar[d] \\
				& & P(n,l) \cdot [\Tone \mysquare (\Tone \mysquare A)^n] \ar[d]^{\id_{P(n,l)} \cdot [\id_{\Tone} \mysquare \omega^n]} \\
				& & P(n,l) \cdot [\Tone \mysquare A^n] \ar[d] \\
				P(n,l) \cdot [\Tone \mysquare A]^n \ar[dr]_{\id_{P(n,l)} \cdot \omega^n\qquad} & & P(n,l) \cdot [\Tone \mysquare A]^n \ar[dl]^{\qquad\id_{P(n,l)} \cdot \omega^n} \\
				& P(n,l) \cdot A^n & \\
		}}$$
		
		This diagram commutes by the fact that $A$ has the structure of a strict $\omega$-category by assumption.

		\newpage
		\textbf{MONOIDAL SUM PRESERVES ACTION}
		
		$$\adjustbox{max width=\columnwidth}{\xymatrix{ & & & & \cdot \ar[rr] \ar[dddd] & & \cdot \ar[rr]^{[\id_{P(n,m)} \cdot \omega^n] \times [\id_{P(l,k)} \cdot \omega^l]} & & \cdot \ar[rrdddd] \ar[rrrrdd]^{\nu_{n,m} \times \nu_{l,k}} & & & & \\
				\\
				& & \cdot \ar[uurr] & & & & & \textbf{B} & & & & & \cdot \ar[dd] \\
				\\
				\circ \ar[uurr]^{\qquad\qquad\boxtimes_{P(n,m) \cdot \Tone, P(l,k) \cdot \Tone, A^n, A^l}} \ar[dddd] & & \textbf{A} & & \cdot \ar[rr] & & \cdot \ar[rrdddd] & & & & \cdot \ar[dddd] & & \circ \\
				& & & & & \textbf{E} & & & & & & \textbf{C} & \\
				& & \cdot \ar[uurr] \ar[rr] & & \cdot \ar[rrdd] & & & & & & & & \\
				& & & \textbf{D} & & & & & & & & & \\
				\cdot \ar[uurr] \ar[ddrr]^{[+^P_{n,m,l,k} \cdot \id_{\Tone \times \Tone}] \mysquare \id_{A^n \times A^l}} \ar[rr] & & \cdot \ar[rr] & & \cdot \ar[rr] & & \cdot \ar[rr] & & \cdot \ar[rr] & & \cdot \ar[rr] & & \cdot \ar[uuuu]_{\nu_{n+l,m+k}} \\
				& & & & & & & \textbf{F} & & & & & \\
				& & \cdot \ar[rr]_{[\id_{P(n+l,m+k)} \cdot \phi] \mysquare \id_{A^n \times A^l}} & & \cdot \ar[rr] & & \cdot \ar[rr] & & \cdot \ar[rr] & & \cdot \ar[rruu]_{\id_{P(n+l,m+k)} \cdot \omega^{n+l}} & & 
		}}$$

		\newpage
		\textbf{DIAGRAM A}
		
		$$\adjustbox{max width=\columnwidth}{\xymatrix{ & [(P(n,m) \cdot \Tone) \times (P(l,k) \cdot \Tone)] \mysquare [A^n \times A^l] \ar[ddl] \ar[ddr]^{\qquad\qquad\boxtimes_{P(n,m) \cdot \Tone, P(l,k) \cdot \Tone, A^n, A^l}} & \\
				\\
				[(P(n,m) \times P(l,k)) \cdot (\Tone \times \Tone)] \mysquare [A^n \times A^l] \ar[ddd] & & [(P(n,m) \cdot \Tone) \mysquare A^n] \times [(P(l,k) \cdot \Tone) \mysquare A^l] \ar[ddd] \\
				\\
				\\
				[P(n,m) \times P(l,k)] \cdot [(\Tone \times \Tone) \mysquare (A^n \times A^l)] \ar[ddr]^{\qquad\qquad\id_{P(n,m) \times P(l,k)} \cdot \boxtimes_{\Tone, \Tone, A^n, A^l}} & & [P(n,m) \cdot (\Tone \mysquare A^n)] \times [P(l,k) \cdot (\Tone \mysquare A^l)] \ar[ddl] \\
				\\
				& [P(n,m) \times P(l,k)] \cdot [(\Tone \mysquare A^n) \times (\Tone \mysquare A^l)] & 
		}}$$
		
		We shall once again justify commutativity by describing how each of the two sides of this diagram act on a generic element.  We start with a pair of globular cells $(h,k)$, both with arity shape $\sigma$, such that the first is indexed by an operation $\phi \in P(n,m)$ and the second is indexed by on operation $\psi \in P(l,k)$.  Hence we may write $h$ as $\sigma_{\phi}$ and $k$ as $\sigma_{\psi}$ to get $(h,k)= (\sigma_{\phi},\sigma_{\psi}) \in (P(n,m)\cdot\Tone)\times(P(l,k)\cdot\Tone)$.  Moreover, $(\sigma_{\phi},\sigma_{\psi})$ is equipped with a coloring of its arity by cells in $A^n \times A^l$.  We now wish to look at two different compositions of maps to see that the corresponding diagram of morphisms commutes.
		
		We begin the first string by applying the middle four interchange to the cells described above.  This gives a pair whose first entry is $\sigma_{\phi}$ equipped with the coloring of its arity by the first $n$ cells in the coloring by cells in $A^n \times A^l$.  We shall denote this `word' coloring $\sigma_{\phi}$ by $\kappa_n \in \mathcal{T}(A^n)$. The second entry of the pair is then the cell $\sigma_{\psi}$ equipped with the coloring of its arity by the last $l$ cells in the coloring, which we shall denote $\kappa_l$.  Hence, the interchange transformation sends $((\sigma_{\phi}, \sigma_{\psi}),(\kappa_n, \kappa_l))$ to $((\sigma_{\phi},\kappa_n),(\sigma_{\psi},\kappa_l))$.  Since a cell of the cartesian product is a tuple of cells all having the same arity shape, we may re-index $((\sigma_{\phi},\kappa_n),(\sigma_{\psi},\kappa_l))$ as $((\sigma,\kappa_n)_{\phi},(\sigma,\kappa_l)_{\psi})$ without any loss of information.  Similarly, we can re-index this tuple as $((\sigma,\kappa_n),(\sigma,\kappa_l))_{(\phi,\psi)}$.
		
		If we instead follow the other composition, we first take $((\sigma_{\phi}, \sigma_{\psi}),(\kappa_n, \kappa_l))$ and, instead of applying the interchange transformation, re-index it as $((\sigma, \sigma)_{(\phi,\psi)},(\kappa_n, \kappa_l))$.  This can then also be re-indexed as $((\sigma, \sigma),(\kappa_n, \kappa_l)) _{(\phi,\psi)}$.  We can then apply the interchange morphism in just the $(\phi,\psi)$ summand to get $((\sigma,\kappa_n),(\sigma,\kappa_l))_{(\phi,\psi)}$ as we had before.
		
		\newpage
		\textbf{DIAGRAM B}
		
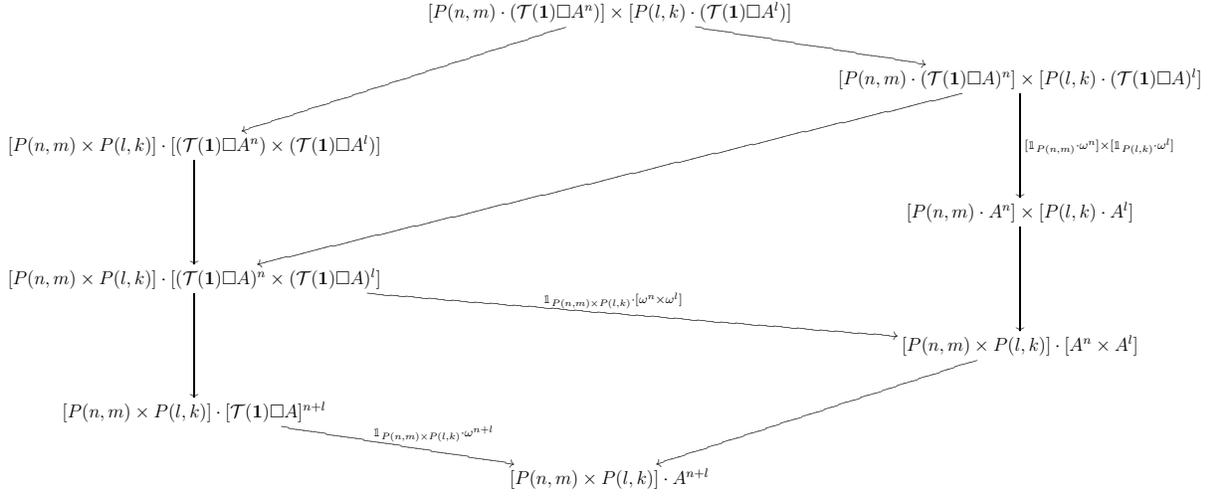
		$$\adjustbox{max width=\columnwidth}{\xymatrix{ & [P(n,m) \cdot (\Tone \mysquare A^n)] \times [P(l,k) \cdot (\Tone \mysquare A^l)] \ar[dr] \ar[ddl] & \\
				& & [P(n,m) \cdot (\Tone \mysquare A)^n] \times [P(l,k) \cdot (\Tone \mysquare A)^l] \ar[dd]^{[\id_{P(n,m)} \cdot \omega^n] \times [\id_{P(l,k)} \cdot \omega^l]} \ar[dddll] \\
				[P(n,m) \times P(l,k)] \cdot [(\Tone \mysquare A^n) \times (\Tone \mysquare A^l)] \ar[dd] & & \\
				& & [P(n,m) \cdot A^n] \times [P(l,k) \cdot A^l] \ar[dd] \\
				[P(n,m) \times P(l,k)] \cdot [(\Tone \mysquare A)^n \times (\Tone \mysquare A)^l] \ar[dd] \ar[drr]^{\id_{P(n,m) \times P(l,k)} \cdot [\omega^n \times \omega^l]} & & \\
				& & [P(n,m) \times P(l,k)] \cdot [A^n \times A^l] \ar[ddl] \\
				[P(n,m) \times P(l,k)] \cdot [\Tone \mysquare A]^{n+l} \ar[dr]^{\qquad\id_{P(n,m) \times P(l,k)} \cdot \omega^{n+l}} & & \\
				& [P(n,m) \times P(l,k)] \cdot A^{n+l} & \\
		}}$$
		
		The two topmost squares commute by the naturality of the operation of reindexing copwers.  The bottom square commutes by the naturality of the associator for $\times$ in $\Coll$.
		\\
		
		\textbf{DIAGRAM C}
		
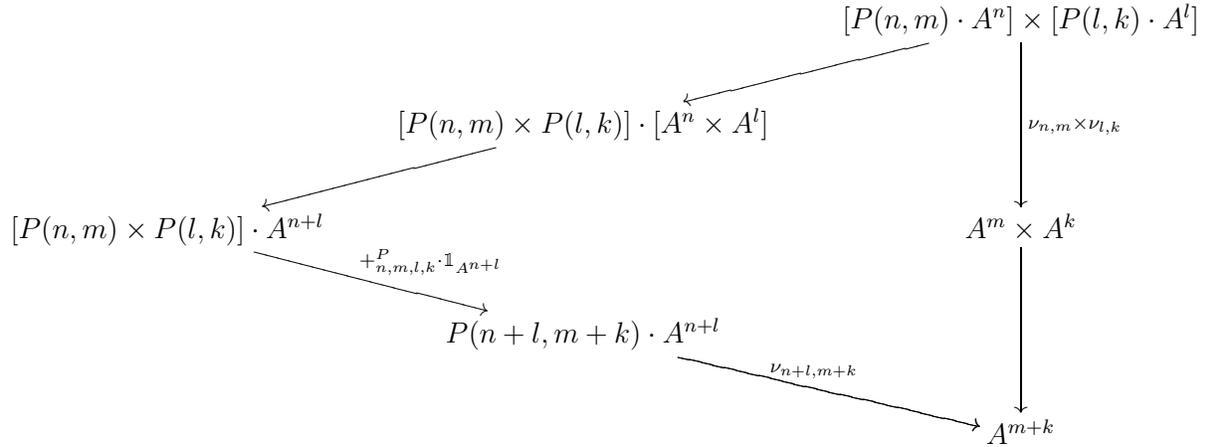
		$$\adjustbox{max width=\columnwidth}{\xymatrix{ & & [P(n,m) \cdot A^n] \times [P(l,k) \cdot A^l] \ar[dl] \ar[dd]^{\nu_{n,m} \times \nu_{l,k}} \\
				& [P(n,m) \times P(l,k)] \cdot [A^n \times A^l] \ar[dl] & \\
				[P(n,m) \times P(l,k)] \cdot A^{n+l} \ar[dr]^{\qquad+^P_{n,m,l,k} \cdot \id_{A^{n+l}}} & & A^m \times A^k \ar[dd] \\
				& P(n+l,m+k) \cdot A^{n+l} \ar[dr]^{\nu_{n+l,m+k}} \ar[dr] & \\
				& & A^{m+k} \\
		}}$$
		
		This diagram commutes by the fact that $+$ is the monoidal product for the underlying set PRO $P$.
		
		\newpage
		\textbf{DIAGRAM D}
		
		$$\adjustbox{max width=\columnwidth}{\xymatrix{ & [(P(n,m) \times P(l,k)) \cdot (\Tone \times \Tone)] \mysquare [A^n \times A^l] \ar[ddl]^{\qquad\qquad\quad[\id_{P(n,m) \times P(l,k)} \cdot \phi] \mysquare \id_{A^n \times A^l}} \ar[dr] & \\
				& & [P(n,m) \times P(l,k)] \cdot [(\Tone \times \Tone) \mysquare (A^n \times A^l)] \ar[dd]^{\id_{P(n,m) \times P(l,k)} \cdot [\phi \mysquare \id_{A^n \times A^l}]} \\
				[(P(n,m) \times P(l,k)) \cdot \Tone] \mysquare [A^n \times A^l] \ar[dd] \ar[drr] & & \\
				& & [P(n,m) \times P(l,k)] \cdot [\Tone \mysquare (A^n \times A^l)] \ar[ddl] \\
				[(P(n,m) \times P(l,k)) \cdot \Tone] \mysquare A^{n+l} \ar[dr] & & \\
				& [P(n,m) \times P(l,k)] \cdot [\Tone \mysquare A^{n+l}] & 
		}}$$
		
		Both of these squares commute by the naturality of the operation of reindexing copowers.
		\\
		
		\textbf{DIAGRAM E}
		
		$$\adjustbox{max width=\columnwidth}{\xymatrix{ & [P(n,m) \times P(l,k)] \cdot [(\Tone \times \Tone) \mysquare (A^n \times A^l)] \ar[ddl] \ar[ddr] & \\
				\\
				[P(n,m) \times P(l,k)] \cdot [\Tone \mysquare (A^n \times A^l)] \ar[ddd] & & [P(n,m) \times P(l,k)] \cdot [(\Tone \mysquare A^n) \times (\Tone \mysquare A^l)] \ar[ddd] \\
				\\
				\\
				[P(n,m) \times P(l,k)] \cdot [\Tone \mysquare A^{n+l}] \ar[ddr] & & [P(n,m) \times P(l,k)] \cdot [(\Tone \mysquare A)^n \times (\Tone \mysquare A)^l] \ar[ddl] \\
				\\
				& [P(n,m) \times P(l,k)] \cdot [\Tone \mysquare A]^{n+l} & 
		}}$$
		
		Note that each object and arrow of this diagram is a copower indexed by $P(n,m) \times P(l,k)$.  By the functoriality of $\cdot$, it is hence enough to show that this diagram commutes prior to taking copowers.  To see this, note initially that the first map along the left hand side (which is a $\mysquare$-product of the unitor for $\times$ in $\Coll$ with an identity map) is invertible.  Note then that the operation of preserving cartesian power in the second variable, which is seen in this diagram in both the last map along the left hand side as well as the second to last along the right hand side, factors as series of inverse unitor maps for $\times$ followed by a series of applications of $\boxtimes$.  We hence see, after inverting the first map along the left hand side, that both sides of this diagram are simply a different choice in the order in which these iterated unitors and copies of $\boxtimes$ are applied.  Hence the diagram commutes by the coherence theorem for lax-monoidal functors.
		
		\newpage
		\textbf{DIAGRAM F}
		
		$$\adjustbox{max width=\columnwidth}{\xymatrix{ & [(P(n,m) \times P(l,k)) \cdot (\Tone \times \Tone)] \mysquare [A^n \times A^l] \ar[dr]^{\qquad\qquad[\id_{P(n,m) \times P(l,k)} \cdot \phi] \mysquare \id_{A^n \times A^l}} \ar[ddl]_{[+^P_{n,m,l,k} \cdot \id_{\Tone \times \Tone}] \mysquare \id_{A^n \times A^l}\qquad\qquad\quad} & \\
				& & [(P(n,m) \times P(l,k)) \cdot \Tone] \mysquare [A^n \times A^l] \ar[dd] \ar[dddll]_{[+^P_{n,m,l,k} \cdot \id_{\Tone}] \mysquare \id_{A^n \times A^l}\qquad\qquad} \\
				[P(n+l,m+k) \cdot (\Tone \times \Tone)] \mysquare [A^n \times A^l] \ar[dd]_{[\id_{P(n+l,m+k)} \cdot \phi] \mysquare \id_{A^n \times A^l}} & & \\
				& & [(P(n,m) \times P(l,k)) \cdot \Tone] \mysquare A^{n+l} \ar[dd] \ar[dddll]_{[+^P_{n,m,l,k} \cdot \id_{\Tone}] \mysquare \id_{A^{n+l}}\qquad\qquad} \\
				[P(n+l,m+k) \cdot \Tone] \mysquare [A^n \times A^l] \ar[dd] & & \\
				& & [P(n,m) \times P(l,k)] \cdot [\Tone \mysquare A^{n+l}] \ar[dd] \ar[dddll]_{+^P_{n,m,l,k} \cdot \id_{\Tone \mysquare A^{n+l}}\qquad\qquad} \\
				[P(n+l,m+k) \cdot \Tone] \mysquare A^{n+l} \ar[dd] & & \\
				& & [P(n,m) \times P(l,k)] \cdot [\Tone \mysquare A]^{n+l} \ar[dd]^{\id_{P(n,m) \times P(l,k)} \cdot \omega^{n+l}} \ar[dddll]_{+^P_{n,m,l,k} \cdot \id_{[\Tone \mysquare A]^{n+l}}\qquad\qquad} \\
				P(n+l,m+k) \cdot [\Tone \mysquare A^{n+l}] \ar[dd]_{} & & \\
				& & [P(n,m) \times P(l,k)] \cdot A^{n+l} \ar[ddl]^{\qquad+^P_{n,m,l,k} \cdot \id_{A^{n+l}}} \\
				P(n+l,m+k) \cdot [\Tone \mysquare A]^{n+l} \ar[dr]_{\id_{P(n+l,m+k)} \cdot \omega^{n+l}\qquad\qquad} & & \\
				& P(n+l,m+k) \cdot A^{n+l} & \\
		}}$$
		
		The topmost square commutes by the functoriality of $\cdot$.  The second square from the top commutes by the naturality of the associator for $\times$ in $\Coll$.  The middle square commutes by the naturality of the operation of reindexing copowers.  The bottom two squares also commute by the functoriality of $\cdot$.
		
		\newpage
		\textbf{UNIT IS REPRESENTED BY THE ACTION}
		
		$$\adjustbox{max width=\columnwidth}
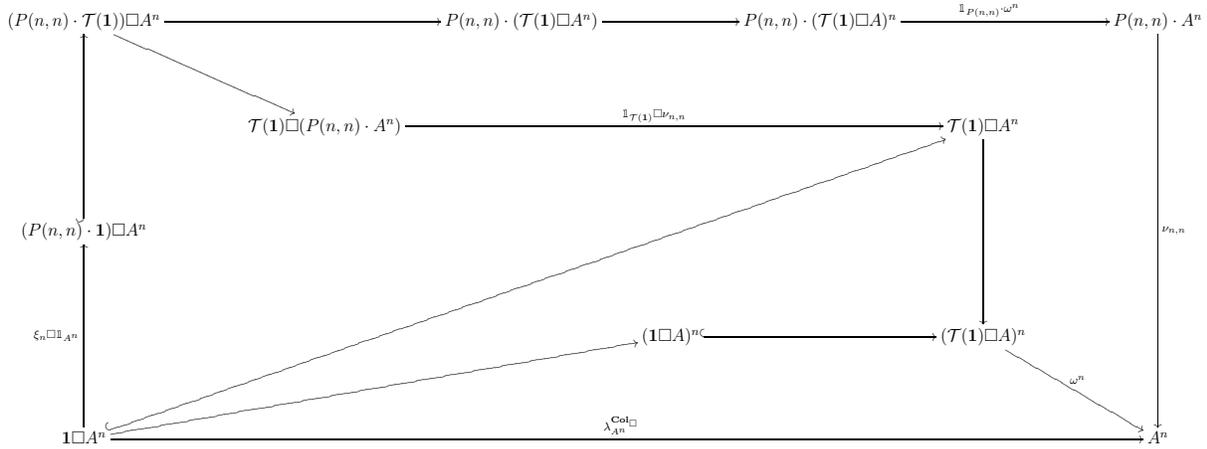
{\xymatrix{(P(n,n) \cdot \Tone) \mysquare A^n \ar[rrr] \ar[ddrr] & & & P(n,n) \cdot (\Tone \mysquare A^n) \ar[rr] & & P(n,n) \cdot (\Tone \mysquare A)^n \ar[rrr]^{\id_{P(n,n)} \cdot \omega^n} & & & P(n,n) \cdot A^n \ar[dddddddd]^{\nu_{n,n}} \\
				\\
				& & \Tone \mysquare (P(n,n) \cdot A^n) \ar[rrrr]^{\id_{\Tone} \mysquare \nu_{n,n}} & & & & \Tone \mysquare A^n \ar[dddd] & & \\
				\\
				(P(n,n) \cdot \one) \mysquare A^n \ar@{^{(}->}[uuuu] & & & & & & & & \\
				\\
				& & & & (\one \mysquare A)^n \ar@{^{(}->}[rr] & & (\Tone \mysquare A)^n \ar[ddrr]^{\omega^n} & & \\
				\\
				\one \mysquare A^n \ar[uuuu]^{\xi_n \mysquare \id_{A^n}} \ar@{^{(}->}[rrrrrruuuuuu] \ar[rrrruu] \ar[rrrrrrrr]^{\lambda^{\Coll_{\mysquare}}_{A^n}} & & & & & & & & A^n
		}}$$
		
		The top right region commutes by the fact that operations in $\mathcal{P}$ act on $A$ as strict $\omega$-functors.  The upper of the two middle regions commutes by the fact that the action in this region is by identities from the set PRO $P$.  The lower of the middle two regions commutes by the naturality of the iterated middle four interchange which sends each $\mysquare$ product with a cartesian power of a collection in the second variable to a canonical cartesian power of a $\mysquare$ product.  To see why the bottom region commutes, consider first the upper path of this region.  After an initial reindexing, this composition amounts to an action on an element of $\mathcal{T}(A^n)$ by a generating globular cell from $\one \subset \Tone$.  But since these cells act as identities with respect to the $\mysquare$ product, this is the same as simply applying $\lambda^{\Coll_{\mysquare}}_{A^n}$ to $\one \mysquare A^n$.
	\end{proof}

	\section{Contractions and Leinster Fibrations}
	    In Leinster's presentation of weak $\omega$-categories in \emph{Higher Operads, Higher Categories}\cite{leinster2004higher}, he proves the existence of an initial globular operad with contraction using a theorem of Kelly\cite{kelly_UnifiedTreat1980} which asserts that the strict pullback in $\Cat$ of two finitary and monadic functors, both of whose target is locally finitely presentable, is monadic.  In his construction, the two finitary and monadic functors are the underlying functors for the monads on $\Coll$ which have as algebras collections with contraction and globular operads respectively.  Hence, his pullback monad, which we shall denote $\mathfrak{O}$, has as algebras globular operads with contraction.  Applying $\mathfrak{O}$ to the initial object $\{\} \in \Coll$ constructs a collection $\mathfrak{O}(\{\})$ that, when thought of as an algebra for $\mathfrak{O}$ when equipped with the structure map ${\mu^{\mathfrak{O}}_{\{\}}:\mathfrak{O}^2(\{\}) \rightarrow \mathfrak{O}(\{\})}$ induced by the component at $\{\}$ of the multiplication transformation for $\mathfrak{O}$, is the initial free globular operad with contraction.   Algebras for the operad $\mathfrak{O}(\{\})$ are then by construction weak $\omega$-categories.
	    
	    We shall eventually use this same trick to construct something much like a globular PRO, whose algebras are by construction weak $\omega$-categorifications of a particular equational algebraic theory.  We do not yet have the machinery to construct such an object.  We need one more piece of structure: a special lifting property.  We begin by recalling Leinster's notion of a contraction structure on a collection\cite{leinster2004higher}.
	
	\begin{definition}
		Given a globular set $(X,s_X,t_X)$, two $n$-cells $\nu^-, \nu^+ \in X$ are \emph{parallel} if $s_X(\nu^-) = s_X(\nu^+)$ and $t_X(\nu^-) = t_X(\nu^+)$.  All zero dimensional cells in $X$ are parallel.
	\end{definition}
	
	Now, given a map $f:X \rightarrow Y$ of globular sets, for each nonzero $n$-cell $\nu \in Y_n$ we may consider the set
	$$\text{Par}_f(\nu) := \{(\rho^-,\rho^+) \in X_{n-1} \times X_{n-1} | \rho^- \text{ and } \rho^+ \text{ are parallel}, f(\rho^-) = s_Y^n(\nu), f(\rho^+) = t_Y^n(\nu) \}$$
	of pairs of parallel $(n-1)$-cells in $X$ that map via $f$ to the boundary of $\nu$ in $Y$.
	
	\begin{definition}
		Given a map $f:X \rightarrow Y$ of globular sets, a $\emph{contraction}$ $(f:X \rightarrow Y, \kappa^{f})$ on $f$ is a sequence of maps $\kappa^{f} = \{\kappa_{\nu}:\text{Par}_f(\nu) \rightarrow X_n \}$, indexed by the nonzero $n$-cell $\nu \in Y_n$, such that for each nonzero $\nu \in Y$
		$$s_X^n(\kappa_{\nu}(\rho^-,\rho^+)) = \rho^-$$
		$$t_X^n(\kappa_{\nu}(\rho^-,\rho^+)) = \rho^+$$
		$$f(\kappa_{\nu}(\rho^-,\rho^+)) = \nu$$
		for every pair $(\rho^-,\rho^+) \in \text{Par}_f(\nu)$.
	\end{definition}
	This definition may be weakened so that for any nonzero $n$-cell $\nu \in Y_n$ and any pair $(\rho^-,\rho^+) \in \text{Par}_f(\nu)$ we require only that there exists an $n$-cell $\kappa \in X_n$ such that $\rho^-$ and $\rho^+$ bound $\kappa$ in $X$ and $f(\kappa) = \nu$.  In other words, whenever we can lift the boundary of an $n$-cell in $Y$ we are furthermore able to lift the entire cell.
	
	\begin{definition}
		A \emph{Leinster fibration} is a globular set map $f:X \rightarrow Y$ which satisfies the property that for all $n \in \mathbb{N}$ and $\nu \in Y_n$, for each pair $(\rho^-,\rho^+) \in \text{Par}_f(\nu)$ there exists a cell $\gamma_{\nu}^{(\rho^-,\rho^+)} \in X_n$ such that
		$$s_X^n(\gamma_{\nu}^{(\rho^-,\rho^+)}) = \rho^-$$
		$$t_X^n(\gamma_{\nu}^{(\rho^-,\rho^+)}) = \rho^+$$
		$$f(\gamma_{\nu}^{(\rho^-,\rho^+)}) = \nu$$
	\end{definition}
	Note that in the presence of the axiom of choice, being a Leinster fibration is equivalent to the existence of a contraction structure on a globular set map.  A contraction structure is a choice of lifts for a Leinster fibration.  In this way, we may think of a contraction structure as a \emph{split Leinster fibration}.
	
	\begin{definition}
		A $\emph{contraction structure on a globular operad}$ $a: A \rightarrow \Tone$ is a contraction on the unique map from $a$ to the terminal collection $\Tone \rightarrow \Tone$.  In particular, it is a contraction on the arity map $a$.
	\end{definition}
	
	We shall now extend this construction to the theory of globular PROs.
	
	\begin{definition}
		A $\emph{contraction structure on a}$ $\mathbb{N}\Coll\emph{-graph homomorphism}$ is a map of $\mathbb{N}\Coll$-graphs $F:\mathbf{G} \rightarrow \mathbf{H}$ such that each component $F_{n,m}:\mathbf{G}(n,m) \rightarrow \mathbf{H}(n,m)$, all of which are maps of globular sets, comes equipped with a specified contraction.  We moreover call $\mathbf{G}$ a $\mathbb{N}\Coll$\emph{-graph with contraction over }$\mathbf{H}$ when equipped with such an $F$.  
	\end{definition}
	
	Most often when working with contractions we will fix the target.  Collectively, $\mathbb{N}\Coll$-graphs with contraction over the $\mathbb{N}\Coll$-graph $\mathbf{G}$ form a category $\Cont(\sfrac{\mathbb{N}\Coll\Graph}{\mathbf{G}})$ whose morphisms are those in $\sfrac{\mathbb{N}\Coll\Graph}{\mathbf{G}}$ which preserve the contraction structure on each hom-object component.
	
	Given any object in $\sfrac{\mathbb{N}\Coll\Graph}{\mathbf{G}}$ we can expand it to have a canonical contraction structure.  In appendix G of Leinster's book\cite{leinster2004higher}, he describes a functorial construction for giving a generic collection $p:P \rightarrow \Tone$ a canonical contraction structure
	$$(C(p):CP \rightarrow \Tone, \kappa^{C(p)})$$
	by inductively adjoining the requisite cells to $P$ at each dimension to get the new collection $C(p)$ which has a natural induced contraction $\kappa^{C(p)}$.  Moreover, Leinster's construction does not require that the object upon which we are adjoining a contraction structure be a collection.  There is an analogous construction which sends any globular set map to a globular set map with contraction.  This allows us to naturally extend this construction to $\mathbb{N}\Coll$-graphs.
	
	\begin{definition}
		Given a $\mathbb{N}\Coll$-graph $\mathbf{G}$, for any object $H: \mathbf{H} \rightarrow \mathbf{G}$ in $\sfrac{\mathbb{N}\Coll\Graph}{\mathbf{G}}$, the $\Coll$-$\emph{graph with freely generated contraction structure over}$ $\mathbf{G}$ $\emph{on}$ $H$, denoted $C_{\mathbf{G}}(H): C_{\mathbf{G}}\mathbf{H} \rightarrow \mathbf{G}$, is constructed by applying Leinster's free contraction construction to each of the globular set map components
		$$H_{n,m}:\mathbf{H}(n,m) \rightarrow \mathbf{G}(n,m)$$
		to make them globular sets with contraction
		$$\left(C_{\mathbf{G}}(H_{n,m}):C_{\mathbf{G}}\mathbf{H}(n,m) \rightarrow \mathbf{G}(n,m),\kappa^{C_{\mathbf{G}}(H_{n,m})}\right)$$
		which collectively induce upon $C_{\mathbf{G}}(H):C_{\mathbf{G}}\mathbf{H} \rightarrow \mathbf{G}$ the structure of a $\mathbb{N}\Coll$-graph with contraction over $\mathbf{G}$.
	\end{definition}
	
	\begin{theorem}\label{ContFinAndMod}
		Given a $\mathbb{N}\Coll$-graph $\mathbf{G}$, $C_{\mathbf{G}}: \sfrac{\mathbb{N}\Coll\Graph}{\mathbf{G}} \rightarrow \Cont(\sfrac{\mathbb{N}\Coll\Graph}{\mathbf{G}})$ has a right adjoint $R_{\mathbf{G}}: \Cont(\sfrac{\mathbb{N}\Coll\Graph}{\mathbf{G}}) \rightarrow \sfrac{\mathbb{N}\Coll\Graph}{\mathbf{G}}$ which is finitary and monadic.
	\end{theorem}
	\begin{proof}
		It is immediately clear that such a right adjoint exists.  Given any $\mathbb{N}\Coll$-graph with contraction in $\Cont(\sfrac{\mathbb{N}\Coll\Graph}{\mathbf{G}})$, $R_{\mathbf{G}}$ simply forgets the contraction with which each hom-object is equipped.  As Leinster showed \cite{leinster2004higher}, the right adjoint for his construction is both finitary and monadic over $\Coll$.  It is then clear by construction that these properties are preserved at the level of $R_{\mathbf{G}}$.
	\end{proof}
	
	We will use this construction on $\mathbb{N}\Coll$-graphs to create something like a globular PRO whose algebras are weak versions of the algebras for a chosen globular PRO.  We can fix a weakenable globular PRO $\mathcal{P}$ whose algebras we wish to weaken.  We then consider the category $\sfrac{\mathbb{N}\Coll\Graph}{U(\mathcal{P})}$ of $\Coll$-graphs with object set $\mathbb{N}$ sliced over the underlying $\Coll$-graph of the PRO whose algebras are the strict models of the theory we wish to weaken.  We then repeat the previous construction with $\mathbf{G} = U(\mathcal{P})$.

	\begin{definition}
		Given a globular PRO $\mathcal{P}$, for any object $G:\textbf{G} \rightarrow U(\mathcal{P})$ in $\sfrac{\mathbb{N}\Coll\Graph}{U(\mathcal{P})}$, the $\Coll$-$\emph{graph with freely generated contraction structure over}$ $U(\mathcal{P})$ $\emph{on}$ $G$, denoted $C_{\mathcal{P}}(G): C_{\mathcal{P}}\mathbf{G} \rightarrow U(\mathcal{P})$, is constructed by applying Leinster's free contraction construction to each of the globular set map components
		$$G_{n,m}:\mathbf{G}(n,m) \rightarrow U(\mathcal{P})(n,m)$$
		to make them globular sets with contraction
		$$\left(C_{\mathcal{P}}(G_{n,m}):C_{\mathcal{P}}\mathbf{G}(n,m) \rightarrow U(\mathcal{P})(n,m),\kappa^{C_{\mathcal{P}}(G_{n,m})}\right)$$
		which collectively induce upon $C_{\mathcal{P}}(G):C_{\mathcal{P}}\mathbf{G} \rightarrow U(\mathcal{P})$ the structure of a $\mathbb{N}\Coll$-graph with contraction over $U(\mathcal{P})$.
	\end{definition}

	\section{The Weakening Monad}
	With the functor $C_{\mathcal{P}}$ we are almost able to construct the monad which will be the key in constructing weak $\omega$-categorifications of a particular algebraic theory.  However, $C_{\mathcal{P}}$ only allows us to construct free contraction structures on $\Coll$-graphs over $U(\mathcal{P})$.  We will need to extend the corresponding monads for the underlying functors $\mathcal{W}: \Mon\mathbb{N}\mathcal{D}\Graph \rightarrow \mathbb{N}\mathcal{D}\Graph$ and $U: \Coll\Cat \rightarrow \Coll\Graph$ from above. First we will need the following theorem and corollary.
	
	\begin{theorem}
		Let $T: \mathcal{C} \rightarrow \mathcal{C}$ be a monad over $\mathcal{C}$ whose associated adjunction is $F \dashv U$, such that $T=UF$, with unit and counit $\eta: \id_{\mathcal{C}} \Rightarrow UF$ and $\epsilon: FU \Rightarrow \id_{\Talg}$ respectively, where $\Talg$ is the category of T-algebras over $\mathcal{C}$.  Then, given any T-algebra $A$, the induced functor $\overline{U}: \sfrac{\Talg}{A} \rightarrow \sfrac{\mathcal{C}}{U(A)}$ is monadic over $\sfrac{\mathcal{C}}{U(A)}$.
	\end{theorem}
	\begin{proof}
		Consider the functor $\mathcal{F}: \sfrac{\mathcal{C}}{U(A)} \rightarrow \sfrac{\Talg}{A}$ which is defined on objects as
		$$\mathcal{F}(x:X \rightarrow U(A)) := \epsilon_A(F(x)): F(X) \rightarrow FU(A) \rightarrow A$$
		and sends a morphism
		$$\xymatrix{X \ar[rr]^{f} \ar[dr]_{x} & & Y \ar[dl]^{y} \\
			& U(A) & 
		}$$
		in $\sfrac{\mathcal{C}}{U(A)}$ to the morphism
		$$\xymatrix{F(X) \ar[rrrr]^{\mathcal{F}(f)} \ar[d]_{F(x)} & & & & F(Y) \ar[d]^{F(y)} \\
			FU(A) \ar[drr]_{\epsilon_A} & & & & FU(A) \ar[dll]^{\epsilon_A} \\
			& & A & & 
		}$$
		in $\sfrac{\Talg}{A}$.  We shall show that $\mathcal{F} \dashv \overline{U}$ by checking that there is a natural isomorphism $\Phi$ between the appropriate hom-sets.  Let $n: N \rightarrow U(A)$ be any object in $\sfrac{\mathcal{C}}{U(A)}$ and $m:M \rightarrow A$ be any object in $\sfrac{\Talg}{A}$.  We first need to show that any morphism $f: \mathcal{F}(n) \rightarrow m$ can be identified with a morphism $\Phi(f): n \rightarrow \overline{U}(m)$.  Consider the following diagram obtained by precomposing $\overline{U}(f)$ with $\eta_N$.
		$$\xymatrix{N \ar[rrr]^{\eta_N} \ar[ddddrrr]_{n} & & & UF(N) \ar[dd]^{UF(n)} \ar[rrr]^{\overline{U}(f)} & & & U(M) \ar[ddddlll]^{U(m)} \\
			\\
			& & & UFU(A) \ar[dd]^{U(\epsilon_A)} & & & \\
			\\
			& & & U(A) & & &
		}$$
		Note that this diagram represents a morphism $N \rightarrow U(A)$ in $\sfrac{\mathcal{C}}{U(A)}$ which is our desired candidate for $\Phi(f)$.  It remains then to show that this diagram commutes.  Immediately the right square commutes by construction as the image of a morphism under a functor.  To see why the left square commutes consider the refinement of the previous diagram
		$$\xymatrix{N \ar[rrr]^{\eta_N} \ar[dddd]_{n} & & & UF(N) \ar[dd]^{UF(n)} \ar[rrr]^{\overline{U}(f)} & & & U(M) \ar[ddddlll]^{U(m)} \\
			\\
			& & & UFU(A) \ar[dd]^{U(\epsilon_A)} & & & \\
			\\
			U(A) \ar[rrr]^{\id_{U(A)}} \ar[uurrr]^{\eta_{U(A)}} & & & U(A) & & &
		}$$
		where the top square commutes by the naturality of $\eta$ and the bottom triangle commutes by the unit/counit relations satisfied by $F$ and $U$ as adjoint functors.  Dually, given a morphism $g:n \rightarrow \overline{U}(m)$ we wish to find a morphism $\Phi^{-1}(g): \mathcal{F}(n) \rightarrow m$.  Consider the diagram obtained by composing $\mathcal{F}$ with $\epsilon_M$.
		$$\xymatrix{F(N) \ar[rrr]^{\mathcal{F}(g)} \ar[ddr]_{F(n)} & & & FU(M) \ar[dd]^{FU(m)} \ar[rrr]^{\epsilon_M} & & & M \ar[ddddlll]^{m} \\
			\\
			& FU(A) \ar[ddrr]_{\epsilon_A} & & FU(A) \ar[dd]^{\epsilon_A} & & & \\
			\\
			& & & A & & &
		}$$
		Here the commutativity of the left pentagon follows by construction as the image of a functor and the right square follows by the naturality of $\epsilon$.  Hence this diagram represents the desired $\Phi^{-1}(g): \mathcal{F}(n) \rightarrow m$ showing that the two homsets are isomorphic.  Moreover, this isomorphism is natural in both variables by construction, as the relevant commutative naturality squares can each be seen as restrictions of the corresponding naturality square of the original adjunction.
		
		Now let $\mathfrak{T} := \overline{U}\mathcal{F}$ and consider the comparison functor $K^{\mathfrak{T}}: \sfrac{\Talg}{A} \rightarrow (\sfrac{\mathcal{C}}{U(A)})^{\mathfrak{T}}$.  We must show that it is an equivalence of categories. To see this we shall here think of $T$-algebras as a pair $(A, T(A) \rightarrow A)$ consisting of the underlying object of the algebra $A$ together with the structure map $T(A) \rightarrow A$.  From this perspective the functor $U$ simply forgets the associated structure map for the pair.  We can then consider a generic element $x^{T}:(X, T(X) \rightarrow X) \rightarrow (A, T(A) \rightarrow A)$ from $\sfrac{\Talg}{A}$ whose image under $K^{\mathfrak{T}}$ is given by the pair $(\overline{U}(x^{T}), \overline{U}(\epsilon_{x^T}^{\mathfrak{T}}))$ where $\epsilon_{x^T}^{\mathfrak{T}}: \mathcal{F}\overline{U}(x^T) \rightarrow x^T$ is the counit for the $\mathcal{F} \vdash \overline{U}$ adjunction.  Hence we have the following:
		$$K^{\mathfrak{T}}(x^T) = \left(\vcenter{\xymatrix{X \ar[d] \\
				A}}, \vcenter{\xymatrix{T(X) \ar[rrr] \ar[d] & & & X \ar[ddll] \\
				T(A) \ar[dr] \\
				& A}}\right)$$
		Consider now the following $\mathfrak{T}$-algebra in $(\sfrac{\mathcal{C}}{U(A)})^{\mathfrak{T}}$:
		$$\left( \vcenter{\xymatrix{Y \ar[d] \\
				A}}, \mathfrak{T}\left(\vcenter{\xymatrix{Y \ar[d] \\
				A}} \right) \rightarrow \vcenter{\xymatrix{Y \ar[d] \\
				A}} \right)
		= \left( \vcenter{\xymatrix{Y \ar[d] \\
				A}}, \vcenter{\xymatrix{T(Y) \ar[rrr] \ar[d] & & & Y \ar[ddll] \\
				T(A) \ar[dr] \\
				& A}} \right)$$
		Together these show that any $\mathfrak{T}$-algebra is the image of a $T$-algebra under $K^{\mathfrak{T}}$ and moreover this functor is in fact the identity functor on $\sfrac{\Talg}{A}$.  Hence $K^{\mathfrak{T}}$ is an equivalence of categories.  Therefore $\mathcal{F} \dashv \overline{U}$ is a monadic adjunction.
	\end{proof}
	
	\begin{corollary}\label{SliceFinAndMod}
		The induced functors $\overline{\mathcal{W}}_{\mathcal{P}}:\sfrac{\Mon\mathbb{N}\Coll\Graph}{U(\mathcal{P})} \rightarrow \sfrac{\mathbb{N}\Coll\Graph}{U(\mathcal{P})}$ and ${\overline{U}_{\mathcal{P}}: \sfrac{\mathbb{N}\Coll\Cat}{\mathcal{P}} \rightarrow \sfrac{\mathbb{N}\Coll\Graph}{U(\mathcal{P})}}$ for some globular PRO $\mathcal{P}$, where $\mathcal{W}$ and $U$ are the underlying functors for the free monoid and free $\Coll$-category monads defined above, are finitary and monadic over $\sfrac{\mathbb{N}\Coll\Graph}{U(\mathcal{P})}$.
	\end{corollary}
	\begin{proof}
		That these two functors are monadic follows immediately from the previous theorem.  Moreover, since the forgetful functor from a slice categories to the original category preserves and creates colimits, it follows that since $\mathcal{W}$ and $U$ are finitary, as shown in theorem \ref{MonoidalFinAndMod} and \ref{UnderCatFinAndMod} respectively, then so are $\overline{\mathcal{W}}_{\mathcal{P}}$ and $\overline{U}_{\mathcal{P}}$.
	\end{proof}
	
	To complete our construction we will also need that the target category $\sfrac{\mathbb{N}\Coll\Graph}{U(\mathcal{P})}$ for each of our three finitary and monadic underlying functors is locally finitely presentable.
	
	\begin{lemma}
		The category $\mathbb{N}\Coll\Graph$ is locally finitely presentable.
	\end{lemma}
	\begin{proof}
		Recall that $\mathbb{N}\Coll\Graph$ is equivalent to the category $\BiGrd\Coll$.  Moreover, we can think of each bi-graded collection as a countable product of ordinary collections over $\mathbb{N} \times \mathbb{N}$.  Hence we can express $\BiGrd\Coll$ as
		$$\BiGrd\Coll \cong \prod_{\mathbb{N} \times \mathbb{N}}\Coll \cong \prod_{\mathbb{N} \times \mathbb{N}} \sfrac{[\mathbb{G}^{op},\Set]}{\Tone} \cong \prod_{\mathbb{N} \times \mathbb{N}} \Set^{Elt(\Tone)^{op}} \cong$$
		$$\left(\Set^{Elt(\Tone)^{op}}\right)^{\mathbb{N} \times \mathbb{N}} \cong \Set^{(Elt(\Tone)^{op}) \times (\mathbb{N} \times \mathbb{N})} \cong \Set^{(Elt(\Tone) \times \mathbb{N} \times \mathbb{N})^{op}}$$
		where $Elt(\Tone)$ is the category of elements for the (covariant) presheaf functor $\Tone: \mathbb{G}^{op} \rightarrow \Set$.  We are here using it in order to perform the standard construction for writing a slice presheaf category as a presheaf category.  Also note that in the functor category $\left(\Set^{Elt(\Tone)^{op}}\right)^{\mathbb{N} \times \mathbb{N}}$, the object $\mathbb{N}$ is being thought of as the descrete category with object set $\mathbb{N}$.  This shows that $\BiGrd\Coll$, and hence $\mathbb{N}\Coll\Graph$, is a presheaf category.  The conclusion then follows from the fact that presheaf categories are locally finitely presentable \cite{borceux_1994}.
	\end{proof}
	
	\begin{corollary} \label{LocalFinPresen}
		Given a globular PRO $\mathcal{P}$, the category $\sfrac{\mathbb{N}\Coll\Graph}{U(\mathcal{P})}$, where the functor $U:\mathbb{N}\Coll\Cat \rightarrow \mathbb{N}\Coll\Graph$ is the underlying functor for the free $\Coll$-category adjunction above, is locally finitely presentable.
	\end{corollary}
	\begin{proof}
		The claim follows from the fact that $\sfrac{\mathbb{N}\Coll\Graph}{U(\mathcal{P})}$ is the slice of a presheaf category since slices of presheaf categories are themselves presheaf categories.
	\end{proof}
	
	\begin{theorem}
		Given a fixed globular PRO $\mathcal{P}$, the pullback of the monads corresponding to $R_{U(\mathcal{P})}, \overline{\mathcal{W}}_{U(\mathcal{P})},$ and $\overline{U}_{\mathcal{P}}$ is a monad over $\sfrac{\mathbb{N}\Coll\Graph}{U(\mathcal{P})}$.
	\end{theorem}
	\begin{proof}
		By Theorem \ref{ContFinAndMod} and corollary \ref{SliceFinAndMod} above, we know that $R_{U(\mathcal{P})}, \overline{\mathcal{W}}_{U(\mathcal{P})},$ and $\overline{U}_{\mathcal{P}}$ are finitary and monadic.  By corollary \ref{LocalFinPresen} we know that $\sfrac{\mathbb{N}\Coll\Graph}{U(\mathcal{P})}$ is locally finitely presentable.  By applying Kelly's theorem $\cite{kelly_UnifiedTreat1980}$, we can construct the monadic pullback functor we desire by forming a pullback cube, each face of which is a pullback square.
	\end{proof}
	
	\begin{definition}
	    For a globular PRO $\mathcal{P}$, we call the monad from the previous theorem the \emph{weakening monad for} $\mathcal{P}$.  We denote this pullback monad $$\mathfrak{G}_{\mathcal{P}}: \sfrac{\mathbb{N}\Coll\Graph}{U(\mathcal{P})} \rightarrow \sfrac{\mathbb{N}\Coll\Graph}{U(\mathcal{P})}$$ Its algebras are by definition $\mathbb{N}\Coll$-graphs with contraction over $U(\mathcal{P})$ that also have the structure of being a $\Coll$-category and a monoid in $\mathbb{N}\Coll\Graph$.
	\end{definition}
	
	Note that because these monads are constructed via a pullback, the algebras for this monad ARE NOT globular PROs.  In particular, there are no axioms/relations required of the algebras for $\mathfrak{G}_{\mathcal{P}}$ which make the monoidal product in each algebra functorial.  For this reason, we will give the algebras for these monads a special name.
	
	\begin{definition}
	    Given any globular PRO $\mathcal{P}$, we call an algebra for the pullback monad $\mathfrak{G}_{\mathcal{P}}$ a \emph{weakening} of the PRO $\mathcal{P}$.
	\end{definition}
	
	\begin{definition}
	    Given a globular PRO $\mathcal{P}$, let $\mathcal{E}$ be a weakening of the globular PRO $\mathcal{P}$.  An algebra for $\mathcal{E}$ is given by a degenerate collection $a:A \rightarrow \Tone$ together with, for all $n,m \in \mathbb{N}$, a family of collection homomorphisms $\Omega_{n,m}:\mathcal{E}(n,m)\mysquare A^n \rightarrow A^m$ which makes commute the analogous diagrams to that for an algebra for a globular PRO.
	\end{definition}
	
	We will apply $\mathfrak{G}_{\mathcal{P}}$ to the initial object in $\sfrac{\mathbb{N}\Coll\Graph}{U(\mathcal{P})}$ to construct an $\mathbb{N}\Coll$-graph with contraction over $U(\mathcal{P})$ that, when viewed as an algebra for $\mathfrak{G}_{\mathcal{P}}$, is the initial free weakening of $\mathcal{P}$.  As we will see in the next section, when $\mathcal{P}$ is the globularization of a classical PRO $P$, the algebras for the initial free weakening will be $\omega$-categorifications of the theory encoded by $P$. However, although this constructions can be performed with any globular PRO, it is less clear what the significance is of this process when applied to those PROs which are not weakenable.
	
	Note that the lack of functoriality for addition in a weakening of a globular PRO is a virtue of this construction.  Remember that our goal is to create weak $\omega$-categorifications of an algebraic theory which can be presented by an ordinary PRO $P$.  And these desired $\omega$-categorifications are the algebras for the initial weakening of the globularization of $P$.  Note that, for reasons analogous to Theorem \ref{PullbackAlg}, an algebra for any weakening is necessarily an algebra for the initial weakening.  And there the relations that would need to hold for the weakened PRO to have a monoidal product functor are now replaced with cells in the weakening that make + a monoidal functor only up to equivalence in the $\omega$-categorical sense.

	\section{Weak Invertibility}
	
	When we apply $\mathfrak{G}_{\mathcal{P}}$ to construct a weakening of an object in $\sfrac{\mathbb{N}\Coll\Graph}{U(\mathcal{P})}$, where $\mathcal{P}$ is the globularization of an ordinary PRO, we will look at the algebras for these weakenings.   These algebras are globular sets with the structure of a weak $\omega$-category.  This is precisely because, when $\mathcal{P}$ is the globularization of a classical PRO $P$, the new contraction cells added to our weakened globular PRO have the effect of transforming relations in the theory captured by $P$ expressed by equations in $P$ with appropriate equivalences in the algebras for the weakening of $\mathcal{P}$.  We shall now make this precise, in a manner similar to the presentation given by Cheng\cite{Cheng_2007DualsGroupoid}.
	
	\begin{definition}
		    Let $\mathcal{C}$ be an $\omega$-category with $\mathcal{C}(k)$ the set of all $k$-cells in $\mathcal{C}$.  A $k$-cell $f$ in $\mathcal{C}$ is \emph{pseudo-invertible} if it can be equipped with a set of `witnesses' $W=\coprod\limits_{j \geq 0}W(j)$ such that
		        \begin{enumerate}
		            \item $W(j) \subset \mathcal{C}(k+j)$
		            \item $f \in W(0)$
		            \item For all $g \in W(j)$ whose source and target $(k+j-1)$-cells are $X$ and $Y$ respectively, there exists a cell $g' \in W(j)$ whose source and target $(k+j-1)$-cells are $Y$ and $X$ respectively.  In other words, for any cell \newline
		            $\begin{tikzcd} 
                      	\cdot 
                    	\arrow[rr, bend left=40, "X", ""{name=A}]
                		\arrow[rr, bend right=40, "Y", swap, ""{name=B}]
                    	&		& 
                    	\cdot
                		\arrow[Rightarrow, "g", shorten <=2pt, shorten >=2pt, swap, from=A, to=B]
                      \end{tikzcd} \in W(j)$ there exists a cell
                    $\begin{tikzcd} 
	                	\cdot 
                		\arrow[rr, bend left=40, "Y", ""{name=A}]
                		\arrow[rr, bend right=40, "X", swap, ""{name=B}]
                		&		& 
                		\cdot
                		\arrow[Rightarrow, "g'", shorten <=2pt, shorten >=2pt, swap, from=A, to=B]
                    \end{tikzcd} \in W(j)$.
		            \item Let $\id_X$ denote the identity $k$-cell on a $(k-1)$-cell $X$. For all $g \in W(j)$ there exists two $(k+j+1)$-cells $i_g, e_g \in W(j+1)$ with source and target $(k+j-1)$-cells as depicted below:
		            $$\begin{tikzcd} 
                      	\cdot 
                    	\arrow[rr, bend left=40, "\id_X", ""{name=A}]
                		\arrow[rr, bend right=40, "g'g", swap, ""{name=B}]
                    	&		& 
                    	\cdot
                		\arrow[Rightarrow, "i_g", shorten <=2pt, shorten >=2pt, swap, from=A, to=B]
                      \end{tikzcd} \hspace{1in}
                    \begin{tikzcd} 
	                	\cdot 
                		\arrow[rr, bend left=40, "gg'", ""{name=A}]
                		\arrow[rr, bend right=40, "\id_Y", swap, ""{name=B}]
                		&		& 
                		\cdot
                		\arrow[Rightarrow, "e_g", shorten <=2pt, shorten >=2pt, swap, from=A, to=B]
                    \end{tikzcd}$$
                    in $W(j)$
		        \end{enumerate}
            The \emph{pseudo-inverse} of $f$ is the witness $k$-cell $f' \in W(0)$ required by the third condition on $W$ from above.
	\end{definition}
		
	The above definition can intuitively be thought of as stating that a $k$-cell $f$ is pseudo-invertible with pseudo-inverse $f'$ if there exist pseudo-invertible $(k+1)$-cells $i_f$ and $e_f$ as described above.  In fact, this intuitive definition can be made fully rigorous using corecursion \cite{Moss_97Corecursion} (the corecursion is necessitated by the fact that, as stated, this intuitive definition requires the existence of other pseudo-invertible cells of higher dimension, seemingly presupposing the original definition).
	
	\begin{definition}
	Two $k$-cells are \emph{weakly equivalent} if there exists a pseudo-invertible cell between them.
	\end{definition}
	
	We will need this notion for the following theorem.

	\begin{theorem}
	Let $\mathcal{P}$ be the globularization of a classical PRO $P$, whose initial weakening is $\mathfrak{P}$.  Let $\mathcal{A}$ be an algebra for $\mathfrak{P}$.  If a pair of cells in $\mathfrak{P}(m,n)$ which sit over the same cell in $U(\mathcal{P})$ act on the same globular word in $A^m$, the two resulting cells in $A^n$ must be weakly equivalent in $A^n$.  In particular, lifts of operations in the original theory that gave equal results now give weakly equivalent results.
	\end{theorem}
	
	\begin{proof}
	Note that the operations in $\mathfrak{P}$ that sit over the same $k$-cell in $U(\mathcal{P})$ are parallel cells in $\mathfrak{P}$ that map to the boundary of a $(k+1)$-cell in $U(\mathcal{P})$, namely the identity on the cell in $U(\mathcal{P})$ over which both operations sit.  By the contraction structure over $U(\mathcal{P})$, this ensures that there is a $(k+1)$-cell in $\mathfrak{P}$ between the two operations corresponding to the two $k$-cells in $\mathcal{A}$.  Moreover, this $(k+1)$-cell is weakly invertible since it is the lift of an identity cell in $U(\mathcal{P})$, as a similar lifting argument ensures the existence of necessary higher dimensional cells for this $(k+1)$-cell to be weakly invertible in $\mathcal{P}$.  The actions of this $(k+1)$-invertible cell and all of its witness on a globular word in $A^m$ produce the weak equivalence between the cells resulting from application of the two initially mentioned operations. 
	\end{proof}
	
	With this theorem, we now have the following desired result.  Let $\mathcal{P}$ be the globularization of a classical PRO $P$, whose initial weakening is $\mathfrak{P}$.  Let $\mathcal{A}$ be an algebra for $\mathfrak{P}$.  Any two cells in $\mathcal{A}$ corresponding to operation cells in $\mathfrak{P}$ constructed from cells in $U(\mathcal{P})$, connected via $+$ and pasting composition, whose lifts in $\mathfrak{P}$ sit over the same cell in $U(\mathcal{P})$, are weakly equivalent in $\mathcal{A}$.  This follows from the previous lemma together with the fact that all possible contractions cells for operations in $\mathfrak{P}$ connected via $+$ and pasting composition must exist in $\mathfrak{P}$ via its construction as a pullback.
	
\nocite{Lawvere_2004}
\nocite{leinster2000operads}
\bibliographystyle{abbrv}
\bibliography{OmegaCatBib}

\end{document}